\documentclass[12pt]{article}
\usepackage{amsthm}
\usepackage{amsmath}
\usepackage{amssymb}
\usepackage{graphicx}
\usepackage{setspace}
\theoremstyle{plain}

\newtheorem{mainthm}{\protect\theoremname}

\newtheorem{thm}{\protect\theoremname}[section]
\theoremstyle{plain}
  \newtheorem{lem}[thm]{\protect\Lemmaname}
    \theoremstyle{definition}
  \newtheorem{rem*}{\protect\remarkname}
  \newtheorem{defn}[thm]{\protect\definitionname}
  \theoremstyle{plain}
  
  \theoremstyle{plain}
  
  \theoremstyle{remark}
  
  \theoremstyle{definition}
  \newtheorem*{example*}{\protect\examplename}

\usepackage{changepage}
\usepackage{caption}
\makeatother
  \providecommand{\claimname}{Claim}
  \providecommand{\corollaryname}{Corollary}
  \providecommand{\definitionname}{Definition}
  \providecommand{\examplename}{Example}
  \providecommand{\Lemmaname}{Lemma}
  \providecommand{\propositionname}{Proposition}
  \providecommand{\remarkname}{Remark}
\providecommand{\theoremname}{Theorem}
\usepackage{color}

\usepackage{answers}
\usepackage{setspace}
\usepackage{graphicx}
\usepackage{float}
\usepackage{enumitem}
\usepackage{multicol}
\usepackage[margin=1in]{geometry}
\usepackage{tikz}
\usetikzlibrary{patterns}
\usepackage{subfig}

\begin{document}

\begin{title}
{Stable motions of high energy particles interacting via a repelling potential}
\end{title}

\author{V. Rom-Kedar\thanks{The Estrin Family Chair of Computer Science and Applied Mathematics, Department of Computer Science and Applied Mathematics,  The Weizmann Institute of Science
  (email:{vered.rom-kedar@weizmann.ac.il}).} \
 and D. Turaev\thanks{Department of Mathematics, Imperial College  (email:{d.turaev@imperial.ac.uk })}.}

\maketitle

\begin{abstract}
The motion of \(N\) particles interacting by a smooth
repelling potential and confined to a compact
\(d\)-dimensional region is proved to be, under
mild conditions, non-ergodic for all sufficiently
large energies. Specifically, choreographic
solutions, for which all particles follow
approximately the same path close to an
elliptic periodic orbit of the single-particle
system, are proved to be KAM stable in the
high energy limit. Finally, it is proved that
the motion of \(N\) repelling particles in
a rectangular box is non-ergodic at high energies
for a generic choice of interacting potential:
there exists a KAM-stable periodic motion by which
the particles move fast only in one direction,
each on its own path, yet in synchrony with all
the other parallel moving particles.\end{abstract}
%Keywords: Choreographic motion, Repelling particles, KAM theory, Billiards,Ergodicity
\tableofcontents
\section{Introduction}

Can a large number \(N\) of repelling particles moving rapidly in a \(d\)
(\(\geqslant 2\))-dimensional domain \(D\), remain forever bounded away from each other? We prove that such stable motion that avoids collisions occurs with positive probability.
Borrowing the terminology from  Celestial Mechanics \cite{Montgomery:2010,Chenciner2000,Chenciner2002,Fukuda2017}, the solutions we construct are of a choreographic type, i.e., the particles move essentially synchronously along the same path (or, along a family of parallel paths) with nearly constant phase shifts between them. It follows that systems of repelling particles are not ergodic, and have, in fact,  KAM-stable states. Moreover, we show that such stable sets persist at
arbitrarily high energies, for any finite number of particles.

Establishing ergodicity of the Liouville measure (the Lebesgue measure restricted to a constant level of the Hamiltonian in the phase space) is a long-standing problem for conservative many-particle systems. The question is related to principal issues of the foundation of statistical mechanics, see e.g. \cite{Balint2021,Kozlov2000} and also \cite{bricmont2022making} in which  this relation is critically discussed. Classical statistical mechanics is based on the assumption, sometimes called {\em Gibbs postulate}, that macroscopic quantities describing the state of a large
system of microscopic particles are averages over the Liouville measure in the phase space (the so-called micro-canonical ensemble). This postulate is supported by an overwhelming experimental evidence; the question is whether it can be inferred
from the Hamiltonian formulation of dynamics by rigorous mathematical arguments. Are there general properties of the Hamiltonian dynamics which make a general Hamiltonian system choose the Liouvile measure over all other invariant measures?

The ergodicty of the Liouville measure could be such a property\footnote{One may argue that macroscopic quantities are, in fact,
time-averages, so they are indeed equal to the averages over the Liouville measure for a full-measure set of initial conditions when the Liouville measure is ergodic, by Birkhoff-Khinchin theorem.  Notably, other mechanisms that do not require ergodicity have been suggested, see e.g. the  recent discussion, examples and a historical review in \cite{bricmont2022making}.}. However, by Kolmogorov-Arnold-Moser (KAM) theorem, the ergodicity is violated for an open set of smooth Hamiltonians - for example, it is violated for energy levels near any non-degenerate minimum or maximum of the Hamiltonian function. Therefore, one cannot simply postulate ergodicity - it has to be justified by certain additional properties of the class of systems under consideration. Below we summarize some of the relevant works on the \(N\)-particle problem which focus on proving ergodicity within the Sinai program of studying the (billiard) dynamics of the gas of hard balls, or, on the contrary, proving non-ergodicity by studying the emergence of stability islands. As we mentioned, at low energy, near local minima of the potential (i.e., near ``ground states'') one expects, by KAM theory, that the system is generically non-ergodic. Therefore, a sensible mathematical question is to study ergodic properties of many-particle systems at high energies.\\

\textbf{Hard spheres in a container:} The idea going back to Boltzmann is that one can neglect the interactions between particles when the potential energy of the interaction is much smaller than their kinetic energy.
This means that in the gas of sufficiently energetic particles, the particles motion is essentially free except for the short instances when the distance between some particles becomes small enough to create a strong repulsion force resulting in the fast change of the momenta. In the limit, one obtains the {\em Boltzmann gas} of \(N\)-hard spheres of diameter \(\rho\), which interact only via momentarily elastic collisions and are confined to a $d$-dimensional container of the characteristic size $L$ such that \(N\rho^{d} \ll\ L^d \). This provides  a {\em universal model} for any system of $N$ particles in such a container for large values of the kinetic energy per particle, irrespective of the precise form of the repelling interaction potential.

Thus, proving the ergodicity of the Boltzmann gas -- the Boltzmann-Sinai ergodic conjecture -- is one of the  corner-stone problems in the mathematical foundations of statistical mechanics. The Sinai program \cite{sinai1963foundations,sinai1970dynamical,Sinai1997}
was inspired by ideas of Krylov \cite{Krylov1979} and culminated in a series of works \cite{Simanyi2002,Simanyi2003,Simanyi2004,Bunimovich1991,Simanyi1999}. By this program, the ergodicity of the Boltzmann gas is inferred from the characteristic ``Krylov-Sinai'' instability of the elastic collision of spheres (or any convex bodies) in $\mathbb{R}^d$ for $d\geqslant 2$: a small change in the momentum of the particle increases exponentially with the number of collisions. One can view the \(N\)-particle hard-sphere gas in \(d\) dimensions as a billiard in an \(Nd\)-dimensional domain \cite{Sinai1970}. The pair-wise collisions of the
spheres correspond to boundaries of the domain -- Krylov-Sinai instability means that these boundaries are (semi)-dispersing, which, for hard spheres moving on a flat torus or in a rectangular box, implies the hyperbolicity of the dynamics \cite{Bunimovich1991,Simanyi1999,Simanyi2002} and leads to the ergodicity of the Liouville measure \cite{Simanyi2003,Simanyi2004}.

The Sinai program has led to seminal works in dynamical systems theory --
it was one of the main sources for the development of ergodic theory of smooth dynamical systems, the theory of billiards and of general
dynamical systems with singularities \cite{Katok1995,Barreira2006,chernov2006chaotic,KozlovBook}.
However, it has also revealed the  inherent difficulties in relating the Boltzmann gas dynamics to the problem of the
ergodicity of multi-particle systems.

A well-recognized difficulty is the strong dependence of the hard-sphere dynamics on the container shape, see
\cite{Bunimovich2003,Lansel2006,Kozlov2016}. When the container boundary has a convex piece, the $Nd$-dimensional
billiard representing the hard-spheres gas acquires a non-dispersing (focusing) boundary component, which makes the establishment of the hyperbolicity problematic.
Notably, even in the case of a concave container, the ergodicity of the Boltzmann gas has been established only in a quite special geometrical set-up, for spheres of a sufficiently large diameter \(\rho\) \cite{Bunimovich1992}.

The grander problem is the singularity of the hard-sphere system: the interaction potential jumps from zero to infinity when the
distance between particles becomes equal to $\rho$. The Boltzmann gas serves as a universal limit of {\em smooth} multi-particle systems. Since this limit is singular, the question of which of its dynamical and statistical properties survive a regularization must be addressed.\\

\textbf{Smooth billiard-like Hamiltonians} provide a natural regularization of billiard dynamics. Such a Hamiltonian $H$ is the sum of
the kinetic energy term (a positive-definite quadratic function of momenta) and a steep potential \(V(q;\delta)\) associated with a billiard domain \(D\subset  \mathbb{R}^d \). The potential is a smooth function of $q\in D$ and it also depends on a small parameter  \(\delta\) (the inverse steepness), so that when
\(\delta\rightarrow0\), the potential vanishes in the interior of \(D\) while staying bounded from below on the billiard boundary \(\partial D\).

For example, in the present paper, we consider $N$-particle systems with a smooth interaction potential $W$ which tends to
$+\infty$ when the distance between the particles approaches $\rho$; the particles are confined to an open bounded
$d$-dimensional region by a smooth potential $V$ which gets infinite on the boundary of this region.
When we restrict the system to the energy surface \(H=\frac{1}{2} N h\) and scale the momenta by \(\sqrt{h}\) (so the energy is scaled by $h$), we obtain a billiard-like system with
the steep potential \(\delta (V+W)\), where \(\delta=1/h\); see the precise setup in Section \ref{sec:substrongrepelling}. The limit
\(\delta\rightarrow0\) for the fixed value of the rescaled energy $H=\frac{1}{2}N$ corresponds to the high energy limit of the unscaled system.

In \cite{RapRKT07apprx,TRK03cor,rom2012billiards,turaev1998elliptic}, we described a large class of billiard-like Hamiltonians with  steep potentials that satisfy some natural growth and smoothness conditions. We proved for this class that the limit billiard dynamics which are represented by {\em regular} orbits -- i.e., those which hit the billiard boundary
$\partial D$ away of its singularities and at angles away from zero -- persist for sufficiently small $\delta$.
Namely, near the regular orbits, the local return maps of the smooth Hamiltonian flows to cross-sections that are bounded away from \(\partial D\) tend
{\em with all derivatives} to those of the billiard as \(\delta\to 0\)   \cite{RapRKT07apprx,rom2012billiards,turaev1998elliptic}. This implies that regular uniformly-hyperbolic sets and KAM-nondegenerate elliptic orbits of the billiard persist for sufficiently small $\delta$ in the smooth billiard-like system \cite{RapRKT07apprx,rom2012billiards}.

On the other hand, we also showed that the regularization of dispersing billiards changes drastically their dynamics near
{\em singular} orbits, such as orbits which are tangent to $\partial D$ or which enter corner points in
$\partial D$. Namely, the inherent hyperbolic structure of dispersing billiards {\em cannot survive the regularization} \cite{turaev1998elliptic,rom2012billiards}. In particular, singular periodic orbits of dispersing billiards give rise to stable periodic motions --
hence to non-ergodic behavior -- in the smooth system at arbitrarily small $\delta$. Indeed, we proved, under quite general conditions, the loss of ergodicity due to the regularization for two-dimensional dispersing billiards  \cite{TRK03cor,turaev1998elliptic} and also for billiards with specific types of corners in any dimension \cite{RapRK063d,RapRKT08stab}. Applying this logic to the billiard that represents the Boltzmann gas, one concludes that
the same dispersing geometry that creates the Krylov-Sinai instability of the colliding spheres is also responsible for the
destruction of the associated hyperbolic structure -- when the hard-spheres model is replaced by a more realistic model of particles interacting via a smooth potential. It is thus natural to conjecture that orbits of the system of \(N\) hard spheres which
undergo sufficiently many instances of brushing (zero angle) collisions between the spheres or end at multi-collision points (simultaneous collisions of more than 2 particles) can produce islands of stability of the corresponding system of \(N \)
smoothly interacting particles at sufficiently high energy (see, e.g., discussion in \cite{RapRKT08stab}).  Actually, in earlier works, Donnay established the non-ergodicity of two planar point particles   interacting via a smooth short-range potential by utilizing this mechanism of folding in phase space that occurs near brushing orbits  \cite{Donnay1996,donnay1999non}. Proving the non-ergodicity of the \(N\)-particle problem for \(N>2\) by exploring these inherit  singularities of the limit system near such collisions has not been realized yet.

In this paper, we explore a new and different mechanism of the ergodicity loss of the hard-spheres system due to the smoothing. We establish the existence and stability of choreographic solutions for which highly-energetic particles, placed on the same periodic path or parallel paths, never come close to collisions. An  example of choreographic motion of repelling particles is given by  the so-called  polygonal Coulomb crystals: a small number of equidistant Coulomb  particles moving in an harmonic (Paul) trap   \cite{amiranashvili1999stability}. One can find such motions in the hard-spheres system as well (just take the diameter
$\rho$ small enough to ensure the spheres on the same path do not overlap and let them move with the same speed). However,
they are unstable, as small discrepancies in the speed eventually lead to collisions of the spheres. As we show, if the speed of the synchronous motion of the particles is
sufficiently high, a generic smoothing of the repulsive interaction potential stabilizes such type of solutions for some discrete set of  particles phases.\\

\textbf{The $N$-body problem} of Celestial Mechanics has much in common with the \(N\)-particle problem discussed here, with the difference that the $N$-body problem usually refers to attracting interactions. In both cases, the pairwise interaction decays at large distances, whereas, at very small distances, the interaction potential is singular (see e.g. \cite{Fejoz2021}). The full characterization of the mixed phase space dynamics for \(N\geqslant 3\) is intractable due to non-integrability \cite{Knauf2002,meyer2008introduction,turaev2010richness,Guardia2022}.
Thus, finding special type of solutions, in particular KAM-stable periodic motions, is an achievement for such problems  \cite{Arnold2007CelestialMechanics,chierchia2010properly,chierchia2011planetary,neishtadt2021stability,simo2015dynamical}. Choreographic solutions
for the \(N\)-body problem were found by fixing the phase shift between the bodies to be constant and utilizing symmetries to establish that such solutions minimize the action \cite{yu2017simple,Chenciner2000,Chenciner2002,Montgomery2015,Moore1993,fusco2011platonic}. The avoidance of collisions for the attracting potential case follows from the observation that the action becomes infinite for sufficiently strong singularities (in particular, the Newtonian potential is not included). Numerically, the choreographic solutions with small $N$ were found by continuation schemes also for the Newtonian potential  \cite{Chenciner2002,Ouyang2015} and for the Lennard-Jones potential  \cite{Fukuda2017}. In these works, the choreographic solution is not induced by an external field, it is a genuine outcome of the particles interaction.
Here, we  study the problem of choreographic solutions that follow a path dictated by a common background potential that governs the uncoupled dynamics.  The method we use for proving the existence and stability of choreographic solutions is based on averaging. While we mostly focus on repelling potentials,
the case of attracting potentials is covered by our scheme too, see Remarks \ref{rem:maxpotential} and \ref{rem:attrpotential} in  Section \ref{sec:setup}. These results may be related to KAM stability of resonant solutions of planetary systems, see \cite{chierchia2011planetary} and references therein.\\

The paper is ordered as follows. In Section \ref{sec:setup} we list our main results regarding the existence and stability of choreography-type solutions in a system of \(N\) interacting particles in a common background potential.  We consider  four different settings. Theorem  \ref{thm:mainbounded} states that such KAM-stable motion exists in the case of \(N\) identical, weakly interacting particles when the particles are subject to  a smooth background potential which admits a non-degenerate elliptic periodic orbit.  Theorem \ref{thm2} states that under some additional conditions the same result applies when the interaction potential is repelling - i.e. it diverges to $+\infty$ when the particles get close to each other. These results correspond to common  setups in which the application of KAM theory is natural. The method of their proofs prepares the ground to the more demanding parts in which the singular billiard limit is considered - Theorems 3-5. Theorem \ref{thm:billiardsystem} states that the same result applies in the high energy limit for interacting particles with a repelling interaction potential when the background potential is billiard-like and the limiting billiard admits a KAM non-degenerate   periodic orbit. Finally, Theorems \ref{thm:billiardbox1} and \ref{thm:billiardbox2} state that under some explicit non-degeneracy conditions there exist  KAM-stable periodic motions of \(N\) repelling particles in a rectangular box. Sections \ref{sec:proofmainthm} -\ref{sec:boxsection} contain the proofs of these theorems. Following the discussion section, the appendix establishes, by applying the results of \cite{rink2001symmetry},  the existence of KAM-tori in  Fermi-Pasta-Ulam type chains, which naturally arise in the averaging of identical particles systems.

\section{\label{sec:setup}Setup and Main results}

By a particle, we mean a $d$ degrees of freedom, autonomous Hamiltonian system with
a Hamiltonian function \(H_{0}(q,p)\), i.e.,
\begin{equation}\label{eq:singleparticlesys}
\dot q=\partial_{p}H_0(q,p),\qquad\dot p=-\partial_{q}H_0(q,p), \qquad(q,p)\in\mathbb{R}^{2d}.
\end{equation}
Let this system have, at a certain energy value $H_0=E^*$, an elliptic periodic orbit
\(L^* \) with period $T=\frac{2\pi }{\omega_0}$. Let the equation of  \(L^*\)  be
$(q,p)=(q^*(\omega_0t),p^*(\omega_0t))$ where \((q^*,p^*)\) are some \(2\pi\)-periodic functions.

Consider the system of  \(N \) identical particles, each controlled by the same Hamiltonian $H_0$,
and allow the particles to interact with each other.
We define the system of interacting particles by the Hamiltonian
\begin{equation}\label{eq:multparsysy}
H=\sum_{n=1,\dots,N }H_0(q^{(n)},p^{(n)})+\delta\mathop{\sum_{n,m=1,\dots,N}}_{n\neq m}\ W(q^{(n)}-q^{(m)})
\end{equation}
where \((q^{(n)},p^{(n)})\) are the coordinates and momenta of the \(n\)-th particle,  $\delta\geqslant 0$ is a small coupling parameter, and $W$ is the interaction potential. We assume that the particles are identical,
so the interaction potential is the same for any pair of particles; similar results hold true also
when the pairwise interaction potentials vary from pair to pair. Note that in (\ref{eq:multparsysy}) we sum over each pair twice and take the  interaction to depend only on the difference between the particles coordinates. With this choice, the potential is  translation invariant and even. Such properties are natural from a physical point of view, yet, mathematically,  they make the question of  genericity more delicate. The proofs of Theorems   \ref{thm:mainbounded} -  \ref{thm:billiardbox1} do not use these properties (rather, overcome them) so these theorems remain valid for arbitrary pair-wise interaction potential.

For \(\delta=0\), Hamiltonian (\ref{eq:multparsysy}) describes the motion of $N$ non-interacting
particles. It has ``choreography'' type solutions, for which each particle moves along the same periodic path \(L^*\) with a given phase shift:
\begin{equation}\label{uncoupledch}
\mathbf{L}^{*}(\theta)=\{q^{(n)}=q^{*}(\omega _{0}t+\theta^{(n)}), \;\; p^{(n)}=p^*(\omega _{0}t+\theta^{(n)}),
\qquad n=1,\dots,N\}
\end{equation}
for an arbitrary  set of fixed phases $\theta=(\theta^{(1)},\dots,\theta^{(N)})\in \mathbb{T}^N$.
Below, we formulate conditions which ensure that for sufficiently small \(\delta>0\)  choreographic motions persist for all time:  the particles, modulo small oscillations, perpetually orbit \(L^*\) with the same frequency and with certain individual phase shifts $\theta$.

The ``equilibrium phases'' $\theta$ are found as minima of the interaction potential averaged over the synchronous collective motion
of the particles along $L^*$. We first perform the averaging for the case of uniformly bounded, smooth (\(C^{\infty}\)) potential $W$;
after that we generalize the results to the case of repelling potentials, i.e., those which tend to
$+\infty$ as $q^{(n)}-q^{(m)}\to 0$.  Then we consider high-energy particles in a container of a generic shape. This corresponds to the single-particle Hamiltonian $H_0$ depending on $\delta$ in a singular way -- the limit motion is a billiard in the domain where the particles are confined. The singularity in $H_0$ requires amendments to the
averaging procedure and, also, additional conditions on the interaction potential for the persistence of choreographic motions along an elliptic periodic orbit of the billiard. Finally, we consider the special case of interacting high-energy particles in a rectangular box. The limit billiard does not have elliptic orbits in this case, however we show that
a generic repelling interaction stabilizes choreographic motions along parabolic periodic orbits in the box.

\subsection{\label{sec:setupsinglepart}Local assumptions on the single-particle system}

First, we impose non-degeneracy conditions on the periodic orbit $L^*$ of
the one-particle system (\ref{eq:singleparticlesys}). We call these conditions single-particle (SP) assumptions. Recall that elliptic
orbits exist in families parameterized by energy $E=H_0(q,p)$. Thus, by the assumption that the one-particle system
(\ref{eq:singleparticlesys}) has at energy \(E^*\) an elliptic periodic orbit \(L^*\), it follows that it has a smooth family $(q,p)=(q(t,E),p(t,E))$ of elliptic
periodic orbits $L(E)$ such that the energy value $E=E^*$ corresponds to the
original periodic orbit $L^*=L(E^{*})$.\\

\noindent{\bf SP1: Acceleration assumption.} {\em Near \(E^*\), the period of  the elliptic periodic orbit $L(E)$ decreases
with energy.}\\

Note that this assumption holds for billiard-like potentials \cite{RapRKT07apprx,rom2012billiards,turaev1998elliptic}, geodesic flows, and other settings where higher energy corresponds to a higher speed of the motion along the same (or almost the same) path in the configuration space.\\

Let us introduce symplectic coordinates
$(I_{0},\theta,z)$ where $I_0\in\mathbb{R}$,
$\theta \in \mathbb{S}^1=\mathbb{R}^{1}/{2\pi\mathbb{Z}}$, $z\in\mathbb{R}^{2(d-1)}$,
such that the surface filled by the periodic orbits $L(E)$ is given by \(z=0\).\footnote{This is done as follows: one first straightens this surface by a smooth coordinate transformation which is not necessarily symplectic; this transformation may change the symplectic form  -- then one brings
the symplectic form back to the standard form by a smooth transformation which leaves
the surface $z=0$ invariant.} Moreover, we choose $(I_0,\theta)$ such that they give action-angle variables for system (\ref{eq:singleparticlesys}) restricted to the surface $z=0$. This means that on any of the periodic orbits $L(E)$ that foliate this surface the value of $I_{0}$ stays constant and equal to the signed area between $L(E)$ and $L(E^*)$ (and the variable $\theta$ is symplectic conjugate to $I_0$). Thus, the Hamiltonian restricted to this surface is
a function of \(I_{0}\) only and the frequency of the periodic orbit $L(E^{*})$ is equal to
$$\omega_0=\partial_{I_0} H_0(I_0,\theta,0)|_{H_0=E^*}.$$
Thus, Assumption SP1 reads as
\begin{equation}\label{mb0}
a=\partial_{I_0I_0} H_0(I_0,\theta,0)|_{H_0=E^*}>0.
\end{equation}
\begin{figure}
\begin{centering}
\includegraphics[scale=0.4]{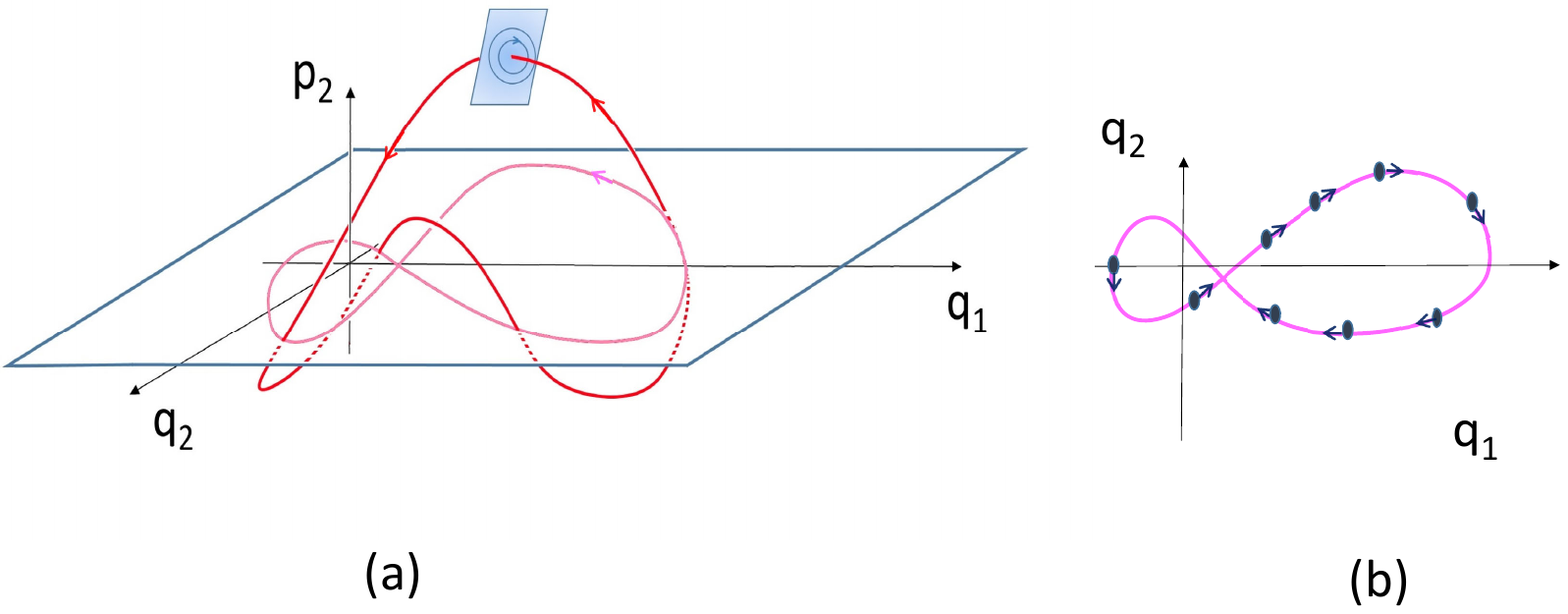}
\par\end{centering}
\protect\caption{\label{fig:ellipticorbit} (a) An elliptic periodic orbit of a 2 degrees of freedom system on the fixed energy level \(H_0=E^{*}\). The elliptic periodic orbit \(L^*\) (red curve)  is surrounded by KAM tori - the blue curves correspond to the intersection of these tori with a transverse  cross-section.
The pink curve shows the projection of $L^*$ to the two-dimensional configuration space - here the projected curve has a single self-intersection point.
(b) Theorem \ref{thm:mainbounded} shows that for a generic smooth interaction potential, any number of weakly interacting particles can orbit the same path in the configuration space. Theorem \ref{thm2} shows that the multi-particle choreography is KAM-stable also  for repelling potentials.}
\end{figure}
Let the multipliers\footnote{Recall that the multipliers of the periodic orbit $L^*$ are defined as follows. Consider a restriction of the system to the $(2d-1)$-dimensional energy level $H_0=E^*$, take a $(2d-2)$-dimensional cross-section to
$L^*$ in this energy level, and consider the Poincar\'e map (the map defined by the
orbits of the system) on the cross-section. The intersection point of
the periodic orbit with the cross-section is a fixed point of the map. The eigenvalues of
the linearization matrix of the Poincar\'e map at this point are called the multipliers of $L^*$.
Since $L^*$ is elliptic, all its multipliers are not
real and lie on the unit circle.} of \(L^* \) be
$e^{\pm i\frac{2\pi }{\omega_{0}}\omega_1},\dots,
e^{\pm  i\frac{2\pi }{\omega_{0}}\omega_{d-1}}$.

\noindent{\bf SP2: Non-resonance assumption.} {\em The frequencies
$\omega=(\omega_0,\dots,\omega_{d-1})$ are not in a
strong resonance:}
\begin{equation}
m_0 \omega_0 + \sum_{j=1}^{d-1} m_j \omega_j \neq 0
\label{eq:singlepartnonres}\end{equation}
{\em for every integer $m_0$ and every integer vector $(m_1,\dots,m_{d-1})$ such that
 $\displaystyle 1\leqslant \sum_{j=1}^{d-1} |m_j|\leqslant 4$.}\\

By this assumption, system (\ref{eq:singleparticlesys}) can be brought to
Birkhoff normal form up to fourth order \cite{Arnold2007CelestialMechanics}. Namely, in a sufficiently small neighborhood of
\(L^*\) one can perform a symplectic coordinate transformation
$$(q,p)=(\hat q(I_{0},\theta,z),\hat p(I_{0},\theta,z))$$
where $(I_0,\theta,z)\in \mathbb{R} \times \mathbb{S}^1 \times \mathbb{R}^{2(d-1)}$,
such that the Hamiltonian $H_0$ takes the form
\begin{equation} \label{eq:birkhoffsingle}
H_0(I_{0},\theta,z) = E^*+\omega I+\frac{1}{2}I^\top AI+g(I_{0},\theta,z),
\end{equation}
where $g(I_{0},\theta,z)=g_0(I_0)+g_1(I_0)\hat I+O(\|z\|^4 |I_{0}|+\|z\|^{5})$
with $g_0=O(I_0^3)$, $g_1 = O(I_0^2)$. Here we use the notation $z=(z_1, \dots, z_{d-1})$ where $z_j=(x_j,y_j)\in R^2$,
$I=(I_0, \hat I)$, where $\hat I=(I_1, \dots, I_{d-1})$ and \(I_j=\frac{1}{2}z_j^2\) denote the actions in the directions transverse to the orbit $L^*$, $j=1,\ldots,d-1$.
We think of $I$ as a column vector, the frequency vector,
$\omega$, is the row vector with components $\omega_0, \omega_1, \dots, \omega_{d-1}$, and $A$
is a symmetric $d\times d$ matrix with constant coefficients. We denote
\begin{equation}\label{amatr}
A=\begin{pmatrix}a &  b \\  b^\top &   \hat A \\ \end{pmatrix},
\end{equation}
where $a$ is a scalar (it is given by (\ref{mb0}) and is strictly positive by Assumption SP1),
$b\in R^{d-1}$ is a row vector, and $\hat A$ is a symmetric $(d-1)\times (d-1)$ matrix
with elements $\{a_{kj}\}_{k,j=1,\ldots,d-1}$.

The system of differential equations defined by the Hamiltonian (\ref{eq:birkhoffsingle}) has the form
$$\begin{array}{l}\dot I_0 = O(\|I\|^{5/2}),\qquad
\dot \theta = \omega_0 + a I_0 + \sum_{j=1}^{d-1} b_j I_j + O(I^2), \vspace{1mm}\\
\vspace{1mm}
\dot x_j = \;(\omega_j + b_j I_0 + \sum_{k=1}^{d-1} a_{kj} I_j) \;y_j + O(I^2),\\
\dot y_j = - (\omega_j + b_j I_0 + \sum_{k=1}^{d-1} a_{kj} I_j) \;x_j + O(I^2),
\qquad j=1, \dots, d-1.\end{array}$$
To the main order, the motion is a nonlinear rotation -- the rotation of the phase $\theta$ corresponds to the motion along the orbit $L^*$ (the circle $I=0$ in these coordinates), and the rotation of $(x_j,y_j)$ describes the transverse oscillations. The frequencies of the oscillations depend on the actions $I_0, \dots, I_{d-1}$, and  we
assume that this dependence is non-degenerate:\\

\noindent{\bf SP3: Twist assumption.} \textit{The orbit $L^*$ satisfies the twist condition and the iso-energetic twist condition}:
\begin{equation}\label{twistel}
\det A=\det\begin{pmatrix}a &  b \\
  b^\top &   \hat A \\
\end{pmatrix} \neq 0
\end{equation}
{\em and}
\begin{equation}\label{twistelw}
\det A_\omega=det \left(\begin{array}{ccc} 0 & \omega_0 & \hat\omega \\  \omega_0 & a &  b \\ \hat\omega^\top &  b^\top & \hat A \end{array}\right)\neq 0,
\end{equation}
{\em where } $\hat\omega=(\omega_1,\dots, \omega_{d-1})$.\\

This completes the list of assumptions on the single-particle motion, assuring that \(L^{*}\) is surrounded by KAM\ tori.

\subsection{\label{sec:condonw}Conditions on the coupling potential}
Next, we impose a non-degeneracy condition on the coupling potential in the multi-particle system (\ref{eq:multparsysy}). We call such conditions Interacting Particles (IP) assumptions. Consider
the motion of $N$ uncoupled particles over the same periodic orbit $L^*$, as given by (\ref{uncoupledch}),
and introduce the {\em averaged interaction potential}
\begin{equation}\label{eq:averagedpotu}
U(\theta^{(1)}, \dots,\theta^{(N)})=\mathop{\sum_{n,m=1,\dots,N}}_{n\neq m}
 W_{avg}(\theta^{(n)}-\theta^{(m)}),
\end{equation}
where
\begin{equation}\label{eq:averagedw}
 W_{avg}(\theta^{(n)}-\theta^{(m)})=
\frac{\omega _{0}}{2\pi}\int_0^{\frac{2\pi}{\omega _{0}}} W(q^*(\omega _{0}t+\theta^{(n)})-q^*(\omega _{0}t+\theta^{(m)}))dt.
\end{equation}
Note that \(U\) is  invariant under translations:
$U(\theta^{(1)}, \dots, \theta^{(N)})=U(\theta^{(1)}+c, \dots, \theta^{(N)}+c)$ for any
constant $c$. As \( U\) is a continuous function on the torus $\mathbb{T}^N$, it must have a point of minimum. By the translation invariance,
the minima of $U$ form lines in $\mathbb{T}^N$:
\begin{equation}\label{lineq}
\theta^{(n)}=\theta^{(n)}_{min}+c, \qquad n=1, \dots, N.
\end{equation}
We take any such line; by varying $c$, we can always make
$\theta^{(1)}_{min} + \ldots + \theta^{(N)}_{min} =0$ in (\ref{lineq}). Introduce coordinates $(\varphi, \psi_1, \dots, \psi_{N-1})$ in a small neighborhood of this line on $\mathbb{T}^N$ such that
$\varphi=\frac{1}{N}(\theta^{(1)}+ \dots +\theta^{(N)})$
and the transformation $(\theta^{(1)}-\theta^{(1)}_{min},\dots, \theta^{(N)}-\theta^{(N)}_{min}) \mapsto (\varphi, \psi_1, \dots, \psi_{N-1})$
is linear and, up to the factor $\frac{1}{\sqrt{N}}$, orthogonal.
In the new coordinates, the line of minima is the line $\psi=0$ and the averaged potential $U$ is independent of $\varphi$, so we denote
\begin{equation}\label{uhat}
\hat U(\psi)=U(\theta^{(1)}(\varphi,\psi), \dots ,\theta^{(N)}(\varphi,\psi)).
\end{equation}
We show in Section \ref{sec:proofmainthm} that for small $\delta$, the evolution of phases of interacting particles moving along the path $L^*$ is, to the main order, governed by the potential $U$. We are looking for stable motions, therefore, we make the following assumption on the interaction potential.\\

\noindent{\bf IP1: KAM assumption.} \textit{The local minimum of  \(\hat U(\psi)\) at the origin is non-degenerate; namely, the eigenvalues of the Hessian matrix
$\left(\partial_{\psi_j\psi_k} \hat U\right)_{j,k=1,\dots,N-1}$ are strictly positive, so the origin is an elliptic equilibrium point for the system defined by the Hamiltonian}
\begin{equation}
H =\frac{a}{2N}J^2+\hat U(\psi),\quad J\in \mathbb{R}^{N-1}.\label{eq:hamav}
\end{equation} \textit{Moreover,  in any neighborhood of the equilibrium this system has a KAM-torus:  a smooth invariant Lagrangian torus with a Diophantine frequency vector such that in a small neighborhood of this torus the system is near integrable and satisfies the twist condition. }\\

In view of the symmetries of \(\hat U(\psi)\) we explain this assumption, which is more general than the usual KAM non-degeneracy assumption, in more details. The  eigenvalues of the Hessian matrix
$\left(\partial_{\psi_j\psi_k} \hat U\right)_{j,k=1,\dots,N-1}$ are  the squares of
the frequencies of the small oscillations near the equilibrium configuration of the phases. The existence of positive measure set of KAM-tori  in a small neighborhood of the equilibrium is, for example,
achieved when
 there are no resonances up to
order four between the frequencies of small
oscillations, and the corresponding Birkhoff normal
form satisfies the twist condition at the equilibrium point.
This is a natural requirement for a generic potential $\hat U$, and assumption IP1 is satisfied in this case.

Since we consider the system of identical particles, the averaged potential
$U=\sum_{m\neq n}  W_{avg}(\theta^{(n)}-\theta^{(m)})$ is symmetric with respect to any permutation of the phases $\theta$ (in particular, since the sum is made over all pairs of different phases, we can always think of $W_{avg}$ as an even function). If the line  (\ref{lineq}) of minimum of $U$ is not preserved by the symmetry, then, the  potential  $\hat U$ is expected to be generic, so, as explained above, for verifying IP1, it is sufficient to check the absence of small resonances and twist condition at the origin for the system
(\ref{eq:hamav}).

On the other hand, when the line   (\ref{lineq}) is
symmetric with respect to some permutation of phases, the potential $\hat U$ inherits this symmetry,
which may lead to resonances. In this case, the question of the existence of KAM-tori in system
(\ref{eq:hamav}) cannot be reduced to the standard genericity assumptions and
has to be specially addressed by constructing a proper resonant normal form and verifying that the normal form has KAM tori as in IP1. Notably, the normal form might be non-integrable, in which case finding KAM tori  is not trivial, yet is doable as shown next.

A natural  example\footnote{A simpler example is the symmetric configuration $\theta^{(1)}=\ldots = \theta^{(N)}$
(all the particles are at the same point), yet it is not very relevant in our setting -- we are interested in the case of repelling potentials, so having all particles glued together for all time should not give a minimum
of the averaged potential. } is given by the equidistant distribution of phases: $\theta^{(n)}= 2\pi \frac{n}{N}$.
It is easy to see that the gradient of $U$ vanishes for this choice of $\theta$'s and, moreover, $U$ indeed
has a minimum at such configuration if, for example, the second derivative of $W_{avg}$ is positive
at the points $2\pi\frac{j}{N}$, $j=1, \dots, N-1$ (which is natural for repelling forces).  The corresponding line of minima (\ref{lineq}) is symmetric with respect to the cyclic permutation
$(\theta^{(1)},\dots,\theta^{(N)})\mapsto (\theta^{(2)},\dots,\theta^{(N)},\theta^{(1)})$; this is
responsible for the unavoidable creation of strong resonances between the frequencies of small oscillations
in system (\ref{eq:hamav}), similarly to the Fermi-Pasta-Ulam (FPU) chain, see Appendix \ref{sec:equidistantsol}. In this case
one cannot bring the system to the Birkhoff normal form. Yet, as explained in Appendix \ref{sec:equidistantsol}, one can generalize the theory that Rink built for the FPU \cite{rink2001symmetry} and show that system
(\ref{eq:hamav}) has KAM-tori near the minimum of $\hat U$
for a generic pairwise interaction potential $W_{avg}$.  Thus, Assumption IP1, of the existence of a smooth invariant  Diophantine torus near which the system is near-integrable and satisfies the twist condition, holds generically for the minimum line of $U$ corresponding to the equidistant distribution of phases. The genericity assumption may be checked by calculating the coefficients of the Rink normal form. Summarizing, even though Assumption IP1 could be difficult to check in general (system (\ref{eq:hamav})
has $(N-1)$ degrees of freedom, which can be arbitrarily large),  it is satisfied in the most natural setting.\\

\subsection{Choreographic solutions of smooth multi-particle systems }
The next theorem establishes the existence of KAM-stable choreographic solutions in the
multi-particle system (\ref{eq:multparsysy}) near \(\mathbf{L}^{*}(\theta_{min})\), the choreographic solution  (\ref{uncoupledch}) of the uncoupled system, where \(\theta_{min}\) is a minimum of the averaged potential. For a positive measure
set of initial conditions, all particles follow approximately the same path
with the phase difference between particles \(n\) and \(m\) remaining close to  \(\theta^{(n)}_{min} - \theta^{(m)}_{min}\)
for all time.

\begin{mainthm} \label{thm:mainbounded} Consider the system (\ref{eq:multparsysy}) where the single-particle system (\ref{eq:singleparticlesys}) has a periodic orbit \(L^{*}\) satisfying
the acceleration (SP1), no-resonance (SP2), and twist (SP3) assumptions, and let the $C^\infty$-smooth bounded pairwise interaction potential $W(q)$ be such that its averaged interaction potential $\hat U$ admits a minimum satisfying the KAM Assumption (IP1). Then, for all sufficiently small
$\delta>0$, system (\ref{eq:multparsysy}) admits a positive measure set of initial conditions corresponding to quasi-periodic solutions which satisfy, uniformly for all time $t$,
\begin{equation}\label{eq:coreogsol}
(q^{(n)}(t),p^{(n)}(t))=(q^*(\bar\omega t+\theta^{(n)}_{min}),
p^*(\bar\omega t+ \theta^{(n)}_{min})) + O(\delta^{1/4}), \quad n=1,\dots,N,
\end{equation}
with some constant $\bar\omega$ (which may depend on initial conditions) such that
$\bar\omega = \omega_0 + O(\delta^{1/2})$.
\end{mainthm}

The theorem is proven in Section \ref{sec:proofmainthm}.
\begin{rem*}\label{rem:maxpotential} The conclusion of Theorem \ref{thm:mainbounded} about the existence of a positive measure set of quasiperiodic choreographic motions also holds near non-degenerate {\em maxima} of  $\hat U$, provided the acceleration assumption SP1 is reversed to deceleration, i.e., if the period of $L^*$ {\em increases} with energy (as is the case near homoclinic loops, see Figure \ref{fig:saddlecenter}). Indeed, the quasiperiodic choreographic solutions remain such if we reverse the direction of time. This corresponds to changing the sign of the Hamiltonian
(\ref{eq:multparsysy}), i.e., of both \(H_{0}\) and \(W\). The change of the sign of $H_0$ makes the period of $L^*$
grow with energy (so the coefficient $a$ in (\ref{eq:hamav}) becomes negative, cf. (\ref{mb0})), while changing the sign of $W$ makes a minimum of the averaged potential a maximum.
\end{rem*}

\subsection{Choreographic solutions for repelling coupling potentials}\label{sec:substrongrepelling}

As the single-particle system (\ref{eq:birkhoffsingle}) near the elliptic orbit is nearly-integrable, the multi-particle system (\ref{eq:multparsysy}) near the set of choreographic solutions (\ref{uncoupledch}) is also nearly-integrable for small $\delta$. Therefore, the existence of KAM-tori established by the above theorem is not surprising. However, this result admits a generalization to systems with singularities, where the concept of near-integrability is not automatically applicable.

The simplest case corresponds to a singularity in the interaction potential.

\begin{defn}\label{defstr} We call a potential \(W(q)\) \textbf{repelling},
if it is smooth and bounded from below for all \(\|q\|>\rho\geqslant0\)  and $W(q)\rightarrow+\infty$ as \(\|q\|\rightarrow\ \rho\).
For \(\|q\|\leqslant\ \rho\), the function \(W\) is infinite. \end{defn}

Note that the repulsion assumption is made only for small distances, at large distances the potential may be attracting (e.g. the theory applies to the Lennard-Jones potential). If the pairwise interaction potential $W$ in (\ref{eq:multparsysy}) satisfies this definition, the perturbation term
$\delta\sum_{n\neq m} W(q^{(n)}-q^{(m)})$ can become large for arbitrary small $\delta$  - this occurs if $q^{(n)}$ gets sufficiently close to $q^{(m)}$
for some $n\neq m$. In particular, the averaged potential $W_{avg}$ in (\ref{eq:averagedw})
may have singularities when two particles moving on the same path $L^*$ collide.
When \(\rho=0\), this happens when two phases are identical
(\(\theta^{(n)}=\theta^{(m)}\) for some \(m\neq n\)) or when the orbit $L^*$ has
self-crossings in the $q$-space, see \ref{fig:ellipticorbit} (i.e.,  \(q^*(\theta^{(n)})=q^*(\theta^{(m)}) \) for some \(\theta^{(n)}\neq\theta^{(m)}\)).
If the dimension $d$ of the $q$-space is larger than $2$, then a typical periodic orbit has no
self-crossings. However, the existence of self-crossings is a robust phenomenon when $d=2$, or, in higher dimensions, when there are certain symmetries.

In general, we define
the {\em collision set} $\Theta_{c} \subseteq \mathbb{T}^N$ as the set of all initial phases $(\theta^{(1)}, \dots, \theta^{(n)})$ for
which the motion of uncoupled particles along the same path $L^*$ (see (\ref{uncoupledch})) leads to a collision at a certain time $t_c$:
$$\Theta_c = \{ (\theta^{(1)}, \dots, \theta^{(n)}) \; | \; \|q^*(\omega_{0}t_{c}+\theta^{(n)})-q^*(\omega_{0}t_{c}+\theta^{(m)})\|\leqslant\rho \text{ for some } n\neq m\ \text{and some } t_c\in[0,T]\}.$$
The collision set is closed. For \(\rho=0\), it is, typically, a union of codimension-1 hypersurfaces in $\mathbb{T}^N$, whereas for \(\rho>0\) this set has a non-empty interior, see Figure \ref{fig:collisionphases}. For too  large \(\rho\)  it coincides with the torus. For $(\theta^{(1)}, \dots, \theta^{(n)})$ outside the collision set, the averaged potential $U(\theta^{(1)}, \dots,\theta^{(N)})$ is well-defined and is a smooth function, bounded from below. Therefore, if
$\Theta_{c}$ is not the whole torus $\mathbb{T}^N$, then $U$ attains a finite minimum.

\begin{figure}
\begin{centering}
\includegraphics[scale=0.275]{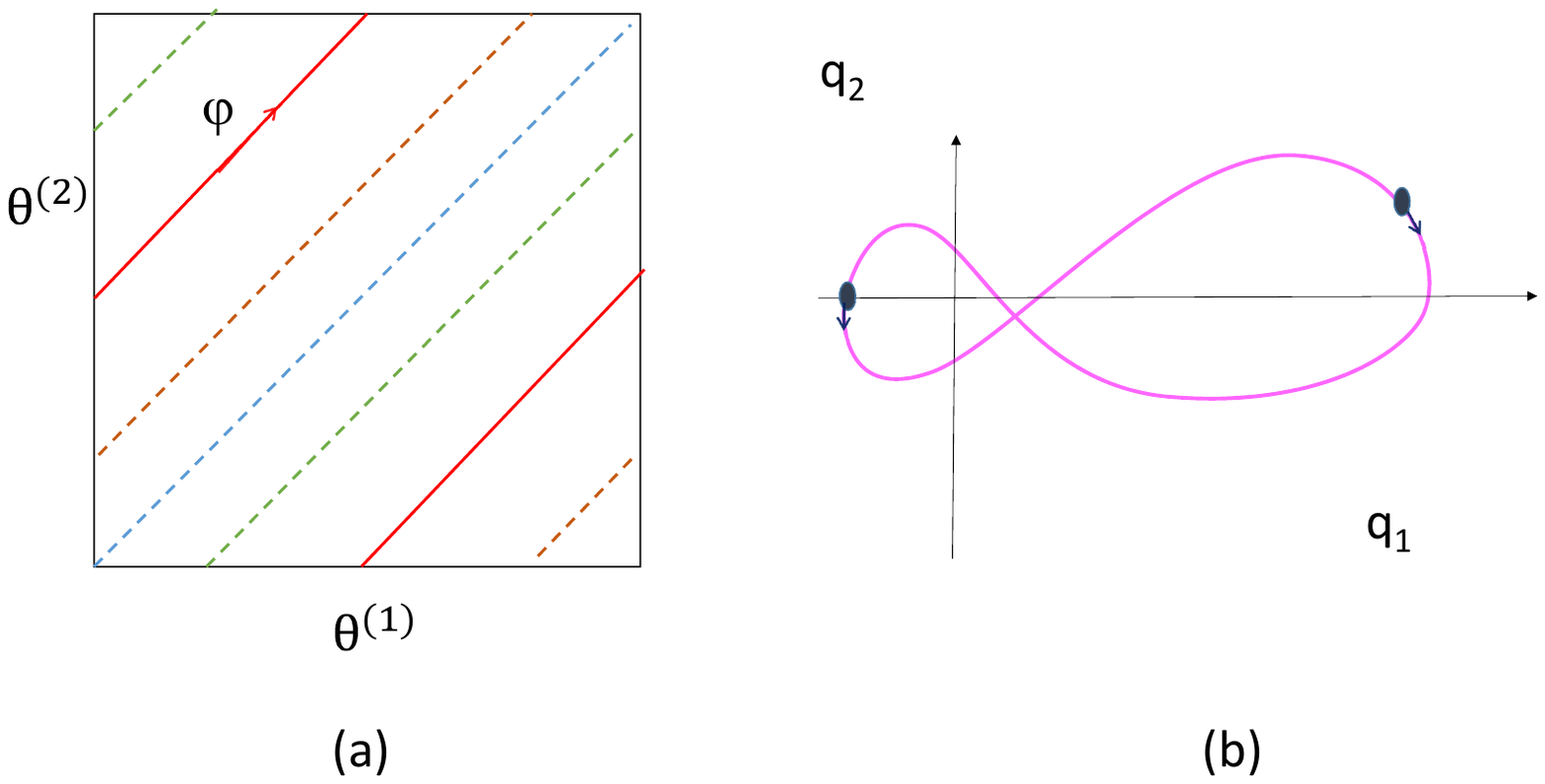}\\
\includegraphics[scale=0.2]{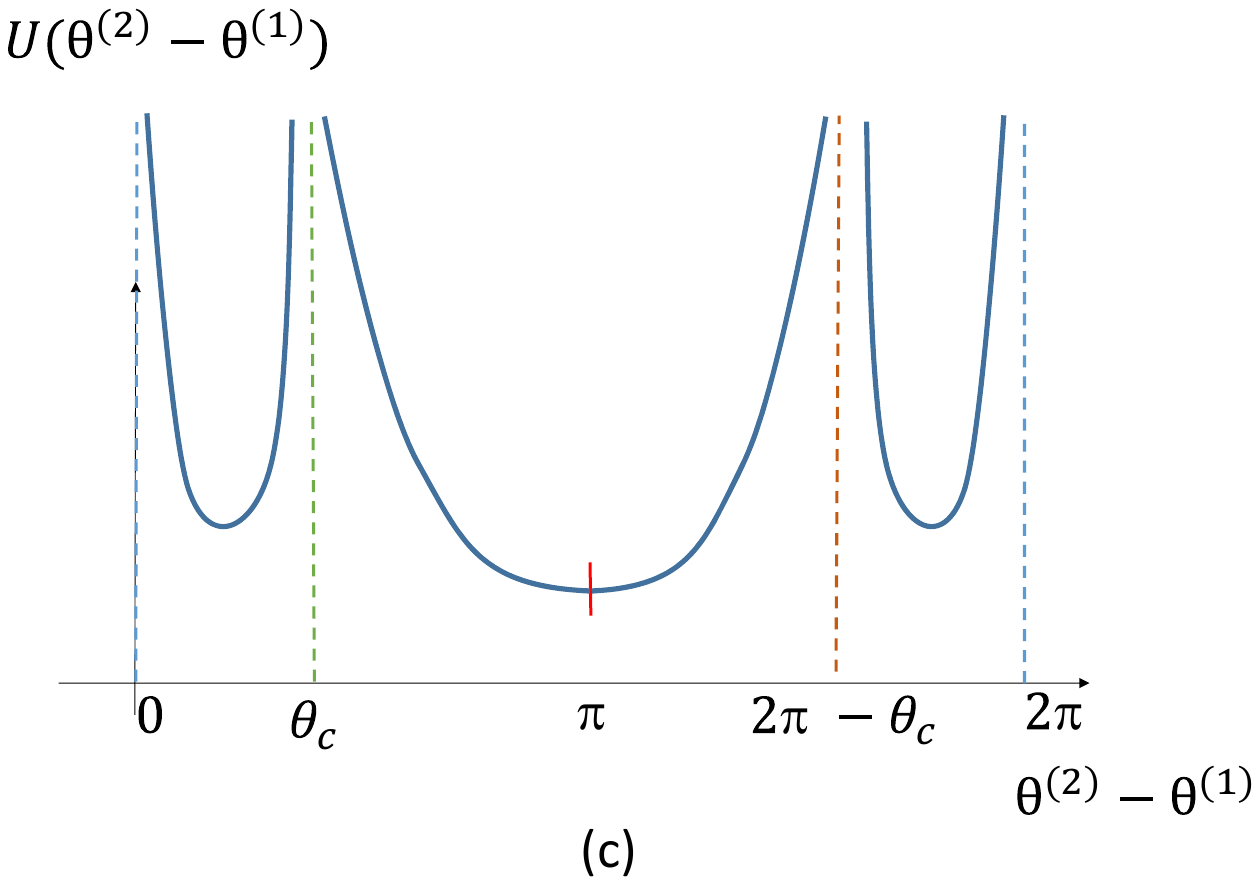}
\includegraphics[scale=0.2]{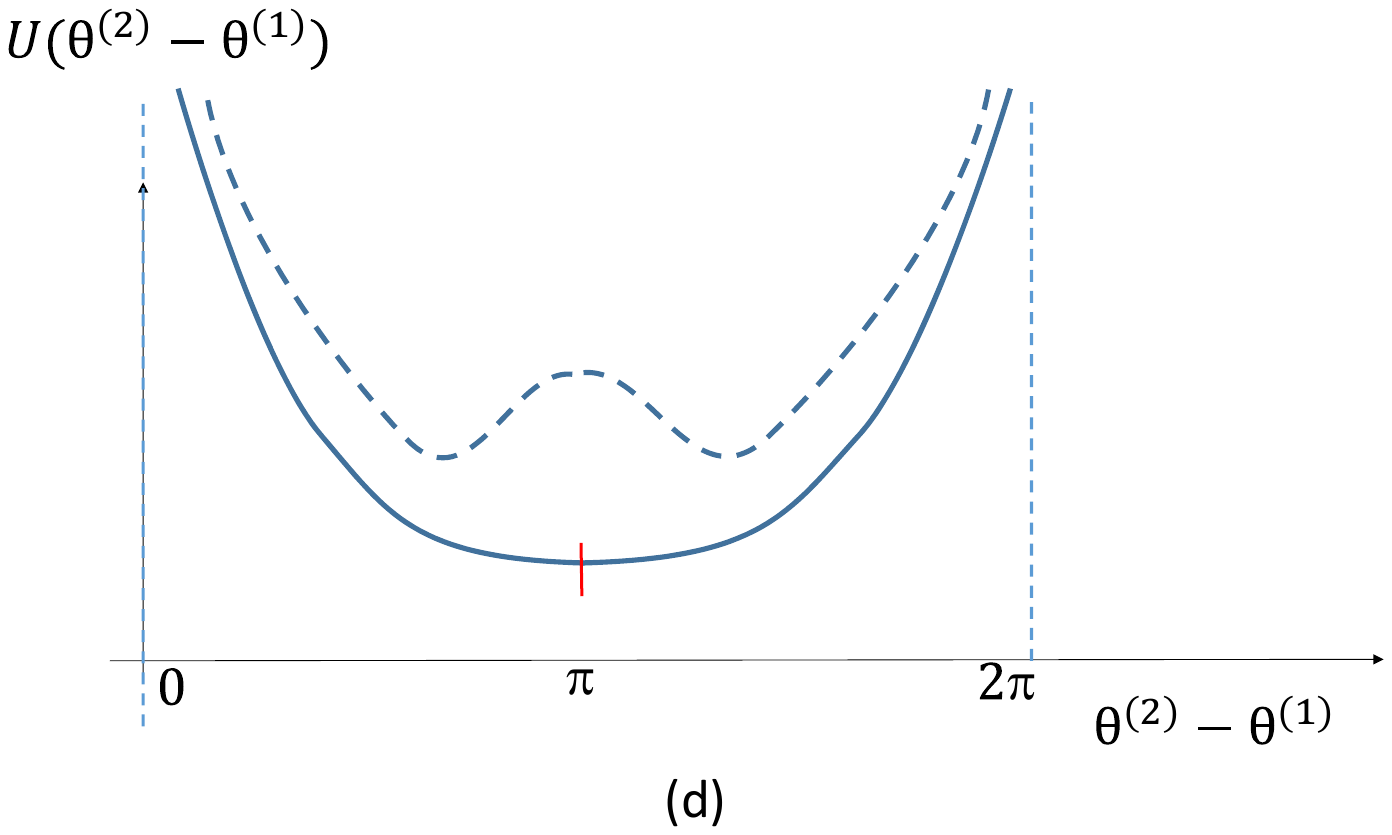}
\par\end{centering}
\protect\caption{\label{fig:collisionphases} A two-particle system for an orbit with one self-intersection  (a-c) and averaged potential \(U\) for the case of no self-intersection (d). (a) The collision set for \(\rho=0\): dashed blue line corresponds to the two particles coalescing,  \(\theta^{(1)}=\theta^{(2)}\); the dashed green and brown lines correspond to the two possible collisions at the intersection point,   \(\theta^{(1)}=\theta^{(2)}+\theta_c\)    and \(\theta^{(1)}=\theta^{(2)}-\theta_c\), respectively. (b) The two particles on the self-intersecting curve. (c) The averaged potential is infinite at the collision set. Provided the collision set does not include \(\theta_c=\pi\), it must have at least 3 minima (it may attain additional minima points). As the averaged potential is even,  \(\pi\)  must be an extremal point, here a minimum. (d) If there is no self-crossing, the only singularity of \(U\) is at the blue line, so, if \(U\)  has a single minimum, it must be at  \(\theta=\pi\). It is also possible that  \(\theta=\pi\) is a maximum of \(U\) (dashed curve), and then non-symmetric minima must exist.      }
\end{figure}

\noindent{\textbf{IP2: No-collision assumption.}} {\em There exists a minimum line (of the form (\ref{lineq})) of the averaged potential $U(\theta^{(1)}, \dots,\theta^{(N)})$ which is collision free, i.e.,
this line does not intersect the collision set \(\Theta_{c}\).}\\

If \(\rho=0\) and the elliptic periodic orbit \((q^*(\omega_{0}t), p^*(\omega_{0} t))\) of (\ref{eq:singleparticlesys})  has no self-intersection points, then Assumption IP2
automatically holds for any repelling potential and for any minimum line of $U$. Indeed, in this case the set $\Theta_c$ is exactly the set where at least two phases are equal. Then, at least two particles on the path $L^*$ have the
same coordinate $q$ for all time, hence the corresponding integrand in (\ref{eq:averagedpotu}) is infinite on the whole interval $[0,T]$, making the average
potential \(U\) infinite everywhere on $\Theta_c$. This cannot be a minimum of $U$.

In order to conclude the same when  $L^*$ has self-crossings  or when \(\rho>0\), it is enough to have  the repelling interaction sufficiently strong. For example:

\begin{lem}\label{lem:asumtion5sat}
If the pairwise interaction potential
\(W\) satisfies, for \(\|q\|>\rho\), the growth condition
\begin{equation}
W(q)\geqslant \frac{C_1}{\|q\|-\rho} - C_2, \quad  \quad\ C_{1}>0,
\label{eq:strongrepelling}\end{equation}
with some constants $C_{1,2}$, then the no-collision assumption IP2 holds, unless the collision set coincides with the whole torus
$\mathbb{T}^N$.
\end{lem}

\begin{proof} It is enough to show that the averaged potential $U$ is infinite for all $(\theta^{(1)}, \dots, \theta^{(n)})$ from the collision set $\Theta_c$. If $(\theta^{(1)}, \dots, \theta^{(n)})\in \Theta_c$, then for some $n\neq m$ either $\|q^*(\omega_{0}t+\theta^{(n)})-q^*(\omega_{0}t+\theta^{(m)})\|\leqslant\rho$ for all $t$, or there exists a value of $t=t_c$ such that $\|q^*(\omega_{0}t_c+\theta^{(n)})-q^*(\omega_{0}t_c+\theta^{(m)})\|=\rho$ and $\|q^*(\omega_{0}t+\theta^{(n)})-q^*(\omega_{0}t+\theta^{(m)})\|>\rho$
when $t$ approaches $t_c$. In the first case, an integrand in (\ref{eq:averagedpotu}) is infinite on the whole interval $[0,T]$, thus making $U$ infinite.  In the second case, since the derivative $\frac{d}{dt} q^*(\omega_{0}t+\theta)$ is bounded, the distance between the particles decays at least linearly in $(t-t_c)$, so the integrand $W(q^*(\omega_0t + \theta^{(n)})-q^*(\omega_0t + \theta^{(m)}))$
grows, by (\ref{eq:strongrepelling}), proportionally to $(t-t_c)^{-1}$ or faster, hence the integral (\ref{eq:averagedw}) diverges, i.e.,
$U$ is infinite in this case as well.
\end{proof}

There can, however, be cases when $\Theta_c=\mathbb{T}^N$ and Assumption  IP2 does not hold. For example, if the system is reversible, then there can exist orbits
(like period-2 orbits in billiards) for which \(q^*(\omega_0 t)=q^*(2\pi-\omega_0 t)\) for all \(t\). In this case, collisions are unavoidable even for $\rho=0$. Note, however, that generically,
since the orbit $L^*$ is elliptic, around it one can find resonant periodic orbits with at most finitely many self-crossings, and the previous remarks apply for sufficiently small $\rho$.

The following result generalizes Theorem \ref{thm:mainbounded} to the case of  repelling interaction potentials.
\begin{mainthm} \label{thm2} Consider system (\ref{eq:multparsysy}) where the single-particle system (\ref{eq:singleparticlesys}) has a periodic orbit \(L^{*}\) satisfying
Assumptions SP1,SP2,SP3, and the $C^\infty$-smooth, repelling potential $W(q)$ is such that   its  averaged interaction potential $\hat U$ admits a minimum satisfying the KAM assumption IP1 and the no-collision assumption IP2. Then, for all sufficiently small
\(\delta \), the system admits a positive measure set of initial conditions corresponding to quasi-periodic solutions as in Theorem \ref{thm:mainbounded}.\end{mainthm}

\begin{proof} By Assumption  IP2, we have that the uncoupled particles moving by the path $L^*$ (see (\ref{uncoupledch})) stay away from collisions, i.e.,
the distances $\|q^*(\omega_0 t+\theta^{(n)}_{min}) - q^*(\omega_0 t+ \theta^{(m)}_{min})\|$
stay bounded away from \(\rho\) for any $m\neq n$ and for all $t$. Therefore,
\begin{equation}\label{cutK}
W(q^*(\omega_0 t+\theta^{(n)}_{min}) - q^*(\omega_0 t+ \theta^{(m)}_{min}))< K
\end{equation}
for some constant $K$.

Replace the potential \(W\) by smooth and everywhere bounded potential
\(W^{cut}\) which coincides with \(W\) when \(W<K+1\). The corresponding averaged potential \(U^{cut}\) coincides with \(U\) in a neighborhood of $(\theta^{(1)}, \dots, \theta^{(N)}) =(\theta^{(1)}_{min}, \dots, \theta^{(N)}_{min})$, i.e., \(U^{cut}\) has the same minimum $(\theta^{(1)}_{min}, \dots, \theta^{(N)}_{min})$. By Theorem \ref{thm:mainbounded}
the multi-particle system (\ref{eq:multparsysy}) with potential \(W^{cut}\) has a positive measure set of quasiperiodic solutions for which $q(t)$ remains \(O(\delta^{1/4})\)-close to \(q^*(\bar \omega t +\theta_{\min}^{(1)}),\dots, q^*(\bar \omega t +\theta_{\min}^{(N)})\) where \(\bar \omega=\omega_0+O(\delta^{1/4})\). For sufficiently small \(\delta\), these solutions correspond to particles which stay away from each other for all times (because, for all times, \(q^*(\omega_0 t+\theta^{(n)}_{min})\) are bounded away from each other, thus, the same is true for \(q^*(\bar \omega t +\theta_{\min}^{(n)})\)).  Hence, for such solutions, if \(\delta\) is small enough,
the potential $W$ is bounded by (\ref{cutK}), so $W^{cut}$ coincides with $W$. Thus, they are also solutions of system
(\ref{eq:multparsysy}) with the original potential $W$.
\end{proof}

\begin{rem*}\label{rem:attrpotential}The result stays true if we reverse time, as in Remark \ref{rem:maxpotential} of Theorem \ref{thm:mainbounded}. This means
that we can replace the  repelling potential $W$ by a potential which {\em attracting}, i.e., bounded from above and tending to $-\infty$ as particles come close to each other. Then, Theorem \ref{thm2} implies that a positive measure set of quasiperiodic choreographic motions exists with phases close to the {\em maximum} of the averaged potential $\hat U$
(if it satisfies IP1 and IP2), provided the elliptic orbit $L^*$ of the single-particle system satisfies the non-degeneracy assumptions SP2 and SP3, and the acceleration assumption SP1 is reversed to deceleration, see Figure \ref{fig:saddlecenter}.
\end{rem*}

\begin{figure}
\begin{centering}
\includegraphics[scale=0.3]{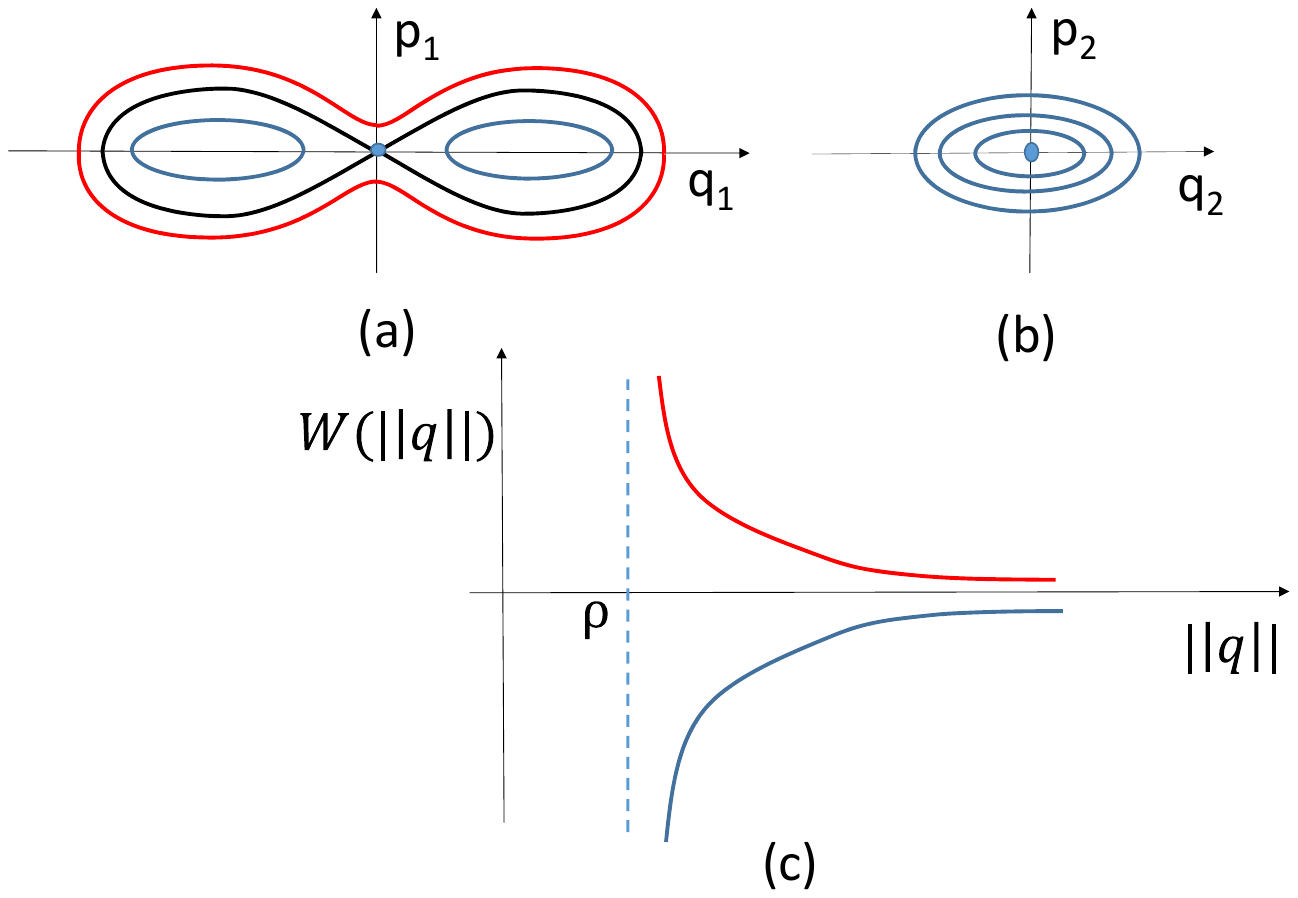}
\par\end{centering}
\protect\caption{\label{fig:saddlecenter} Accelerating and decelerating families of elliptic orbits coexist near homoclinic loops to a saddle-center.  Consider a product system of a Duffing oscillator (a) and a center (b). Then,  in the vicinity of the  homoclinic loop to the saddle-center there  exist three families of periodic orbits corresponding to  \((q_{2},p_2)=0\). For the family of periodic orbits which are outside the separatrix  (red curve in the \((q_{1},p_1)\) space) the period decreases with energy, whereas for the two families of periodic orbits which are inside the separtrix (blue curves in the \((q_{1},p_1)\) space) the period decreases with energy. By Theorem  \ref{thm:mainbounded} and Remark \ref{rem:maxpotential},   for a smooth interaction potential, choreographic solutions exist near the averaged potential minima lines for any fixed outer orbit, and near the   potential maxima lines for any given inner orbit. For smooth potentials, both types of lines must exist. Similar conclusions apply to singular potentials, where here,  when the potential is repelling (red curve in (c)) a minimum line must exist,  whereas, when  the potential is attracting (blue curve in (c))  a maximum line must exist, see Remark  \ref{rem:attrpotential}.}
\end{figure}

\subsection{\label{sec:highenergy}High-energy particles in a bounded domain}
Next, we apply the above methodology to the case of a system of repelling particles which are confined in a bounded domain. Let \(D\subset \mathbb{R}^d\) be a domain with a smooth (\(C^{\infty}\)) or piecewise smooth boundary $\partial D$ (in the piecewise smooth case, we call the points where $\partial D$ is smooth non-singular).
A particle confined in \(D\) is described by the Hamiltonian:
$$H=\frac{p^{2}}{2} +  V(q),$$
where the potential \(V(q)\) is a $C^\infty$-function defined in the interior of $D$ and tending to $+\infty$ on
$\partial D$ (here, if the particle is  a ball of a finite diameter \(\rho\), the domain \(D\) is the set of all possible positions of the ball center). When the growth of $V$ at the approach to $\partial D$ is reasonably regular, the high-energy motion limits to the billiard in \(D\), as described in \cite{rom2012billiards,RapRKT07apprx}. In order to simplify the analysis of the transition to the billiard limit, we restrict the class of confining potentials by assuming a power-law growth of \(V\) near $\partial D$ (we use the notation BD for assumptions we make on the single particle confined in the bounded domain).\\

\noindent\textbf{BD1: Power-law growth assumption.}  {\em Given any compact subset of the non-singular part of
$\partial D$, in a small neighborhood of this set the potential $V$ is given by
\begin{equation}\label{eq:potentialboundarylayer}
V(q)=\frac{1}{Q(q)^{\alpha}},
\end{equation}
where $\alpha>0$, and the $C^\infty$ function $Q$ measures the distance to the boundary of $D$, i.e.,
$Q(q)|_{q\in\partial D}=0$
and $\nabla Q(q)\neq0$; we also choose the sign of $Q$ such that \(Q(q)>0\) inside the domain \(D\).}\\

When we consider \(N\) mutually repelling particles moving in the potential field $V$, their motion is described by the Hamiltonian
\begin{equation}
H=\sum_{n=1,\dots, N }\left[\frac{(p^{(n)})^2}{2}+ V(q^{(n)})\right]+
\mathop{\sum_{n,m=1,\dots,N}}_{n\neq m}  W(q^{(n)}-q^{(m)}),\label{eq:nparticlesbillnorescale}
\end{equation}
where $(q^{(n)},p^{(n)})$ are coordinates and momenta of the $n$-th particle and \(W\) is a  repelling potential, as in Definition \ref{defstr} (note that we do not assume that the interaction potential is small here).

We consider a limit of large energy per particle, namely, we study the behavior at the energy level
$H((q^{(1)},p^{(1)}), \ldots, (q^{(N)},p^{(N)})) = N h$ for large $h$. We scale the momenta $p^{(n)}$ to
$\sqrt{2h}$, so the Hamiltonian transforms to
\begin{equation}\label{eq:multparsysybill0}
H=\sum_{n=1,\dots,N}\left[\frac{(p^{(n)})^2}{2}+\delta V(q^{(n)})\right]+
\delta\mathop{\sum_{n,m=1,\dots,N}}_{n\neq m} W(q^{(n)}-q^{(m)}),
\end{equation}
where \(\delta=\frac{1}{2h}\) (the inverse temperature). Now, the goal is to study the behavior on the fixed energy level \(H=\frac{N}{2}\) in the limit $\delta\to 0^+$.

System (\ref{eq:multparsysybill0}) is similar to system (\ref{eq:multparsysy}). However, here
the single-particle system
\begin{equation}\label{eq:singlepartbiliarddel}
H_0(q,p;\delta)=\frac{p^{2}}{2} + \delta V(q),
\end{equation}
depends on \(\delta\) in a singular way. The formal limit of the potential energy term as $\delta\rightarrow 0^+$ is the billiard potential, which is zero inside $D$ and infinite at the boundary of $D$. The corresponding dynamical
system, the billiard in the domain \(D\) \cite{chernov2006chaotic}, is not smooth, so our
Theorems \ref{thm:mainbounded} and \ref{thm2} cannot be directly applied. However, the method we used there
can be carried over to this case as well, with the help of an enhanced version of our theory of billiard-like potentials \cite{RapRKT07apprx,rom2012billiards}.

Recall that the billiard dynamics can be viewed as a motion of a particle along straight segments with speed $1$, interrupted by jumps in momenta as the particle reflects from the boundary. The jumps are defined by the elastic reflection law, with the angle of incidence equal to the angle of reflection. Equivalently, the dynamics are determined by the billiard map, which records the position and the angle of reflection at impacts. The dynamics of the smooth system at small $\delta$ can be quite different from the dynamics of the formal billiard limit.  Still, this formal limit provides good approximation for {\em regular billiard orbits}, which are defined as orbits for which all impact points  are bounded away from singularities of the billiard boundary, and
{\em all the impact angles are bounded away from zero} \cite{RapRKT07apprx,rom2012billiards,turaev1998elliptic}.

Thus, let \(L^*=\{(q^*(\omega_0t), p^*(\omega_0t)) \}\) denote a regular periodic orbit of the billiard in
\(D \), which hits the billiard boundary at points $M^1, \dots, M^{k^*}$ (we call them
{\em impact points}, to distinguish from multi-particle {\em collision} points). Let $t^1, \dots t^{k^*}$ be the
impact moments or time, i.e., $M^j= q^*(\omega_0 t^j)$, $j=1,\dots, k^*$.
The functions $q^*$ and $p^*$ are $2\pi$-periodic. As this is a billiard orbit,
\(p^{*} \) is a piece-wise constant function of time, with the jumps of $p^*$ happening at
$t=t^j\; {\rm mod} 2\pi$, $j=1,\dots, k^*$. The energy conservation implies that $\|p^*\|$ stays constant:
\(\|p^{*}\|=1\); thus, the frequency $\omega_0$ is such that $\frac{2\pi}{\omega_0}$ equals the length of $L^*$.
The function \(q^*\) is continuous and piece-wise linear, since $\frac{d}{dt} q^*=p^*$ when
$t\neq t^j\; {\rm mod} 2\pi$, $j=1,\dots, k^*$. The regularity of the orbit $L^*$ means that the boundary
of $D$ is smooth at each of the points $M^j$ and the vectors $p^*(\omega_0 t^j\pm 0)$
are not tangent to the boundary of $D$ at $M^j$, $j=1,\dots, k^*$.
The impact points $M^1, \dots, M^{k^*}$ comprise a periodic orbit of the billiard map \(B_{D}\): each of them is a fixed point of the {\em  billiard return map}  \((B_{D})^{k^*}\). Since the impact points are non-singular and the impacts are non-tangent, this map is smooth in a small neighborhood of any of the impact points.

\noindent\textbf{BD2: Elliptic orbit assumption.} {\em The regular billiard periodic orbit \(L^{*}\) is elliptic and KAM-nondegenerate. Namely, the point $M^1$ is a KAM-nondegenerate elliptic periodic point of the billiard map.
This means that two conditions are fulfilled. First, the multipliers \((\exp(\pm i\frac{2\pi }{\omega_{0} }\omega_{1}),\ldots,\exp(\pm i\frac{2\pi }{\omega_{0} }\omega_{d-1}))\) (the eigenvalues of the derivative of
 \((B_{D})^{k^*}\) at the point $M^1$) are non-resonant up to order 4, namely
$\displaystyle m_0 \omega_0 + \sum_{j=1}^{d-1} m_j \omega_j \neq 0$
for all integer $m_0$ and $m_1,\dots, m_{d-1}$ such that $\displaystyle 1\leqslant \sum_{j=1}^{d-1} |m_j|\leqslant 4$. This implies that
the Birkhoff normal form for \((B_{D})^{k^*}\) in the action-angle coordinates
\((I,\Phi)\in\mathbb{R}^{d-1}\times \mathbb{T}^{d-1}\) near $M_1^*$ is given by
\((I,\Phi)\mapsto (\bar I,\bar\Phi)\), where
 \begin{equation}\label{eq:birkhoffbillmap}
\bar I=I+o(I),\ \bar \Phi=\Phi+\frac{2\pi }{\omega_{0} }\hat \omega +\Omega I+o(I),
\end{equation}
with  constant $\hat\omega=(\omega_1,\dots, \omega_{d-1})$ and  \(\omega_{0}=\frac{2\pi}{|L^{*}|}\), where $|L^*|$ is the length of $L^*$.
The second KAM-nondegeneracy condition (the twist condition) is
\begin{equation}\label{eq:nondegOmega}
\det(\Omega)\neq 0.
\end{equation}}

\begin{figure}
\begin{centering}
\includegraphics[scale=0.3]{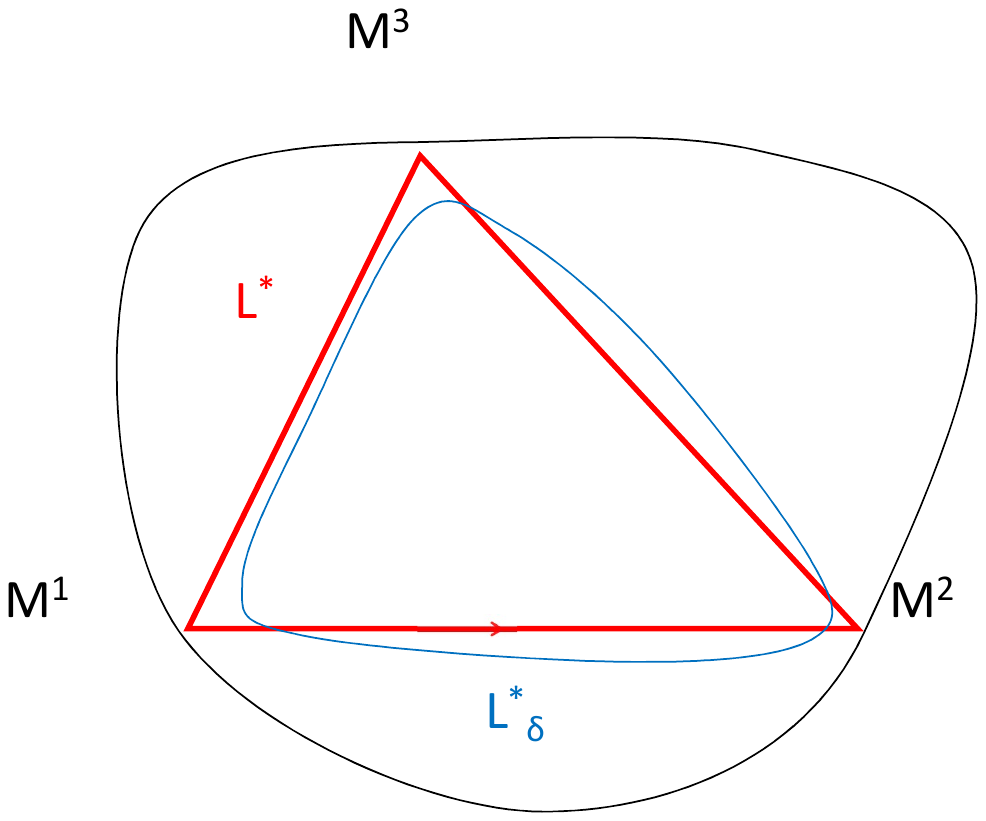}
\par\end{centering}
\protect\caption{\label{fig:bill} Stable periodic orbit in a convex billiard table (\(L^{*}\)) and a nearby stable periodic orbit  (\(L_{\delta}^{*}\)) for the smooth billiard-like potential at high energy (\(\delta=\frac{1}{2h}\), see (\ref{eq:singlepartbiliarddel})).        }
\end{figure}

The existence of a periodic orbit satisfying this assumption holds true for an open set of billiards; for convex billiards in the plane this assumption is also open and dense \cite{Dvorin1973}.

The billiard return map \((B_{D})^{k^*}\) is smoothly conjugate (by the billiard flow) to the return map to any small
cross-section to \(L^{*}\) chosen in the interior of \(D\). As shown in \cite{turaev1998elliptic,RapRKT07apprx}, in the limit
\(\delta\rightarrow +0\), the return map of the smooth flow of (\ref{eq:singlepartbiliarddel}) on such cross-section tends, in $C^\infty$, to the return map of the billiard flow. Thus, up to a change of coordinates, the billiard return map describes the limit of the smooth dynamics defined by (\ref{eq:singlepartbiliarddel}). In particular, with Assumptions BD1 and BD2, the single-particle system (\ref{eq:singlepartbiliarddel}) has, for all sufficiently small
\(\delta\),  a KAM-nondegenerate elliptic periodic orbit \(L^*_{\delta}\) in the energy level $H_0=\frac{1}{2}$, which is close to
\(L^{*}\).\\

Now, we consider the product of \(N\) billiard flows in \(D\), with the invariant set as in (\ref{uncoupledch}):
\begin{equation}\label{uncoupledchD}
q^{(n)}=q^*(\omega _{0}t+\theta^{(n)}), \;\; p^{(n)}=p^*(\omega _{0}t+\theta^{(n)}),
\qquad n=1,\dots,N.
\end{equation}
The projection of this invariant set to the \(Nd\)-dimensional configuration space is an \(N\)-dimensional torus (continuous, but only piece-wise smooth).  The averaged potential on this torus is defined as in (\ref{eq:averagedpotu}):
\begin{equation}\label{eq:averagedpotuD}
U(\theta^{(1)}, \dots,\theta^{(N)})=\mathop{\sum_{n,m=1,\dots,N}}_{n\neq m}
 \frac{1}{T}\int_0^{T} W(q^*(\omega _{0}t+\theta^{(n)})-q^*(\omega _{0}t+\theta^{(m)}))dt.
\end{equation}
We assume that there exists a line (\ref{lineq}) of minima of $U$:
\begin{equation}\label{eq:lineminbill}\theta^{(n)}=\theta^{(n)}_{min}+c, \qquad n=1, \dots, N,\end{equation}
which satisfies the KAM assumption IP1 of Section \ref{sec:condonw} and the no-collision assumption IP2 of Section \ref{sec:substrongrepelling}.

The KAM assumption requires a sufficient smoothness of the averaged potential, which does not, a priori, hold
for billiard orbits because $q^*(t)$ is not smooth at the impact points. In general, the non-smoothness of the
system at $\delta=0$ can make the averaging procedure invalid and lead to dynamics different from those in the smooth case. However, we show that these issues do not materialize (e.g. we prove the smoothness of the averaged potential, see  Lemma \ref{Lemmaavbil}) if the  non-interacting particles moving along the same billiard trajectory $L^*$ with the phase shifts
$\theta^{(n)}_{min}$ {\em never hit the billiard boundary simultaneously}:

\noindent\textbf{IP3:  Non-simultaneous impacts assumption.} \textit{The impacts of
\(q^*(\omega_0t+\theta_{min}^{(n)})\) with the
billiard boundary do not happen simultaneously, namely, if
\(\omega_0 t+\theta_{min}^{(n)}=\omega_0 t^j\) (mod \(2\pi\)) for some \(j\), then
 \(\omega_0 t+\theta_{min}^{(m)}\neq \omega_0 t^k\) (mod \(2\pi\)) for all \(k\) and all $m\neq n$:}
\begin{equation}\label{eq:asynchrD}
\theta_{min}^{(n)}-\theta^{(m)}_{min}\neq\omega_0(t^j-t^k)  \mod 2\pi.
\end{equation}

For convenience, we can always assume (by redefining $c$ in (\ref{lineq})) that
\begin{equation}\label{eq:sumtminD}
\sum_{n=1}^N \theta^{(n)}_{min}=0;
\end{equation}
by a shift of time, we can also achieve that
\begin{equation}\label{eq:noimpactthetnD}
\theta^{(n)}_{min}\neq\omega_0 t^j,\quad \mbox{ for all } \;\; n=1, \dots, N,\quad j=1,..k^*.
\end{equation}
\begin{mainthm}
\label{thm:billiardsystem} Consider $N$  repelling particles that are confined to a region
\(D\)  by a trapping potential satisfying the power-law assumption BD1. Assume that the billiard table \(D\)  has a regular elliptic periodic orbit  \(L^* \) which satisfies the elliptic orbit assumption BD2, and that the averaged interaction potential has a minima line satisfying the KAM assumption IP1, the no-collision assumption IP2, and the non-simultaneous impacts assumption IP3. Then, for all sufficiently high values of the energy-per-particle \(h\), the $N$-particle system
(\ref{eq:nparticlesbillnorescale}) has a positive measure set of initial conditions corresponding to quasi-periodic solutions as in Theorem \ref{thm:mainbounded}, with
$\bar\omega =\frac{2\pi \sqrt{2h} }{|L^*|}(1+ o_{h\rightarrow\infty}(1))$. In particular, this system is not ergodic for all sufficiently high energies. \end{mainthm}

The proof is in Section \ref{sec:proofbilliards}. It is an empirical fact that Hamiltonian systems with low number of degrees of freedom have elliptic periodic orbits easily, unless the system is specially prepared to have a (partially) hyperbolic structure on every energy level. Therefore, a common belief (and a challenging conjecture to prove) is that a generic Hamiltonian system without the uniform partially-hyperbolic structure possesses a non-degenerate elliptic orbit. The billiard counterpart of such claim would be that a generic billiard which is not of the dispersing or defocusing type \cite{Wojtkowski2007}
has a non-degenerate elliptic orbit. Currently, no methods are known for proving such conjecture in any reasonable regularity class. But, once we accept this conjecture for systems with low number of degrees of freedom, Theorem \ref{thm:billiardsystem} implies that the gas of any number of repelling particles confined in
a domain with a sufficiently smooth boundary is {\em generically non-ergodic} for all sufficiently high temperatures.

\subsection{Particles in a rectangular box}
The single-particle billiard in a rectangular box has no elliptic periodic orbits: it is an integrable system (with partial oscillatory motions parallel to different coordinate axes independent of each other), so the periodic orbits are parabolic. The orbits of the same period go in several continuous $(d-1)$-parameter families; the orbits in the same family can be distinguished by the phase differences between the partial oscillations or by the coordinates of the impact points. Namely, if the box sizes
are $(l_1,\dots,l_d)$ and the conserved kinetic energies of the corresponding partial oscillations are $E_j=\frac{1}{2}p_j^2$, then the frequencies of the partial oscillations are $2\pi \frac{|p_j|}{l_j}$, and the single-particle motion is periodic if and only if the ratio of each two of these frequencies is a rational number. Thus, we have a discrete set of possible choices of partial energies, for which the motion with any initial point $(q_1,\dots,q_d)$ in the box is periodic with the same period (for any choice of the signs of $p_j$).

\begin{figure}
\begin{centering}
\includegraphics[scale=0.3]{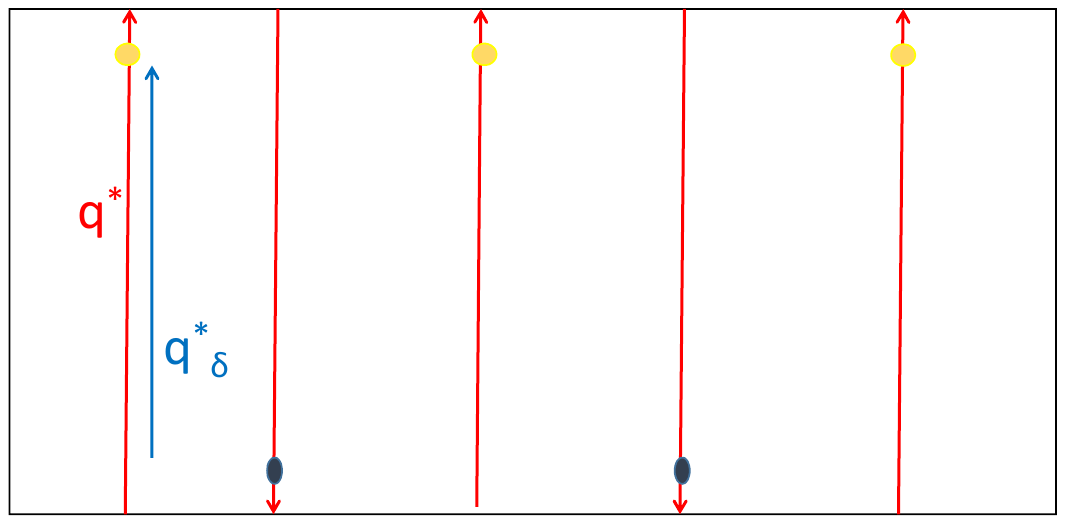}
\par\end{centering}
\protect\caption{\label{fig:box} In a rectangular box, synchronized parallel motion by which all particles move vertically and hit the  boundary simultaneously is KAM-stable under the conditions of Theorem \ref{thm:billiardbox2}. Less ordered parallel motion, by which all particles have the same period yet each particle turns at a different moment, is KAM-stable  under the conditions of Theorem \ref{thm:billiardbox1}.       }
\end{figure}

While similar computations can be performed for any of these families, we choose the simplest one, where all the particles move strictly along one of the coordinate axes, i.e., the family is given by the equation $p_1=\ldots = p_{d-1}=0$. We call such oscillations vertical; the particle moves up for a half of the period and it moves down for the other half. When the energy is fixed, different periodic orbits in this family are distinguished by the values of the ``horizontal'' coordinates $(q_1,\ldots, q_{d-1})$, which do not change with time. In the same spirit as before, one can place any number of non-interacting particles on this family (each particle with the same kinetic energy, but on its own path, i.e., with different values of the horizontal coordinates). The difference with the previous cases is that we now allow the particles to spread over a continuous family and not over just one orbit. If the particle energy is high enough, switching the repulsion between the particles on makes only a small perturbation of the fast vertical motion. We show below that the slow evolution of the horizontal degrees of freedom and the differences between the phases of the vertical oscillations are governed, to the main order, by the averaged potential; its non-degenerate minima correspond to elliptic orbits of the multi-particle system.

Note that we have only one fast degree of freedom for the entire multi-particle system in this setting (the sum of the phases of the vertical oscillations). This makes the averaging procedure simpler than in the previous cases. However,
the non-simultaneous impacts assumption IP3, which is crucial for justification of the averaging in
Theorem \ref{thm:billiardsystem}, can be violated for the family of vertically oscillating particles for an open set of repelling potentials (see below). We therefore develop a different approach for the case of simultaneous impacts.

Let us describe the assumptions we impose on the system in the box.

\noindent\textbf{Box1: Separability assumption.}  {\em The single-particle Hamiltonian is given by
\begin{equation}\label{eq:vinbox}
H_0(q,p)=\sum_{i=1}^d (\frac{p_i^2}{2} + V_{i}(q_{i})),
\end{equation}
with
\begin{equation}\label{eq:potentialboundarylayerbox}
V_{i}(q_{i})=\frac{1}{Q_{i}(q_{i})^{\alpha}},\quad  \alpha>0
\end{equation}
where the $C^\infty$ function  $Q_{i}$ measures the distance to the box boundary in the $i$-th coordinate direction, i.e., $Q_{i}(0)=0$, $Q_i(l_i)=0$, $Q'_i(0)>0$, $Q_i'(l_i)<0$, and $Q_i(q_i)>0$ for $q_i\in (0,l_{i})$. Finally assume that the potential is symmetric\footnote{This symmetry assumption appears to be non-essential. It is not used at all in Theorem  \ref{thm:billiardbox1}. We include it here as it is natural and makes some notations and computations in the proof of Theorem \ref{thm:billiardbox2} simpler. } in the vertical direction: \(Q_{d}(l_{d}-q_{d})=Q_d(q_d).\)}

Thus, the $N$-particle Hamiltonian has the form
\begin{equation}\label{eq:Hamparticleinbox}
H=\sum_{n=1,\dots,N}\sum_{i=1}^d \left[\frac{(p_{i}^{(n)})^2}{2}+V_{i}(q_{i}^{(n)})\right]+\mathop{\sum_{n,m=1,\dots,N}}_{n\neq m}  W(q^{(n)}-q^{(m)}),
\end{equation}
where $W$ is a repelling potential, $C^\infty$ for $\|q^{(n)}-q^{(m)}\|> \rho$ (see Definition \ref{defstr}).

We consider the limit of large energy per particle, and look for motions which are fast only in the last coordinate. Namely, we study the behavior at the energy level $H = N h$ for a fixed $N$ and large $h$ where most of the particles' energy is at the vertical motion.
We scale the vertical momenta $p_{d}^{(n)}$  by $\sqrt{2h}$, and the Hamiltonian transforms to
\begin{equation}\label{eq:multparsysybox}
H=\sum_{n=1,\dots, N}\left[\frac{(p_{d}^{(n)})^2}{2}+\delta V_{d}(q_{d}^{(n)})\right]+
\delta\sum_{n=1,\dots, N}\sum_{i=1}^{d-1}\left[\frac{(p_{i}^{(n)})^2}{2}+ V_{i}(q_{i}^{(n)})\right]+
\delta\mathop{\sum_{n,m=1,\dots,N}}_{n\neq m} W(q^{(n)}-q^{(m)}),
\end{equation}
where \(\delta=\frac{1}{2h}\); as in section  \ref{sec:highenergy}, we study the behavior on the fixed energy
level \(H=\frac{N}{2}\) in the limit $\delta\to+0$.

In the limit $\delta=0$, the Hamiltonian describes $N$ independent vertical, constant speed, saw-tooth motions.
Setting all the particles to have the same speed $|p_{d}^{(n)}|=1$ in the limit $\delta=0$, and choosing the vertical size of the box $l_d=\pi$, we obtain the limiting  family of solutions
in the form $ q_{i}^{(n)}(t)=constant$ for \( i=1,\dots,d-1\) and \(q_{d}^{(n)}(t)=q^*(t+\theta^{(n)})\) where $q^*$ is the $2\pi$-periodic saw-tooth function:
\begin{equation}\label{eq:sawtoothsol}
q^*(t)=\left\{\begin{array}{cc}t, & \mbox{ for } ~ t\in[0,\pi], \\
2\pi-t, & \mbox{ for } ~ t\in[\pi,2\pi]. \\
\end{array}\right.
\end{equation}Denote the horizontal coordinates of the \(n-\)th particle  by \(\xi^{(n)}\) (so \(\xi^{(n)}=(\xi ^{(n)}_1,\dots, \xi^{(n)}_{d-1})=(q^{(n)}_1,\dots, q^{(n)}_{d-1})\)). Define the averaged potential, \begin{equation}\label{eq:averagedparaldelta0}
U(\theta,\xi)=\mathop{\sum_{n,m=1,\dots,N}}_{n\neq m}
 W_{avg}(\theta^{(n)}-\theta^{(m)},\xi^{(n)}-\xi^{(m)})+\sum_{n=1,\dots,N}\sum_{i=1}^{d-1}  V_{i}(\xi_i^{(n)})
\end{equation}
where
\begin{equation}\label{eq:potsawtooth}
 W_{avg}=\frac{1}{2\pi}\int_0^{2\pi} W(q^*(s+\theta^{(n)})-q^*(s+\theta^{(m)}),\xi^{(n)}-\xi^{(m)})ds,
\end{equation}and $\theta=(\theta^{(1)},\dots,\theta^{(N)})$, $\xi=(\xi_1^{(1)},\dots, \xi_{d-1}^{(1)},\dots, \xi_1^{(N)},\dots, \xi_{d-1}^{(N)})$.
Let
$\|\xi^{(n)}- \xi^{(m)}\|>\rho$ for every $n\neq m$.  We establish in Section \ref{sec:boxsection} (see Lemma \ref{lem:wtildc2inthet}) that  the averaged potential is $C^\infty$-smooth if  $\theta^{(n)}\neq \theta^{(m)} \mod \pi$ for every $n\neq m$. Moreover, we also show that under the parity assumption Box4 below, the averaged potential is, along with all its derivatives with respect to $\xi$, at least $C^2$-smooth function of $\theta$ even if
$\theta^{(n)}= \theta^{(m)} \mod \pi$ for some, or all, $n\neq m$.

Since \(V_{i}\) and  \(W_{avg}\) are bounded from below,  the averaged potential $U(\theta,\xi)$ must have a minimum line
\begin{equation}\label{lnqbox}
\theta^{(n)}=\theta^{(n)}_{min}+c, \qquad \xi^{(n)}=\xi^{(n)}_{min},
\end{equation}
where $c$ is an arbitrary constant. Like in Section \ref{sec:condonw},
we can introduce coordinates $(\varphi, \psi, \xi)$ in a small neighborhood of this line such that
$\varphi=\frac{1}{N}(\theta^{(1)}+ \dots +\theta^{(N)})$ and the coordinates $\psi=(\psi_1, \dots, \psi_{N-1})$ are linear combinations of
the phase differences $(\theta^{(n)}-\theta^{(n)}_{min}) - (\theta^{(m)}-\theta^{(m)}_{min})$. Since the averaged potential $U$
depends only on the differences of the phases, we obtain that it is independent of $\varphi$, so, as in (\ref{uhat}), we set
$$U(\theta,\xi)=\hat U(\psi, \xi).$$

\noindent\textbf{Box2: Non-degenerate minimum assumption.}  {\em The minimum of the averaged potential $\hat U$ corresponds to
$\|\xi^{(n)}_{min}- \xi^{(m)}_{min}\|>\rho$ for all $n\neq m$. The Hessian matrix of  $\hat U$ at the minimum $\psi=0, \xi=\xi_{min}$ is non-degenerate,
and all its eigenvalues are simple.}

The condition
$\|\xi^{(n)}_{min}- \xi^{(m)}_{min}\|>\rho$ means that the particles stay away from each other, each on its own path. This assumption  is fulfilled automatically when, for example, $W$ satisfies (\ref{eq:strongrepelling}). Indeed, then,  by Lemma \ref{lem:asumtion5sat},  the minimum of the averaged potential cannot correspond to collisions, yet, two particles on the same vertical path  collide unavoidably.

The non-degeneracy of the Hessian is a generic condition, implying that
$(\psi=0,\xi=\xi_{min})$ corresponds to the elliptic equilibrium of the Hamiltonian
\begin{equation}\label{eq:Hrectavrg2assum}
\begin{array}{ll}
H=\frac{1}{2N} J^{2} + \frac{p_{\xi}^2}{2} +\hat U(\psi,\xi),   \end{array}
\end{equation}
where $J$ and $p_{\xi}$ denote the conjugate momenta to $\psi$ and $\xi$, respectively.

Let us first consider the case of {\em non-simultaneous impacts} motion at which
\(\theta^{(n)}_{min}\neq \theta^{(m)}_{min} \mod \pi \) for all \(n\neq m\) (a particle impacts to the boundary happen exactly at each half-period,
so this condition, obviously, means that no two particles hit the boundary simultaneously). By Lemma  \ref{lem:udelt0isc2} the averaged potential near the minimum is  $C^\infty$-smooth,  so we  impose the following genericity condition (which involves the Taylor expansion up to order 4).

\noindent\textbf{Box3: KAM assumption (the case of non-simultaneous impacts).} \textit{The local minimum of  \(\hat U\)  at \((\psi=0,\xi=\xi_{min})\)
is KAM-non-degenerate: the corresponding elliptic equilibrium of the Hamiltonian system (\ref{eq:Hrectavrg2assum}) has no resonances up to order 4
and its Birkhoff normal form satisfies the twist condition.}

\begin{mainthm}\label{thm:billiardbox1} Consider $N$  repelling particles that are confined to a box by a trapping potential satisfying the separability assumption Box1. Let the averaged interaction potential have a minima line, corresponding to non-simultaneous impacts and satisfying the  nondegeneracy assumptions Box2 and Box3. Then, for all sufficiently high values $h$ of the energy per particle,  the $N$-particle system
(\ref{eq:Hamparticleinbox}) has a non-degenerate elliptic periodic orbit accompanied by a positive measure set of quasi-periodic solutions. In particular, for this set of initial conditions, each particle stays bounded away from all other particles for all time, so the system is not ergodic. \end{mainthm}

The assumption of non-simultaneous impacts is generic when the interaction potential has no special symmetries. However, for the most natural class of potentials which depend only on the Euclidian distance between particles, there is an inherit symmetry which can lock the impacts to become simultaneous. Such potentials satisfy

\noindent\textbf{Box4: Parity assumption.}  {\em The repelling interaction potential $W$  is
even in $q_i^{(n)}-q_i^{(m)}$ for each $i=1,\dots,d$.}

In this case, the average potential $U$ is an even function of
$(\theta^{(n)}-\theta^{(m)})$,  for any pair of $n$ and $m$. It is also $2\pi$-periodic in
$(\theta^{(n)}-\theta^{(m)})$. We conclude that if Box4 is satisfied, then\begin{equation}
U(\theta^{(n)}-\theta^{(m)},\cdot)=U(\theta^{(m)}-\theta^{(n)},\cdot) \text{ and } U(\theta^{(n)}-\theta^{(m)}-\pi,\cdot)=U(\pi-(\theta^{(n)}-\theta^{(m)}),\cdot).\label{eq:uparity}
\end{equation}  It follows that the first derivative of $U$ with respect to $\theta$ vanishes when
$\theta^{(n)}=\theta^{(m)} \mod \pi$ {\em for all} $m$ {\em and} $n$. Therefore, when the parity assumption holds, there can exist minima of $U$ for which {\em all the particles hit the boundary walls simultaneously} (each half-period, some particles hit $q_{d}=0$ while, at the same time, the others hit $q_{d}=\pi$); moreover, this {\em simultaneous impacts} property can persist for small perturbations of the potential $W$
within the class of potentials satisfying Assumption Box4.
The non-simultaneous impacts assumption, which is crucial for the averaged procedure we use in Theorems  \ref{thm:billiardsystem} and
\ref{thm:billiardbox1}, does not hold for such minima. We, therefore,  consider this case separately.

First note that the averaged potential $U$ in the simultaneous impact case is not, in general, $C^3$-smooth with respect to $\theta$
(see Lemma \ref{lem:wtildc2inthet}). Hence, there cannot be a direct analogue of the KAM-nondegeneracy assumption Box3, which involves
derivatives of $U$ up to order 4 (recall that the twist condition on  quadratic terms in the action variables corresponds to  algebraic relation between the Taylor coefficients of the Hamiltonian up to order 4 \cite{Arnold2007CelestialMechanics}). Instead of assumption Box3 which is formulated in terms of  the Hamiltonian system (\ref{eq:Hrectavrg2assum}), we formulate KAM nondegeneracy assumption in terms of the two auxiliary Hamiltonians:
\begin{equation}\label{eq:haveragedinabox}
H_{0}^\xi(p_{\xi},\xi)=\frac{p_{\xi}^2}{2}+U(\theta_{min},\xi),
\end{equation}
and
\begin{equation}\label{eq:haveragedthetasyncinabox}
H_{0}^{\theta}(p_{d},\theta) = \sum_{n=1,\dots N}  \frac{(p_{d}^{(n)})^2}{2}+
\mathop{\sum_{n,m=1,\dots,N}}_{n\neq m}\ \gamma_{nm}(\theta^{(n)}-\theta^{(m)}-\vartheta_{nm})^2 -\beta_{nm}(\theta^{(n)}-\theta^{(m)}-\vartheta_{nm})^4,
\end{equation}
where  $\vartheta_{nm}=\theta^{(n)}_{min}-\theta^{(m)}_{min} \mod 2\pi$ is either $ 0$ or $\pi,$ and
$$
\gamma_{nm} =\left\{\begin{array}{cl}
\frac{\partial^2 W}{\partial q_{d}^2} (0,\xi^{(n)}_{min}-\xi^{(m)}_{min}) & \text{for }\vartheta_{nm}=0 , \\
-\frac{1}{\pi}\frac{\partial W}{\partial q_{d}}(\pi,\xi^{(n)}_{min}-\xi^{(m)}_{min}) & \text{for }\vartheta_{nm}=\pi ,
\end{array}\right.$$

$$
\beta_{nm} =\left\{\begin{array}{cl}
-\frac{\partial^2 W}{\partial q_{d}^{2}}(0,\xi^{(n)}_{min}-\xi^{(m)}_{min}) & \text{for }\vartheta_{nm}=0 , \\
\frac{\partial^2 W}{\partial q_{d}^{2}}(\pi,\xi^{(n)}_{min}-\xi^{(m)}_{min}) & \text{for }\vartheta_{nm}=\pi .
\end{array}\right.$$
$$ $$
The quadratic part of the potential in \(H_{0}^{\theta}\) coincides with the quadratic term of the Taylor expansion of $U(\theta,\xi_{min})$
at $\theta=\theta_{min}$ (see Lemma \ref{lem:wavatorigin}). Due to the simultaneous impacts property, by redefining the constant $c$ in (\ref{lnqbox}), if necessary, we can always make $\theta^{(n)}_{min}=0 \mod \pi$. Then, by (\ref{eq:uparity}), the averaged potential $U$ is an even function of $(\theta-\theta_{min})$,
so the Hessian of $U$ is block-diagonal: the derivatives $\frac{\partial^2 U}{\partial \theta \partial \xi}$ vanish at the minimum
with simultaneous impacts. Therefore, the non-degenerate minimum assumption Box2 implies that both $U(\theta_{min},\xi)$ and $\hat U(\psi, \xi_{min})$
have non-degenerate minima (at $\xi=\xi_{min}$ and at $\psi=0$, respectively). We conclude that, under Assumption Box2, Hamiltonian (\ref{eq:haveragedinabox}) has an elliptic equilibrium at $\xi=\xi_{min}$, $p_{\xi}=0$, and Hamiltonian (\ref{eq:haveragedthetasyncinabox}) has a family of elliptic periodic orbits $\theta^{(n)}=\theta^{(n)}_{min} + \omega t$ ($n=1,\dots,N$), $|p_{d}^{(1)}|=\ldots=|p_{d}^{(N)}|=\omega=const$. This is a relative equilibrium, i.e., it becomes an equilibrium when we go to
translation invariant $(J,\psi)$-coordinates, like in Section \ref{sec:condonw}. Denote the reduced Hamiltonian  of (\ref{eq:haveragedthetasyncinabox}) by \( H^{\theta,0}(J,\psi)\).

\noindent\textbf{Box5: KAM assumption (the case of simultaneous impacts).} {\em Let a nondegenerate minimum line of $U$ satisfy $\theta^{(n)}_{min}=\theta^{(m)}_{min} \mod \pi$ for all $m$ and $n$. Assume that at the local minimum of  \(\hat U\)  at \((\psi=0,\xi=\xi_{min})\)
the frequencies of small oscillations  have no resonances up to order 4.
 Furthermore, assume that the elliptic equilibrium
of $H_{0}^\xi$ at $p_{\xi}=0,\xi=\xi_{min}$ and the elliptic  equilibrium of \(H^{\theta,0}\) at $(J,\psi)=0$ are KAM-non-degenerate, i.e. their Birkhoff normal forms satisfy the twist condition. }

\begin{mainthm}\label{thm:billiardbox2} Consider $N$  repelling particles that are confined to a box by a trapping potential satisfying the separability assumption Box1 with \(\alpha>6\). Assume the parity assumption Box4 holds. Assume the averaged interaction potential $U$ has a minima line with simultaneous impacts and let it satisfy the non-degeneracy assumptions of Box2 and  Box5. Then, for all sufficiently high values of the energy per particle, the $N$-particle system
(\ref{eq:Hamparticleinbox}) has a non-degenerate elliptic periodic orbit accompanied by a positive measure set of quasi-periodic solutions. In particular, for this set of initial conditions, each particle stays bounded away from all other particles for all time, so the system is not ergodic. \end{mainthm}

Thus, the gas of any number of highly-energetic repelling particles confined to a rectangular box by a sufficiently steep potential
is, generically, non-ergodic.

\section{\label{sec:proofmainthm}Smooth N-particle systems.}

We present here the proof of Theorem \ref{thm:mainbounded}. Its outline is as follows. First, we consider the  \(N\)-particle system (\ref{eq:multparsysy}) near \(\mathbf{L}^{*}(\theta_{min})\), the choreographic periodic orbit of the uncoupled system, and scale the action coordinates by
\(\delta^{1/2}\).  Then we average, i.e. we make a change of coordinates, after which the angle-dependent terms become
\(O(\delta^{3/4})\), see equation (\ref{eq:npartav}) in Lemma \ref{Lemma1.1}.
Next,  we study the behavior of the Poincar\'e map of the resulting system:  we show that near a line of minima, the Poincar\'e map  of  system (\ref{eq:npartav}) is
\(O(\delta^{3/4})\)-close to the flow map of the Hamiltonian system of a specific form  (\ref{hamtr}) (Lemma \ref{lem:poinc1}). We need to take
\(O(1/\delta^{1/2}\))  iterates of the return map.  In Lemma \ref{Lemma1.2}, we establish that taking \(O(1/\delta^{1/2}\))  iterates of any map which  is  \(O(\delta^{3/4})\) close to the return map of system (\ref{hamtr}) leads to a map which is
\(O(\delta^{1/4})\) close to the time-1 map for the Hamiltonian (\ref{eq:timeoneham}) near the origin. In Lemma \ref{lem:positivekam}, we prove that the Hamiltonian (\ref{eq:timeoneham}) admits a positive measure set of KAM tori. By
Lemma \ref{Lemma1.2}, one infers the results of Theorem  \ref{thm:mainbounded} from the KAM theorem.

The proof strategy, of using Poincar\'e maps (see Lemmas  \ref{lem:poinc1}-\ref{lem:positivekam}) instead of other partial averaging techniques that could be possibly applied here  \cite{Arnold2007CelestialMechanics},   is chosen because  it is also utilized for establishing the KAM-stability in\ the non-smooth case, i.e. in the proofs of Theorems \ref{thm:billiardsystem} and \ref{thm:billiardbox1}.

{\em Proof of Theorem  \ref{thm:mainbounded}}: Recall that  $(q^{(n)},p^{(n)})$ denotes the coordinates and conjugate momenta of the $n$-th particle, $n=1,\dots,N$. We apply the same symplectic transformation that brings the single-particle system to the normal form (\ref{eq:birkhoffsingle}) to each of the $N$ particles, namely, we let
$(q^{(n)},p^{(n)})=(\hat q(I^{(n)}_0,\theta^{(n)},z^{(n)}),\hat p(I^{(n)},\theta^{(n)},z^{(n)}))$.
As for the single particle case, this change of coordinates is smooth and preserves the standard symplectic form. In these coordinates
the \(N\)-particle system (\ref{eq:multparsysy}) takes the form:
\begin{equation}\label{eq:npartactionangle}\begin{array}{l}\displaystyle
H =\sum_{n=1}^N[\omega I^{(n)} +\frac{1}{2}I^{(n)} A I^{(n)}+g(I_0^{(n)},\theta^{(n)},z^{(n)})]+\\ \displaystyle \qquad\qquad
\qquad\qquad + \;\delta\sum_{n\neq m}\ W(\hat q(I_0^{(n)},\theta^{(n)},z^{(n)})-
\hat q(I_0^{(m)},\theta^{(m)},z^{(m)})),\end{array}
\end{equation}
where we denote
$I^{(n)}=(I_0^{(n)}, I_1^{(n)}\!\!=\!\frac{\left(z_1^{(n)}\right)^2}{2}, \dots, I_{d-1}^{(n)}\!=
\!\frac{\left(z_{d-1}^{(n)}\right)^2}{2})$.
We will look at the motion of the particles near the periodic trajectory $L^*$.
Namely, we will write
\begin{equation} \label{eq:periodicorbit}
\hat q(I_0^{(n)},\theta^{(n)},z^{(n)})=q^*(\theta^{(n)}) + O(|I_0^{(n)}|+\|z^{(n)}\|).
\end{equation}
Since the pairwise interaction potential \(W\) is smooth,
\begin{equation} \label{eq:wperiodicorbit}\begin{array}{l}
W(\hat q(I_0^{(n)},\theta^{(n)},z^{(n)})-\hat q(I_0^{(m)},\theta^{(m)},z^{(m)}))=
W(q^*(\theta^{(n)})-q^*(\theta^{(m)}))+\\ \qquad\qquad\qquad\qquad\qquad
+O(|I_0^{(n)}|+\|z^{(n)}\| + |I_0^{(m)}|+\|z^{(m)}\|).\end{array}
\end{equation}

We restrict our attention to the region of the phase space where the actions $I^{(n)}$ are of order $\delta^{1/2}$.
For that, we scale the variables $z^{(n)}$ to $\delta^{1/4}$ and the variables $I_0^{(n)}$ to $\delta^{1/2}$, i.e., we
make a replacement \(z^{(n)}\rightarrow \delta^{1/4} z^{(n)}\), \(I_0^{(n)}\rightarrow \delta^{1/2}I_0^{(n)}\).
With this scaling, we have
$$
W=W(q^*(\theta^{(n)})-q^*(\theta^{(m)}))+O(\delta^{1/4}),
$$
(see (\ref{eq:wperiodicorbit})) and so, by (\ref{eq:birkhoffsingle}), the remainder terms in (\ref{eq:npartactionangle}), \(g\), satisfy
$$g(I_0^{(n)},\theta^{(n)},z^{(n)})=O(\delta^{5/4}).$$
 The motion in the rescaled variables is described
by the rescaled Hamiltonian:
\begin{equation}\label{eq:npartnew}\begin{array}{ll}\displaystyle
 H_{scal} &= \delta^{-1/2}\; H\!\left((\delta^{1/2}I_0^{(n)},\theta^{(n)},\delta^{1/4}z^{(n)})_{_{n=1,\dots,N}}\right)
\\&=\sum_{n=1}^N \omega I^{(n)} + \delta^{1/2} \sum_{n=1}^N \frac{1}{2} I^{(n)} A I^{(n)}+
\delta^{1/2}\sum_{n\neq m}\ W(q^*(\theta^{(n)})-q^*(\theta^{(m)})) +O(\delta^{3/4}). \end{array}
\end{equation}
We emphasize that from now on, till the end of the proof, the variables are  rescaled, so the domain in which we study  \(H_{scal}\) is of order one in all variables. To avoid cumbersome notation we use the same letters for the rescaled and the original variables.  So,  \(O(\delta^{s})\)  stands for  terms which are \(O(\delta^{s})\)  small with all derivatives in a domain which is independent of \(\delta\).

We now average the Hamiltonian with respect to the motion along the periodic orbit.
\begin{lem}\label{Lemma1.1}
There exists a smooth symplectic change of coordinates which brings the Hamiltonian (\ref{eq:npartnew}) to the form
\begin{equation}\label{eq:npartav}\begin{array}{l}\displaystyle
 H =\sum_{n=1}^N \omega I^{(n)} + \delta^{1/2} \sum_{n=1}^N \frac{1}{2} I^{(n)} A I^{(n)}+
\delta^{1/2} U(\theta^{(1)},\dots, \theta^{(N)})+O(\delta^{3/4}), \end{array}
\end{equation}
where the averaged potential $U$ is given by
(\ref{eq:averagedpotu}).
\end{lem}
\begin{proof}
Recall that the (non-averaged) interaction potential
$$\tilde W(\theta^{(1)},\dots,\theta^{(N)})=\sum_{n\neq m}\ W(q^*(\theta^{(n)})-q^*(\theta^{(m)}))$$
is $2\pi$-periodic in each of the variables $\theta^{(n)}$, $n=1,\dots, N$. Therefore, we can write its Fourier expansion:
\begin{equation}\label{eq:wfourier}
\tilde W(\theta^{(1)},\dots,\theta^{(N)})=\ \sum_{n\neq m} \;\sum_{(k_{1},k_{2})\in \mathbb{Z}^2 }
w_{k_1,k_2,n,m} e^{i (k_{1}\theta^{(n)}+k_{2}\theta^{(m)})}.
\end{equation}
The function $\tilde W$ is of class $C^\infty$, so the Fourier coefficients $w$ decay fast
as $k_{1,2}$ grow. In particular, the series
\begin{equation}\label{eq:Psismoothdef}
\Psi (\theta^{(1)},\dots,\theta^{(N)})
=-i\ \; \sum_{n\neq m} \;\sum_{k_{1}+k_{2}\neq 0} \frac{w_{k_1,k_2,n,m}}{\omega_0(k_1+k_2)}
 e^{i (k_{1}\theta^{(n)}+k_{2}\theta^{(m)})}\end{equation}
is absolutely convergent, and the sum is a $C^\infty$ function of $\theta^{(1)},\dots,\theta^{(N)}$.
By construction,
\begin{equation}\label{psiwu}
\omega_0 (\partial_{\theta^{(1)}}\Psi + \ldots +\partial_{\theta^{(N)}}\Psi)  = \tilde W - U,
\end{equation}
where
$$U = \sum_{n\neq m} \; \sum_{k_{1}+k_{2}=0} w_{k_1,k_2,n,m} e^{i(k_{1}\theta^{(n)}+k_{2}\theta^{(m)})}
= \sum_{n\neq m} \; \sum_{k\in \mathbb{Z}}
w_{k,-k,n,m} e^{i k(\theta^{(n)}-\theta^{(m)})}.$$
Substituting (\ref{eq:wfourier}) in (\ref{eq:averagedw}) and integrating over time shows that the above  $U$ is indeed the averaged potential given by (\ref{eq:averagedpotu}).

Now, we perform a symplectic coordinate change
\begin{equation}\label{subst1}
I_{0}^{(n)} \rightarrow  I_{0}^{(n)} - \delta^{1/2}\partial_{\theta^{(n)}}\Psi, \qquad n=1,\dots, N
\end{equation}
(all other variables remain unchanged).
The first term in the Hamiltonian (\ref{eq:npartnew}), $$\sum_{n=1}^N \omega I^{(n)}=\omega_0 (I_0^{(1)}+\dots + I_0^{(N)})+ \sum_{n=1}^N \sum_{j=1}^{d-1} \omega_j I_j^{(n)},$$
after substituting (\ref{subst1}),  produces,  by (\ref{psiwu}), the additional term
$-\delta^{1/2} (\tilde W-U)$. Substituting  (\ref{subst1}) in the $O(\delta^{1/2})$ terms leads only to $O(\delta)$ corrections. Hence, the Hamiltonian takes the required form (\ref{eq:npartav}).\end{proof}

Now, as explained in Section \ref{sec:setup}, we utilize the translation symmetry of the averaged potential
\(U(\theta^{(1)}, \dots, \theta^{(N)})\)
and introduce the collective phase
$\varphi=\frac{1}{N}(\theta^{(1)}+ \dots +\theta^{(N)})$
and coordinates $\psi=(\psi_1,\dots,\psi_{N-1})$ which measure small deviations of the longitudinal motion from the line  (\ref{lineq}) of minima defined by \(\theta_{min}\). The precise definition of the coordinates $\psi$
is as follows.
 We choose an $(N\times N)$ orthogonal matrix $R$ such that its first row is
 $\frac{1}{\sqrt{N}} \;(1, \dots, 1)$, and define
\begin{equation}
(\varphi,\psi)^\top=\frac{1}{\sqrt{N}} R \left( \begin{array}{c} \theta^{(1)}-\theta^{(1)}_{min}\\ \vdots \\
\theta^{(N)}-\theta^{(N)}_{min}\end{array} \right).\label{eq:phipsitrans}
\end{equation}
By construction, the second and further rows of $R$ are all orthogonal to $(1, \dots, 1)$, which implies that $\psi=0$ for every
point of the line (\ref{lineq}). Notice that while \(\varphi\) is an angle (the $O(\delta^{3/4})$ term in (\ref{eq:npartav}) is periodic in  \(\varphi\)),  the variables \(\psi_{n}\) correspond to small deviations from the minimum and we do not define them globally (so they are not angular variables). In the new coordinates the averaged potential $U$ is independent of $\varphi$, and  is given by $\hat U(\psi)$ of  (\ref{uhat}).
Next, we define conjugate momenta $(P, J_1,\dots,J_{N-1})$ corresponding to the variables $(\varphi,\psi)$:
\begin{equation}\label{eq:pjdef}
(P,J)^\top=\sqrt{N} R \left(\begin{array}{c} I_0^{(1)}\\ \vdots \\ I_0^{(N)}\end{array}\right).
\end{equation}
In particular,
$$P= \sum_{n=1}^N I_0^{(n)}.$$
Since the first row of the orthogonal matrix $R$ is
 $\frac{1}{\sqrt{N}} \;(1, \dots, 1)$, it follows that
$$I_0^{(n)}=\frac{1}{\sqrt{N}}(\frac{P}{\sqrt{N}}+\sum _{m=1}^{N-1}R_{m+1,n}J_m)).$$
This transformation $(\theta_0^{(1)}, \dots, \theta_0^{(N)},I_0^{(1)}, \dots, I_0^{(N)}) \mapsto (\varphi,\psi,P,J)$ is defined by the generating function $\sqrt{N}(\varphi,\psi) R  (I_0^{(1)}, \dots, I_0^{(N)})^\top$,
 so it is symplectic.

Therefore, we can perform this transformation in the Hamiltonian function (\ref{eq:npartav}) directly. As a result, we obtain the new Hamiltonian (recall that the matrix $A$ in (\ref{eq:npartav}) is given by (\ref{amatr})):
\begin{equation}
\begin{array}{l}
H=\omega_{0}P +\sum_{n=1}^N\sum_{j=1}^{d-1}\omega_j I^{(n)}_j +\\
\\ \quad + \delta^{1/2} (\; a \frac{P^2+J^2}{2N}+
\sum_{n=1}^{N} \frac{1}{\sqrt{N}}(\frac{P}{\sqrt{N}}+\sum _{m=1}^{N-1}R_{m+1,n}J_m) \sum_{j=1}^{d-1} b_{j} I^{(n)}_j+
\frac{1}{2} \sum_{n=1}^{N} \sum_{i,j=1}^{d-1} \hat a_{ij} I^{(n)}_jI^{(n)}_i)+\\ \\  \qquad
+ \delta^{1/2}\hat U(\psi)+O(\delta^{3/4}).
\end{array}\label{eq:hamlitafterphipsi}
\end{equation}
Denoting  $\hat \omega= (\omega_1,\dots,\omega_{d-1})$, $\hat I^{(n)}= (I^{(n)}_1,\dots,I^{(n)}_{d-1})^{\top}$
(recall that $\displaystyle I_j^{(n)}=\frac{(z_j^{(n)})^2}{2}= \frac{(x_j^{(n)})^2+(y_j^{(n)})^2}{2}$), and
$R_n=(R_{2,n}, \dots, R_{N,n})^\top$,  the Hamiltonian recasts as
\begin{equation}\label{hprep}
\!\!\!\!\!\!\begin{array}{l} \displaystyle
H(P,J,\varphi,\psi,\{z_{j}^{(n)}\}_{j=1,..,d-1,n=1,..N})\\=\omega_{0}P +\sum_{n=1}^N \hat\omega \hat I^{(n)}\! +
\delta^{1/2} \; \left(a\; \frac{P^2+J^2}{2N}+
\frac{1}{\sqrt{N}} \sum_{n=1}^{N} (\frac{P}{\sqrt{N}}+R_n^\top J) b \hat I^{(n)}+ \sum_{n=1}^{N}
\frac{1}{2}\; \hat I^{(n)} \hat A \hat I^{(n)}\right)+\\ \displaystyle \qquad\qquad\qquad
+ \delta^{1/2}\hat U(\psi)+O(\delta^{3/4}).
\end{array}
\end{equation}
Note that $\dot\varphi = \partial_P H = \omega_0 +O(\delta^{1/2}) >0$ in this system. Therefore, the Poincar\'e return map
from the hypersurface \(\varphi=0\) to \(\varphi=2\pi\) (i.e., to itself) is well-defined.
\begin{lem}\label{lem:poinc1}
The Poincar\'e return map for system (\ref{hprep}) restricted to the energy level $H=h$ is $O(\delta^{3/4})$-close to
the time-\(\frac{2\pi}{\omega_0}\) map of the system
\begin{equation} \label{hamtr}\begin{array}{l}
\dot J =-\delta^{1/2}\;\partial_\psi \hat U(\psi),  \qquad
\dot \psi =\delta^{1/2}\;(\frac{a}{N} J + \frac{1}{\sqrt{N}} \sum_{n=1}^N (b \hat I^{(n)})R_n),\\ \\
\dot x^{(n)}_j = \Omega_j^{(n)} y^{(n)}_j,\qquad
\dot y^{(n)}_j = - \Omega_j^{(n)} x^{(n)}_j\qquad (j=1,\dots, d-1; \;n=1, \dots, N),\\
\end{array}
\end{equation}
where
\begin{equation}\label{lastom}
\begin{array}{l}\Omega^{(n)}_j= \frac{h}{\omega_0 N} (b_j - \frac{a}{\omega_0}\omega_j)  +
\frac{1}{\omega_0\sqrt{N}} (R_n^\top J)b_j + \hat A_j  \hat I^{(n)}-\\  \\ \qquad\qquad\qquad\qquad -\;
 \frac{1}{\omega_0 N}\sum_{l=1}^{d-1}\sum_{m=1}^N(\omega_j b_l +
b_j\omega_l - \frac{a}{\omega_0}\omega_l\omega_j) I^{(m)}_l;
\end{array}
\end{equation}
here
$b_j$ stands for the $j$-th element of the vector $b$, and $\hat A_j$ for the $j$-th row of the matrix $\hat A$.
\end{lem}
\begin{proof}
The system of differential equations defined by Hamiltonian (\ref{hprep}) is
\begin{equation} \label{hamflow}\begin{array}{l}
\dot P =O(\delta^{3/4}),  \\
\dot \varphi =\omega_{0} +\delta^{1/2}\; (\frac{a}{N}P+ \frac{1}{N}\sum_{n=1}^N b \hat I^{(n)})+O(\delta^{3/4}),\\ \\
\dot J =-\delta^{1/2}\; \partial_\psi \hat U(\psi)+O(\delta^{3/4}),  \\
\dot \psi =\delta^{1/2}\; (\frac{a}{N} J + \frac{1}{\sqrt{N}} \sum_{n=1}^N (b \hat I^{(n)}) R_n) +O(\delta^{3/4}),\\ \\
\dot x^{(n)}_j =  \partial_{y^{(n)}_j} H = \tilde \Omega_j^{(n)} y^{(n)}_j +O(\delta^{3/4}),\\
\dot y^{(n)}_j = -\partial_{x^{(n)}_j} H = - \tilde \Omega_j^{(n)} x^{(n)}_j +O(\delta^{3/4}) \qquad (j=1,\dots, d-1; \;n=1, \dots, N),\\
\end{array}
\end{equation}
where
$\tilde \Omega_j^{(n)}=\frac{\partial H}{\partial I_{j}^{(n)}}=\omega_j+\delta^{1/2}(\frac{1}{\sqrt{N}} (\frac{P}{\sqrt{N}}+R_n^\top J) b_j + \hat A_j \hat I^{(n)})$.
 Applying the inverse function theorem to (\ref{hprep}), we can express \(P\) as a function of all other variables on the energy level \(H=h \) $$P =\frac{1}{\omega_0}(h-\hat \omega\sum_{m=1}^N  \hat I^{(m)}) +O(\delta^{1/2}).$$
We substitute this expression into (\ref{hamflow}) and choose $\varphi/\omega_0$ as the new time variable (i.e., we divide
$\dot J$, $\dot\psi$, $\dot x^{(n)}_j$, and $\dot y^{(n)}_j$ to $\dot\varphi/\omega_0$). One can see that the result is $O(\delta^{3/4})$-close to system (\ref{hamtr}). Since the sought Poincar\'e map is the time-$\frac{2\pi}{\omega_0}$ map in the new time, we  immediately obtain the lemma.
\end{proof}
\begin{lem}\label{Lemma1.2}
Let $K=\left \lfloor\frac{\omega_0}{2\pi\delta^{1/2}}\right\rfloor$ and $\nu=(\nu_1,\dots,\nu_{d-1})$, where $\nu_j=2\pi\left\{K\frac{\omega_j}{\omega_0}\!\right\}$, $j=1,\dots,d-1$.
Then, for all small $h$, the $K$-th iteration of any map which is $O(\delta^{3/4})$-close to the time-\(\frac{2\pi}{\omega_0}\) map of system  (\ref{hamtr}), is $O(\delta^{1/4})$-close to
the time-1 map of the flow defined by the Hamiltonian
\begin{equation}\label{eq:timeoneham}
\begin{array}{l}\displaystyle
\bar H =\frac{a}{2N}J^2+\hat U(\psi)+ \frac{1}{\sqrt{N}} \sum_{n=1}^N  (b \hat I^{(n)}) (R_n^\top J)
+  (\nu+\frac{h}{\omega_0 N}(b - \frac{a }{\omega_0}\hat \omega)) \sum_{n=1}^N \hat I^{(n)} +\\ \\
\qquad\qquad\qquad\qquad\qquad \displaystyle +\;
\frac{1}{2} \sum_{n=1}^N \hat I^{(n)} \hat A \hat I^{(n)} -
\frac{1}{2} \sum_{n=1}^N\sum_{m=1}^N \hat I^{(n)} S \hat I^{(m)}, \end{array}
\end{equation}where the symmetric matrix $S$ is given by
\begin{equation}\label{smatrdef}
S=\frac{1}{\omega_0 N}(\hat \omega^{\top} b +b^{\top} \hat\omega\ - \frac{a}{\omega_0}\hat\omega^{\top} \hat\omega),\end{equation}
namely
$$S_{jl}=\frac{1}{\omega_0 N}(\omega_j b_l + \omega_l b_j - \frac{a}{\omega_0}\omega_l\omega_j),
\qquad j,l=1,\dots,d-1.$$
\end{lem}
\begin{proof}
Denote $(x^{(n)}_j, y^{(n)}_j)=
\sqrt{2 I^{(n)}_j} (\cos \phi^{(n)}_j,\sin \phi^{(n)}_j )$. For system (\ref{hamtr}), the actions
\(I^{(n)}_j=\frac{(x^{(n)}_j)^2+(y^{(n)}_j)^2}{2}\)
are constants of motion. So, for this system, the time-$\frac{2\pi}{\omega_0}$ map
$(J, \psi, I^{(n)}_j, \phi^{(n)}_j) \mapsto (\bar J, \bar \psi, \bar I^{(n)}_j, \bar \phi^{(n)}_j)$ ($j=1,\dots,d-1; n=1,\dots, N$)
is given by
\begin{equation}
\begin{array}{l}
\bar J = J - \frac{2\pi}{\omega_0} \delta^{1/2} \;\partial_\psi \hat U(\psi)+O(\delta),  \\
\bar \psi =  \psi +\frac{2\pi}{\omega_0} \delta^{1/2} \;
(\frac{a}{N} J + \frac{1}{\sqrt{N}} \sum_{n=1}^N (b \hat I^{(n)}) R_n)+O(\delta), \\ \\
\bar I^{(n)}_j = I^{(n)}_j,\\
\bar \phi^{(n)}_j = \phi_j^{(n)} +  \frac{2\pi}{\omega_0} \omega_j +
 \frac{2\pi}{\omega_0} \delta^{1/2}\; \Omega^{(n)}_j +O(\delta).
\end{array}\label{eq:mapijpsi}
\end{equation}

This map is \(O(\delta^{1/2})\)-close to an isometry (a rigid rotation of the variables $(x^{(n)}_j,y^{(n)}_j)$).  When iterating such maps,  a small error added at each iteration
will propagate linearly as long as the number of iterations is of order \(1/\delta^{1/2}\). Since
$K=O\left(\frac{1}{\delta^{1/2}}\right)$, it follows that the $K$-th iteration of any map which is $O(\delta^{3/4})$-close
to the time-$\frac{2\pi}{\omega_0}$ map of the flow of (\ref{hamtr}) is $O(K\delta^{3/4})$-close
to the time-$K\frac{2\pi}{\omega_0}$ map of (\ref{hamtr}), i.e., $O(\delta^{1/4})$-close to the $K$-th iteration of (\ref{eq:mapijpsi}).

Thus, to prove the lemma, it is enough to show that the $K$-th iteration of (\ref{eq:mapijpsi}) is $O(\delta^{1/4})$-close to
the time-$1$ map of (\ref{eq:timeoneham}) (in fact we prove that it is $O(\delta^{1/2})$-close). We do this by moving to a rotating coordinate frame. Denote
$$\lambda_j=\frac{1}{K}\left\lfloor K\frac{\omega_j}{\omega_0}\right\rfloor, \qquad \nu_j=2\pi K (\frac{\omega_j}{\omega_0}-\lambda_j)$$
(note that these are the same $\nu_j$ as in the statement of the lemma). The rotating coordinate frame corresponds to the new variables:
$$\alpha^{(n)}_j (k)= \phi^{(n)}_j (k)- 2\pi k  \lambda_j$$
 i.e., at each iteration of the map  (\ref{eq:mapijpsi}) we subtract  $2\pi \lambda_j$
from $\phi^{(n)}_j$. This brings the map (\ref{eq:mapijpsi}) to the form:
\begin{equation}\label{pmrot}
\begin{array}{l}
\bar J = J - \frac{1}{K} \;\partial_\psi \hat U(\psi)+O(\delta),  \\
\bar \psi =  \psi +\frac{1}{K} \;
(\frac{a}{N} J + \frac{1}{\sqrt{N}} \sum_{n=1}^N (b \hat I^{(n)}) R_n)+O(\delta), \\ \\
\bar I^{(n)}_j = I^{(n)}_j,\\
\bar \alpha^{(n)}_j = \alpha_j^{(n)} +  \frac{1}{K}\;(\nu_j+\Omega^{(n)}_j) +O(\delta)\end{array}
\end{equation}
(we use here that $\frac{1}{K}= \frac{2\pi}{\omega_0} \delta^{1/2}+ O(\delta)$). This is a near-identity map which is
$O(\delta)$-close to the time-$\frac{1}{K}$ map of the system
\begin{equation}\label{eqdh}
\begin{array}{l}
\dot J = - \partial_\psi \hat U(\psi),  \\
\dot \psi =  \frac{a}{N} J + \frac{1}{\sqrt{N}} \sum_{n=1}^N (b \hat I^{(n)}) R_n,\\
\dot I^{(n)}_j = 0,\\
\dot \alpha^{(n)}_j =\nu_j+ \Omega^{(n)}_j. \end{array}
\end{equation}
Therefore, the $K$-th iteration of the map (\ref{pmrot}) is $O(K \delta)$- (i.e.,  $O(\delta^{1/2})$-) close to the time-$1$ map
of this system. Returning to the non-rotating phases $\phi^{(n)}_j$ does not change the $K$-th iteration
of the map: since $K\lambda_j$ are integers by construction, $\phi^{(n)}_j$ coincides, after the $K$-th iteration,
with $\alpha^{(n)}_j$ modulo $2\pi$, for all $j$.

It remains to note that system (\ref{eqdh}) with \(\Omega^{(n)}_j\) defined by (\ref{lastom})  indeed corresponds to the Hamiltonian (\ref{eq:timeoneham}) (where
$\alpha^{(n)}_j$ are the angular variables conjugate to the actions $I_j^{(n)}$).
\end{proof}

We show next that the twist Assumption SP2 and the KAM assumption IP1 imply:
\begin{lem}\label{lem:positivekam} At \(h=0 \), the system (\ref{eqdh}) (corresponding to the Hamiltonian (\ref{eq:timeoneham})) has a positive measure set of KAM tori near its equilibrium at the origin.
\end{lem}
\begin{proof}
First, we make a change of coordinates which decouples the \((J,\psi)\) and
$(\hat  I,\alpha)$ degrees of freedom in (\ref{eq:timeoneham}). We achieve this goal by replacing
\begin{equation}\label{jjr}
J\rightarrow J - \frac{\sqrt{N}}{a} \sum_{n=1}^N (b\hat  I^{(n)}) R_n,
\end{equation}
where $R_n=(R_{2,n}, \dots, R_{N,n})$. Then, the right-hand side of the equation for $\dot\psi$ will be independent of $\hat I$ (i.e., independent of $(x,y)$). In order to make this a symplectic transformation,
we write it as
$$J_j= \tilde J_j - \frac{\sqrt{N}}{a} \sum_{n=1}^N \sum_{l=1}^{d-1} b_l  I^{(n)}_l R_{j+1,n}, \quad j=1,\dots, N-1$$
and also transform the $\alpha$-variables:
$$\alpha^{(n)}_l= \tilde \alpha^{(n)}_l +b_l \frac{\sqrt{N}}{a} \sum_{j=1}^{N-1}  \psi_j R_{j+1,n}, \quad l=1,\dots, d-1, \;\; n=1,\dots N.$$
The simplecticity of the transformation
\((J,\psi,\hat I,\alpha)\) to \((\tilde J,\tilde\psi=\psi,\tilde I=\hat I,\tilde\alpha)\) follows because it is defined by the generating function  \(\sum_{j=1}^{N-1}( (J_j + \frac{\sqrt{N}}{a} \sum_{n=1}^N \sum_{l=1}^{d-1} b_l   I^{(n)}_l R_{j+1,n})\tilde\psi_j+\sum_{n=1}^N I_j^{(n)}\tilde \alpha_j^{(n)}\)).

Performing this change of variables directly in the Hamiltonian (\ref{eq:timeoneham}) at $h=0$,  the new Hamiltonian is (omitting the tilde signs)
\begin{equation}\label{eq:sumofHIandUpsi}
\begin{array}{l}\displaystyle
H =\frac{a}{2N}J^2+\hat U(\psi) - \frac{1}{2a} \left(\sum_{n=1}^N  (b \hat I^{(n)}) R_n\right)^2
+  \nu \sum_{n=1}^N \hat I^{(n)} +\\ \\
\qquad\qquad\qquad\qquad\qquad \displaystyle +\;
\frac{1}{2} \sum_{n=1}^N \hat I^{(n)} \hat A \hat I^{(n)} -
\frac{1}{2} \sum_{n=1}^N\sum_{m=1}^N \hat I^{(n)} S \hat I^{(m)}. \end{array}
\end{equation}
This is the sum of the Hamiltonian (\ref{eq:hamav})
 that depends only on $J$ and $\psi$ and describes oscillations around the equilibrium at $(J=0,\psi=0)$,
and the Hamiltonian
\begin{equation}\label{eq:HsubI}
H_I = \nu \sum_{n=1}^N \hat I^{(n)} - \frac{1}{2a} \left(\sum_{n=1}^N  (b \hat I^{(n)}) R_n\right)^2
+
\frac{1}{2} \sum_{n=1}^N \hat I^{(n)} \hat A \hat I^{(n)} -
\frac{1}{2} \sum_{n=1}^N\sum_{m=1}^N \hat I^{(n)} S \hat I^{(m)},
\end{equation}which depends only on $\hat I$ variables and describes rotations of the phases $\alpha$.
By the KAM Assumption IP1,   the Hamiltonian (\ref{eq:hamav}) has a positive measure set of KAM tori near the zero equilibrium. This means we only need to check that the Hamiltonian $H_I$ also has a positive set of KAM tori near the origin, which is proved next, in Lemma \ref{lem:HIhastwist}.
\end{proof}

\begin{lem}\label{lem:HIhastwist} The Hamiltonian \(H_I\) of (\ref{eq:HsubI}) satisfies the twist condition.
\end{lem}
\begin{proof}
This condition is the requirement that the matrix of second derivatives
of $H_I$ with respect to $\hat I^{(n)}_j$ is non-degenerate, i.e., the quadratic form
\begin{equation}
\; -\frac{1}{2a} \left(\sum_{n=1}^N  (b \hat I^{(n)}) R_n\right)^2
+
\frac{1}{2} \sum_{n=1}^N \hat I^{(n)} \hat A \hat I^{(n)} -
\frac{1}{2} \sum_{n=1}^N\sum_{m=1}^N \hat I^{(n)} S \hat I^{(m)}
\label{eq:quadraticform1}
\end{equation}
is non-degenerate.
This is equivalent to the non-degeneracy of the quadratic form
$$\frac{a}{2N} J^2 -
\frac{1}{2a} \left(\sum_{n=1}^N  (b \hat I^{(n)}) R_n\right)^2
+
\frac{1}{2} \sum_{n=1}^N \hat I^{(n)} \hat A \hat I^{(n)} -
\frac{1}{2} \sum_{n=1}^N\sum_{m=1}^N \hat I^{(n)} S \hat I^{(m)},$$
where we added the dummy variables $J=(J_1,\dots, J_{N-1})^{\top}$. Replacing
$$J=\tilde  J + \frac{\sqrt{N}}{a} \sum_{n=1}^N (b \hat I^{(n)}) R_n$$
(note that this is the inverse of (\ref{jjr})), and omitting the tilde sign, we obtain the quadratic form
$$\frac{a}{2N} J^2 + \frac{1}{\sqrt{N}} \sum_{n=1}^N  (b \hat I^{(n)}) (R_n^\top J)
+
\frac{1}{2} \sum_{n=1}^N \hat I^{(n)} \hat A \hat I^{(n)} -
\frac{1}{2} \sum_{n=1}^N\sum_{m=1}^N \hat I^{(n)} S \hat I^{(m)};$$
Proving its non-degeneracy amounts to
showing the non-vanishing of the determinant of the following matrix
\begin{equation}\label{matrm}
M=\begin{pmatrix}\begin{pmatrix}\frac{a}{N} & \ldots & 0 \\
0 & \ddots\  &0 \\
0 & \ldots &  \frac{a}{N}\\
\end{pmatrix} & \frac{1}{\sqrt{N}}R_{1} b& \ldots & \ldots & \ldots & \!\frac{1}{\sqrt{N}}R_{N} b \\
 \frac{1}{\sqrt{N}}b^\top R_1^\top & \hat A-S & -S & \ldots & \ldots & \!\!-S\\
 \vdots& -S & \hat A-S & -S & \ldots &\!\!-S \\
\vdots & \vdots\ & \ldots & \ddots & \ldots & \!\vdots \\ \vdots & \!\vdots & &  &\; \ddots\\
\frac{1}{\sqrt{N}}b^\top R_N^\top & -S & \ldots & \ldots & \!\!-S & \!\!\!\!\hat A-S
\end{pmatrix}.
\end{equation}
Let us show that
\begin{equation}\label{eq:detM}
\det M =- \; \frac{1}{\omega_0^2} \left(\frac{1}{N}\det  A\right)^{N-1}\; \det  A_\omega.
\end{equation}
so, by the single particle twist Assumption, SP3,
$$\det M\neq0,$$
which will prove the lemma and the theorem.

Recall that \(R_n\)'s in the expression (\ref{matrm}) are the columns of an orthogonal matrix \(R\) without its first row. The first row of $R$ equals to
$\frac{1}{\sqrt{N}} \;(1, \dots, 1)$. Hence, it follows from $R^\top R=id$ that
\begin{equation}\label{rmn}
R_n^\top R_m = \left\{\begin{array}{cl} -1/N & \;\mbox{if} \;\; m\neq n,\\ 1-1/N & \;\mbox{if} \;\; m= n,\end{array}\right.
\end{equation}
and since \(R R^\top =id\)
\begin{equation}\label{rsm}
R_1+ \ldots + R_N=0.
\end{equation}
Subtract the last column in formula (\ref{matrm}) from each other column, except for the first one.
The resulting matrix
$$\begin{pmatrix}\begin{pmatrix}\frac{a}{N} & \ldots & 0 \\
0 & \ddots\  &0 \\
0 & \ldots &  \frac{a}{N}\\
\end{pmatrix} & \frac{1}{\sqrt{N}}(R_{1}-R_N) b& \!\frac{1}{\sqrt{N}}(R_2-R_N)b &\!\!\!\ldots & \ldots & \!\frac{1}{\sqrt{N}}R_{N} b \\
 \frac{1}{\sqrt{N}}b^\top R_{1}^\top & \qquad\qquad\;\;\hat A & 0 & \!\!\!\ldots& \ldots & \!\!-S\\
 \vdots& \qquad\qquad\;\;0 & \hat A& \!\!\!\!0 & \ldots &\!\!-S \\
\vdots &\qquad\qquad\;\; \vdots\ & \ldots & \!\!\!\!\!\!\ddots & \ldots & \!\vdots \\
\vdots &\qquad\qquad\;\; 0 & \ldots &  \!\!\!\!0&\; \hat A & \!-S\\
\frac{1}{\sqrt{N}}b^\top R_N^\top & \qquad\qquad-\hat A & \ldots & \!\!\!\ldots & \!\!-\hat A & \hat A-S
\end{pmatrix}$$
has the same determinant as $M$. We again get a matrix with the same determinant when, in the last formula,
we add all
rows, except for the first one, to the last row. The result is
$$\begin{pmatrix}\begin{pmatrix}\frac{a}{N} & \ldots & 0 \\
0 & \ddots\  &0 \\
0 & \ldots &  \frac{a}{N}\\
\end{pmatrix} & \frac{1}{\sqrt{N}}(R_1-R_N) b& \!\frac{1}{\sqrt{N}}(R_2-R_N)b &\!\!\!\ldots & \ldots & \!\frac{1}{\sqrt{N}}R_{N} b \\
 \frac{1}{\sqrt{N}}b^\top R_{1}^\top & \qquad\qquad\;\;\hat A & 0 & \!\!\!\ldots& \ldots & \!\!-S\\
 \vdots& \qquad\qquad\;\;0 & \hat A& \!\!\!\!0 & \ldots &\!\!-S \\
\vdots & \qquad\qquad\;\;\vdots\ & \ldots & \!\!\!\!\!\!\ddots & \ldots & \!\vdots
\\  \frac{1}{\sqrt{N}}b^\top R_{N-1}^\top & \qquad\qquad\;\;0 & \ldots &  \!\!\!\!0& \!\!\hat A & \!-S\\
0 & \qquad\qquad\;\;0 & \ldots & \!\!\!\ldots & \!\!0 & \;\hat A- NS
\end{pmatrix}.$$
Note that the utmost left bottom block in this matrix equals to $\frac{1}{\sqrt{N}}b^\top (R_1+\ldots+R_N)^\top$ and
is zero by (\ref{rsm}).

Next, for each $n=1, \dots, N-1$ we multiply the first row in this formula by $\frac{\sqrt{N}}{a}b^\top R_n^\top$ and subtract the result from the $(n+1)$-th row
(so we do not change the first and the last rows). By (\ref{rmn}), this gives us the block-triangular matrix
$$\begin{pmatrix}\begin{pmatrix}\frac{a}{N} & \ldots & 0 \\
0 & \ddots\  &0 \\
0 & \ldots &  \frac{a}{N}\\
\end{pmatrix} & \frac{1}{\sqrt{N}}(R_1-R_N) b& \!\frac{1}{\sqrt{N}}(R_2-R_N)b &\!\!\!\ldots & \ldots & \!\frac{1}{\sqrt{N}}R_{N} b \\
0 & \qquad\qquad\;\;\hat A - \frac{1}{a} b^\top b & 0 & \!\!\!\ldots& \ldots & \!\!-S+\frac{1}{aN} b^\top b\\
 \vdots& \qquad\qquad\;\;0 & \hat A-\frac{1}{a} b^\top b& \!\!\!\!0 & \ldots &\!\!-S+\frac{1}{aN} b^\top b \\
\vdots & \qquad\qquad\;\;\vdots\ & \ldots & \!\!\!\!\!\!\ddots & \ldots & \!\vdots
\\  \vdots & \qquad\qquad\;\;0 & \ldots &  \!\!\!\!0& \hat A- \frac{1}{a} b^\top b & \!-S+\frac{1}{aN} b^\top b\\
0 & \qquad\qquad\;\;0 & \ldots & \!\!\!\ldots & \!\!0 & \;\hat A- NS
\end{pmatrix}.$$
By construction, its determinant equals to the determinant of $M$, which gives
$$\det M = \left(\frac{a}{N}\right)^{N-1} \; \det (\hat A - \frac{1}{a} b^\top b)^{N-1}\; \det(\hat A- NS).$$
Now, formula (\ref{eq:detM}) follows, since
$$\det A=\det \begin{pmatrix}a & b \\
b^\top  & \hat A \\
\end{pmatrix}=\det \begin{pmatrix}a & 0 \\
b^\top  & \hat A-\frac{1}{a}b^\top b \\
\end{pmatrix} \begin{pmatrix}1 & \frac{1}{a}b \\
0  &  I \\
\end{pmatrix}=a \det (\hat A - \frac{1}{a} b^\top b)  ,$$
 and, by (\ref{smatrdef}),
 $$-\omega_0^2 \det (\hat A - NS) = -\omega_0^2 \det(\hat A -  \frac{1}{\omega_0}(\hat \omega^{T} b +b^{T} \hat\omega\ - \frac{a}{\omega_0}\hat\omega^{T} \hat\omega))=\det\left(\!\!\begin{array}{ccc}
 -\frac{\omega_0^2}{a} & \omega_0 & \hat\omega \\  0 & a &  b \\ 0 &  0 & \hat A+\frac{a}{\omega_0^2}(\hat\omega^\top-\frac{\omega_0}{a}b^\top)\hat\omega-\frac{\hat\omega^\top b}{\omega_0} \end{array}\!\!\right)$$
$$=\det\left(\begin{array}{ccc} -\frac{\omega_0^2}{a} & \omega_0 & \hat\omega \\  0 & a &  b \\ 0 &  b^\top+\frac{a}{\omega_0}(\hat\omega^\top-\frac{\omega_{0}}{a}b^\top) & \hat A+\frac{a}{\omega_0^2}(\hat\omega^\top-\frac{\omega_{0}}{a}b^\top)\hat\omega \end{array}\right)=
\det\left(\begin{array}{ccc} -\frac{\omega_0^2}{a} & \omega_0 & \hat\omega \\  0 & a &  b \\ \hat\omega^\top-\frac{\omega_{0}}{a}b^\top &  b^\top & \hat A \end{array}\right)=$$
$$= \det\left(\begin{array}{ccc} 0 & \omega_0 & \hat\omega \\  \omega_0 & a &  b \\ \hat\omega^\top &  b^\top & \hat A \end{array}\right)
=\det  A_\omega.$$

\end{proof}

This completes the proof of Lemma \ref{lem:positivekam}, showing that the Hamiltonian system (\ref{eq:timeoneham}) has a positive measure set of KAM tori.

Now, the claim of Theorem \ref{thm:mainbounded} follows. Indeed, as KAM-tori persist at small perturbations, Lemmas \ref{lem:poinc1} and \ref{Lemma1.2} imply that system (\ref{hprep}) also has a positive measure set of KAM-tori on every energy level $H=h$ with small $h$. Since the system defined by (\ref{hprep}) is smoothly conjugate to the original  \(N\)-particle  system (\ref{eq:npartactionangle})  near \(\mathbf{L}^{*}(\theta_{min})\), with a scaling factor \(O(\delta^{1/4})\), the KAM theory implies that there are quasi-periodic orbits that are \(O(\delta^{1/4})\)-close to \(\mathbf{L}^{*}(\theta_{min})\).
By (\ref{hamflow}),  on such tori, the return time to the cross-section \(\varphi=0\) is  \(O(\delta^{1/2})\)-close to
\(\frac{2\pi}{\omega_{0}}\). It follows that the averaged return time,  \(\frac{2\pi}{\bar \omega}\) is also
\(O(\delta^{1/2})\)-close to \(\frac{2\pi}{\omega_{0}}\),  completing the proof of Theorem \ref{thm:mainbounded}.

\section{\label{sec:proofbilliards}Mutually repelling particles in a container }
The motion of mutually repelling particles confined in a bounded domain $D$ is described by the Hamiltonian  (\ref{eq:nparticlesbillnorescale}).  In the limit of high energy per particle (i.e. for \(\delta=\frac{1}{2h}\rightarrow0)\), one can view the system as a set of weakly interacting particles in a steep billiard-like potential as described by (\ref{eq:multparsysybill0}), with the single-particle dynamics governed by  (\ref{eq:singlepartbiliarddel}).

The singularity of the single-particle system at $\delta=0$ makes the proof of Theorem \ref{thm:billiardsystem} more involved
than for Theorem \ref{thm:mainbounded}. The outline of the proof is as follows. In Section \ref{sec:singlepartbil}, we study  the single-particle system  (\ref{eq:singlepartbiliarddel}) in the small \(\delta\) limit.  Recall that this system  has an elliptic periodic orbit for \(\delta=0\), and hence, by assumptions Box1 and Box2, this orbit persists also for sufficiently small \(\delta\). By studying the singular behavior near impacts, we construct a transformation to action-angle coordinates  near this periodic orbit, with a singularity of the transformation near the billiard boundary. We prove that the Hamiltonian expressed in these coordinates has a smooth limit at \(\delta=0\)  (see Section \ref{sec:periodofL}). We then show that for all sufficiently small \(\delta\geqslant0\) this periodic orbit satisfies Assumptions SP1-SP3 (Lemma \ref{lem:billasum1to3}). In Section \ref{sec:billiardmultipart}, we study the multi-particle dynamics. Here, using the analysis of
Section \ref{sec:singlepartbil}, we show that  the  return map to a cross-section at which all particles are away from the billiards' boundary is not singular at $\delta=0$ and is $o(\delta^{1/2})$-close to the return map of a certain truncated system
(system (\ref{eq:nparttrunbillnew}), see Lemma \ref{lem:poincaremapclosebilnew}). We then analyze the return map of the truncated system by averaging  (Lemma \ref{Lemmaavbil}). Due to the singular nature of the impacts, we use the pair-wise structure of the interaction terms and not the Fourier expansion which was used in Lemma \ref{Lemma1.1}. We then establish that the return map of the truncated system  (\ref{eq:nparttrunbillnew}) is close to that of the truncated averaged system (\ref{eq:nparttrunbill}). This system is of the same form as the truncated averaged system considered in Section \ref{sec:proofmainthm} (cf. (\ref{eq:npartav})), with the only difference that the coefficients now depend on \(\delta\). Since the dependence of the coefficients is non-singular (continuous)  for all \(\delta\geqslant0\), we can conclude the proof as in Theorem \ref{thm:mainbounded}.

\subsection{\label{sec:singlepartbil} A single particle at high energy.}
Here we study the solutions of   (\ref{eq:singlepartbiliarddel})  for small \(\delta\).
Away from the boundary, \(\delta V(q)\) is uniformly small, so the trajectories follow closely the corresponding billiard trajectories. Near the billiard boundary, a more precise analysis is needed.

\subsubsection{The boundary layer dynamics}
By the boundary layer, we mean a sufficiently small neighborhood of the boundary of $D$ where we have
\(V(q)=1/Q(q)^{\alpha} \), with \(Q(q)\) measuring the distance to the boundary of $D$ (see Assumption BD1).

\begin{lem}\label{lem:qnearimpact} Take a small neighborhood of a regular point
$M\in\partial D$.
Then, one can define functions $\tilde q$, $\tilde p$, $q_{impact}$, $t_{in}$
such that the following holds. Given any initial condition \((q_0,p_0) \) which is close to a regular
impact (i.e., \(q_0\) is in the small neighborhood of $M$ but is bounded away from \(\partial D\),
and \(p_{0}\cdot\nabla Q(M)\) is negative and bounded away from zero), the trajectory
$(q(t,q_0,p_0), p(t,q_0,p_0))$
of \((q_0,p_0)\) can be written, in the boundary layer, in the following form:
\begin{equation}\label{eq:scaledq}
\begin{array}{ll}
q(t,q_0,p_0)&=q_{impact}+\delta^{1/\alpha} \tilde q(t_{s},q_0,p_0;\delta),\\
p(t,q_0,p_0)&=\tilde p(t_{s},q_0,p_0;\delta),\\
\end{array}
\end{equation}
where \(t_{s}\) denotes the rescaled time
\begin{equation}\label{eq:deftimpact}
t_{s}=\frac{t-t_{in}(q_0,p_0;\delta)}{\delta^{1/\alpha}}.
\end{equation}
The functions \(\tilde p,\tilde q\) depend smoothly on \((t_s,q_{0},p_0)\) and, along with the derivatives, depend continuously on $\delta$ for all \(\delta\geqslant0\): the function $\tilde p$, along with the derivatives, is
uniformly bounded and uniformly continuous for all \(t_{s}\) and \(\delta\geqslant0\), and
the function $\tilde q$ is given by
\begin{equation}\label{formtq}
\tilde q(t_{s},q_0,p_0;\delta)=\tilde q(0,q_0,p_0;\delta)+\int^{t_s}_0\tilde p(u,q_0,p_0;\delta)du,
\end{equation}
where $\tilde q(0,q_0,p_0;\delta)$ depends continuously on $\delta$, along with the derivatives, for all
$\delta\geqslant 0$.
The function \(q_{impact}(q_0,p_0)\) is smooth and independent of $\delta$, and is determined by the billiard impact event:
\begin{equation}\label{impbev}
q_{impact}(q_0,p_0)\in \partial D.
\end{equation}
The function $t_{in}(q_0,p_0;\delta)$ depends smoothly on \((q_{0},p_0)\)
and, along with the derivatives, depends continuously on \(\delta\) for all
\(\delta\geqslant0\). In the limit
\(\delta\rightarrow0\), the trajectory approaches the billiard trajectory, i.e.,
\begin{equation}\label{bimp}
q_0 + t_{in} p_0 = q_{impact} \;\mbox{ for }\; \delta=0,
\end{equation}
and
\begin{equation}\label{impp}
\lim_{t_s\rightarrow -\infty}\tilde p(t_s,q_0,p_0;0)=p_0, \qquad
\lim_{t_s\rightarrow +\infty}\tilde p(t_s,q_0,p_0;0)=p_0 - 2 (p_0\cdot \vec{n}) \vec{n}
\end{equation}
(the billiard reflection law) where $\vec{n}$ is the outer normal to $\partial D$ at the impact point
$q_{impact}(q_0,p_0;0)$. \end{lem}
\begin{proof} We follow the strategy of \cite{RapRKT07apprx,turaev1998elliptic}. Put the origin of the coordinate system at the point $M$,
and let $q_{\parallel}$ denote coordinates corresponding to directions tangent to the boundary at the origin, and the $q_{\bot}$-coordinate axis be
orthogonal to the boundary at $M$. So,
near the origin, we can write
\begin{equation}\label{qf}
Q(q_{\parallel},q_{\bot})=q_{\bot} + O(q_{\parallel}^2+q_{\bot}^2).
\end{equation}
The equation of motion for the single-particle Hamiltonian (\ref{eq:singlepartbiliarddel}) are
\begin{equation}\label{eq:qparp}
\begin{array}{ll}
\frac{d}{dt}  q_{\parallel} = p_{\parallel},  \qquad\qquad\qquad\quad\ \frac{d}{dt}   q_{\bot} =  p_{\bot}, \\
\frac{d}{dt}   p_{\parallel} =  \frac{\alpha\delta\nabla_{\parallel}  Q}{ Q^{\alpha+1}} , \qquad \qquad \quad
\frac{d}{dt}   p_{\bot} =  \frac{\alpha\delta\nabla_{\bot}  Q}{ Q^{\alpha+1}}.
\end{array}
\end{equation}
We take a small \(\eta>0\) and consider an $O(\eta)$-neighborhood of \(M\). In this region,
\(q=O(\eta)\), and \(\nabla Q( q)=(\nabla_{\parallel} Q,\nabla_{\bot} Q)_{ q}=(O(\eta),1+O(\eta))\).
Then, we see that the value of $ p_{\bot}$ monotonically increases with time and the change in $p_{\parallel}$ is much smaller than the change of $p_{\bot}$ (because
$d p_{\parallel}/d p_{\bot} =O(\eta)$ which is small).
Since the range of possible values of $ p_{\bot}$ is bounded by the energy conservation, it follows that $p_{\parallel}$ is an almost conserved quantity (can change at most by \(O(\eta)\)).

Take a sufficiently large constant $K$. Define the outer boundary layer as the region
\(\delta^{1/\alpha}K<Q(q)<\eta\) (in particular,  we consider sufficiently small \(\delta\) so that \(\delta^{1/\alpha}\ll\eta\)).   For sufficiently small
\(\delta\), the initial value \(q_0\) belongs to this region. In this outer layer, the value of the potential energy
$\frac{\delta}{Q^\alpha}$ is small of order \(O(K^{-\alpha})\), so the maximal possible change in the kinetic energy is
$O(K^{-\alpha})$ as long as the trajectory stays in the layer. Thus, the velocity vector
$p=\dot q$ remains almost constant, $p(t)= p_0+O(K^{-\alpha}+\eta)$, by the approximate conservation of the kinetic energy and of
\(p_{\parallel}\). It follows that the trajectory is close to a straight line (i.e., to the billiard trajectory) and $q(t)$ moves inward, towards the impact. Therefore, there exists some time,
\(t_{in}(q_0,p_0;\delta)=O(\eta)\) at which the trajectory crosses the surface
$S_{K,\delta}:  Q( q)=\delta^{1/\alpha}K$.

\begin{figure}
\begin{centering}
\includegraphics[scale=0.3]{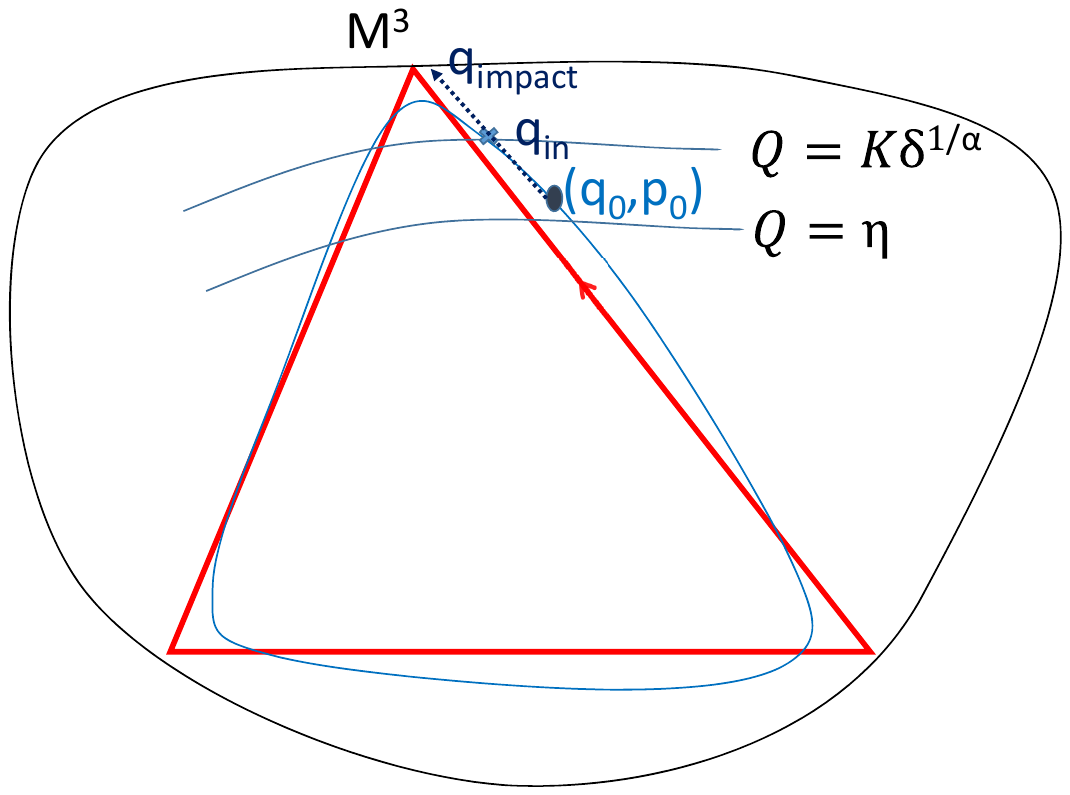}
\par\end{centering}
\protect\caption{\label{fig:boundarylayer} Boundary layers near impact for small \(\delta\). The outer boundary layer lies between the level sets \(Q=\eta\) and  \(Q=K\delta^{1/\alpha}\) whereas the inner layer is defined by \(Q<K\delta^{1/\alpha}\).
%The impact point \(q_{impact}(q_0,p_0,\delta)\) is the crossing point of the   trajectory with the boundary between these layers,   \(Q=K\delta^{1/\alpha}\), and it depends smoothly on initial conditions for all \(\delta\geqslant0\).
}
\end{figure}

Let us obtain more precise estimates for $(q(t),p(t))$ for $t\leq t_{in}$. On this time interval,
\(\frac{dQ}{dt}=\nabla Q(q)\cdot p=p_{\bot}(t)+O(\eta)=p_{\bot}(0)+O(K^{-\alpha}+\eta)<0\). Hence, we can choose \(Q\) as a new time. In fact, it is more convenient to choose
\(s=\delta^{-1/\alpha}Q\) as the new rescaled time, so the equations become:
\begin{equation}\label{eqm}
\begin{array}{ll}
\frac{dq}{ds}   &=\delta^{1/\alpha}\frac{p}{ f(q,p)}, \\
\frac{dp}{ds}   &=\frac{\alpha}{s^{\alpha+1}}  \frac{\nabla  Q(q)}{f(q,p)},\\
 \frac{dt}{ds}   &=\delta^{1/\alpha}  \frac{1}{f(q,p)},
\end{array}
\end{equation}
where \(f=\nabla Q(q)\cdot p\) is a smooth function of \((q,p)\).
Hence,
\begin{equation}\label{eq:scaledoutint}
\begin{array}{ll}
q(s)   &  \displaystyle=q_{0}+\delta^{1/\alpha}\int_{s_0}^{s}\frac{p(\sigma)}{ f(q(\sigma),p(\sigma))}d\sigma =
q_{0}+\int_{\delta^{1/\alpha}s_0}^{\delta^{1/\alpha}s}\frac{p(\sigma)}{ f(q(\sigma),p(\sigma))}d(\delta^{1/\alpha}\sigma),  \\ \\
p(s)   & \displaystyle=p_{0}+\int_{s_{0}}^{s}\frac{\nabla  Q(q(\sigma),p(\sigma))}{ f(q(\sigma),p(\sigma))}\frac{\alpha d\sigma}{\sigma^{\alpha+1}} =
p_{0}-\int_{s^{-\alpha}_0}^{s^{-\alpha}}\frac{\nabla  Q(q(\sigma),p(\sigma))}{ f(q(\sigma),p(\sigma))}d(\sigma^{-\alpha}) ,\\ \\
 t(s)   & \displaystyle= \delta^{1/\alpha}\int_{s_0}^{s}\frac{1}{ f(q(\sigma),p(\sigma))}d\sigma =\int_{\delta^{1/\alpha}s_{0}}^{\delta^{1/\alpha}s}    \frac{d(\delta^{1/\alpha}\sigma)}{ f(q(\sigma),p(\sigma))}
\end{array}
\end{equation}
where  \(s\in[K,\delta^{-1/\alpha}\eta]\). The solution of this system
of integral equations is obtained by the contraction
mapping principle: the integrands are bounded with all derivatives with respect to \(q,p \) and the integration intervals are small (of order
\(O(\eta)\) in the first and third equations and \(O(K^{-\alpha})\) in the second equation). It follows that, for any
fixed \(s_{0},s\) in the outer boundary layer, we have smooth dependence of \((q,p,t)\) on \((q_{0},p_0)\),
for all small \(\delta\), including the limit \(\delta=0\). The smoothness with respect to \(s_{0},s\) follows from the system (\ref{eqm}), as \(s\) is bounded away from zero.

We denote the solution of this system by
\((q^{int}(s,q_{0},p_0;\delta),p^{int}(s,q_{0},p_0;\delta),t^{int}(s,q_{0},p_0;\delta))\). Note that
 \((q^{int}(s_{0},q_{0},p_0;\delta), p^{int}(s_{0},q_{0},p_0;\delta),
t^{int}(s_{0},q_{0},p_0,t_{0};\delta))=(q_{0},p_0,0)\). Since the time $t_{in}$ corresponds to the time instance the trajectory hits the cross-section $S_{K,\delta}$, i.e., it corresponds to $s=\delta^{-1/\alpha} Q = K$, we obtain that
$$t_{in}(q_{0},p_0;\delta)=t^{int}(K,q_{0},p_0;\delta).$$
We also define
$$q_{in}(q_{0},p_0;\delta)=q^{int}(K,q_{0},p_0;\delta),$$
i.e., the point where the orbit of $(q_0,p_0)$ hits $S_{K,\delta}: Q(q)=K\delta^{1/\alpha}$.
As we have shown, these are smooth functions of $(q_0,p_0)$, uniformly for all \(\delta\geqslant 0\).

Introducing the rescaled time $t_s$ by formula (\ref{eq:deftimpact}), we
find that \(\frac{\partial}{\partial s}t_s=\frac{\partial}{\partial s}\frac{t^{int}}{\delta^{1/\alpha}} =  \frac{1}{ f(q^{int},p^{int})}\),  which is bounded with all derivatives and is bounded away from zero. Hence, \(s(t_s,q_0,p_0;\delta)\) is a smooth function of its arguments for all \(\delta\geqslant0\) (large
\(s\) corresponds to approaching the outer boundary of the boundary layer, where
\(t_s\rightarrow-\infty\)). Therefore, in the outer boundary layer, the function
\begin{equation}\label{tldpeqpint}
\tilde p(t_{s},q_0,p_0;\delta)=p^{int}(s(t_{s},q_{0},p_0;\delta),q_{0},p_0;\delta)
\end{equation}
is a smooth function of its arguments. Since $\dot q=p$, and
$q(t_{in},q_0,p_0)=q_{in}$,
we have
$$
q(t;q_0,p_0)=q_{in}+\int_{t_{in}}^t\tilde p(t_{s},q_0,p_0;\delta)dt=
q_{impact}+(q_{in}-q_{impact}) +\delta^{1/\alpha} \int_{0}^{t_s}\tilde p(u,q_0,p_0;\delta)du.
$$
Note that $q_{in}$ is the point on $S_{K,\delta}$, i.e.,
$Q(q_{in})=K\delta^{1/\alpha}$. We also have $Q(q_{impact})=0$ (see (\ref{impbev})), hence
$Q(q_{in})-Q(q_{impact})=K\delta^{1/\alpha}$. Since the gradient of $Q$ is bounded away from zero,
it follows that   \((q_{in}-q_{impact})\) is \(\delta^{1/\alpha}\)  times  a smooth function of $(q_0,p_0)$,
continuously depending on $\delta\geqslant 0$ with all derivatives. Hence it can be incorporated into \(\tilde q\)  : \begin{equation}
\tilde q(t_{s},q_0,p_0;\delta)=\delta^{-1/\alpha} (q_{in}-q_{impact})+\int^{t_s}_0\tilde p(u,q_0,p_0;\delta)du.
\nonumber \end{equation}As we see, the claim of the lemma, including formulas (\ref{eq:scaledq}),(\ref{formtq}), follows for the initial segment of the orbit (i.e., as long as it stays in the outer layer).

Let us now prove formula (\ref{bimp}). By (\ref{eq:scaledq}), (\ref{eq:deftimpact}), (\ref{formtq}), we have
$$q_0 + p_0 t_{in}(q_{0},p_0;\delta)=q_{impact}+O(\delta^{1/\alpha})+
\delta^{1/\alpha} \int_0^{-\delta^{-1/\alpha} t_{in}}\!\!\!\!\!\!\!\!\!\!(\tilde p(t_s,q_0,p_0;\delta) - p_0) dt_s.
$$
 This implies (\ref{bimp}) because the last term in the above formula tends to zero, along with all derivatives, as $\delta\to 0$. Indeed, by (\ref{tldpeqpint}) and by the second equation of (\ref{eq:scaledoutint}) we have
\begin{equation}\label{dvpo}
\tilde p(t_s,q_0,p_0;\delta) - p_0 = O(s^{-\alpha}),
\end{equation}
along with derivatives up to any given order. So, since $\frac{dt_s}{ds}=\frac{dt_s}{dt} \frac{dt}{ds}$ is uniformly bounded with derivatives (by (\ref{eq:deftimpact}) and the third equation of (\ref{eqm})), we have that (recall that \(s_{0}\)  is bounded by \(\eta\delta^{-1/\alpha}\)):\begin{equation}\label{qvpo}
\delta^{1/\alpha}\int^{s_{0}}_{K}(\tilde p(t_s,q_0,p_0;\delta) - p_0)\frac{dt_s}{ds}ds= \begin{cases}\delta^{1/\alpha}O(s_0^{1-\alpha})=O(\delta), & \alpha<1, \\
\delta^{1/\alpha}O(\ln (s_{0}))=O(\delta^{1/\alpha}\ln \delta),\ & \alpha=1, \\
 O(\delta^{1/\alpha}), & \alpha>1, \\
\end{cases}
\end{equation}
as required.

Next, we study the trajectory in the inner layer $Q(q)\leqslant\delta^{1/\alpha}K$. Here, \(p_{\bot}\) changes rapidly, so the trajectory quickly exits this inner layer, intersects the surface $S_{K,\delta}$
again, and returns back to the outer layer. After the rescaling
\(\bar q=\delta^{-1/\alpha}(q-q_{in}), t_s= \delta^{-1/\alpha}(t-t_{in}\)), the Hamiltonian (\ref{eq:singlepartbiliarddel}) becomes
\begin{equation}\label{eq:rescbilsingle}
H_0=\frac{p^2}{2} + \bar Q(\bar q,\delta)^{-\alpha}
\end{equation}
where
\begin{equation}
\bar Q(\bar q,\delta)=\frac{Q(q_{in}+\delta^{1/\alpha} \bar q)}{\delta^{1/\alpha}}.
\end{equation}
Note that \(\bar Q(\bar q,\delta)\) is a smooth function of $\bar q$ with bounded derivatives for all
\(\delta\geqslant0\). By construction, the inner layer is given by $K \geqslant \bar Q \geqslant H_0^{-1/\alpha}>0$ (where $H_0$ is the conserved energy, see (\ref{eq:rescbilsingle})).

The rescaled system, as given by the Hamiltonian (\ref{eq:rescbilsingle}), is
\begin{equation}\label{eq:qbarpar}
\frac{d\bar q}{d t_s}  =p,\qquad
\frac{d p_{\parallel}}{d t_s} =  \frac{\alpha\nabla_{\parallel} \bar Q}{\bar Q^{\alpha+1}} , \qquad \quad
\frac{d p_{\bot}}{d t_s} =  \frac{\alpha\nabla_{\bot} \bar Q}{\bar Q^{\alpha+1}},
\end{equation}
where
$\nabla_{\parallel} \bar Q = \nabla_{\parallel} Q = O(\|q_{in}\|+\delta^{1/\alpha})$
and \(\nabla_{\bot} \bar Q=\nabla_{\bot}  Q= 1+O( K\delta^{1/\alpha}) > 0\).
Since $\bar Q$ is uniformly bounded in the inner layer, it follows that $\frac{d}{d t_s} p_{\bot}$ is positive and bounded away from zero. By the conservation of energy, $p_{\bot}$ cannot grow unbounded, hence the orbit must leave the inner layer in a finite time, which we denote
\(t_{s,out}\). This time is bounded for all small $\delta\geqslant 0$ (so the unscaled passage time is of
order \(O(\delta^{1/\alpha})\)).

The system (\ref{eq:qbarpar}) is well-defined at $\delta=0$, so the solution on any finite interval
of the integration time $t_s$ is a smooth function of the initial conditions and parameters, continuously depending on $\delta$ for all small $\delta\geqslant 0$. Note that the initial condition at $t_s=0$ is
$\bar q= 0$,
$p = p^{int}(K,q_0,p_0;\delta)$; the right-hand side also depends smoothly on the value of
$q_{in}$, which is a smooth function of $q_0,p_0$. Thus, we have a smooth
dependence on $q_0, p_0$ and
the scaled time $t_s$ for all small $\delta\geqslant0$, i.e., the claim of the lemma continues to hold
as long as the solution is in the inner layer.

Let us show that the exit time $t_{s,out}$ is a smooth function of the initial conditions and parameters of the system, i.e., it is a smooth function of $q_0$ and $p_0$. This moment of time corresponds to arriving at the cross-section
$S_{K,\delta}: \bar Q=K$, so we just need to show that $\frac{d}{d t_s} \bar Q$ is bounded away from zero. To do that, note that
since the time in the boundary layer is bounded, the change in $p_{\parallel}$ is small, of order
$O(\nabla_{\parallel} \bar Q)  = O(\|q_{in}\|+\delta^{1/\alpha})$. Since the energy (\ref{eq:rescbilsingle}) is conserved and the value of $\bar Q$ at the entrance and the exit from the boundary layer is the same, the kinetic energy
$\frac{1}{2} (p_{\parallel}^2+ p_{\bot}^2)$ is also the same. Hence $|p_{\bot}|$ at the moment of exit is
$O(\|q_{in}\|+\delta^{1/\alpha})$-close to the value of $|p_{\bot}|$ at the moment of entrance, so it is
$O(K^{-\alpha})$-close to $|p_{0\bot}|$ by (\ref{dvpo}) (the sign of $p_\bot$ must change since the orbit is going away from the billiard boundary now). It follows that
$$\frac{d}{d t_s} \bar Q = \nabla_{\bot} Q\cdot\ p_\bot + \nabla_{\parallel} Q\cdot p_\parallel = - p_{0\bot} + O(\|q_{in}\|+\delta^{1/\alpha}+K^{-\alpha}) >0$$
is bounded away from zero, as required.

As $t_{s,out}$ depends smoothly on $q_0$ and $p_0$ for all $\delta\geqslant 0$, the values of $q=q_{out}$ and $p=p_{out}$ at the moment of exiting the inner layer also depend smoothly on $q_0$ and $p_0$. Note that
we have just shown that
\begin{equation}\label{impoutp}
p_{out}=(p_{0\parallel}, - p_{0\bot}) + O(\|q_{in}\|+\delta^{1/\alpha}+K^{-\alpha}).
\end{equation}

Once the trajectory crosses $S_{K,\delta}$ towards the outer boundary layer
 (i.e., \(p_{\bot}>0\) now), we can again use the integral equations (\ref{eq:scaledoutint}) to establish the smooth dependence on the initial conditions
$(q_{out}, p_{out})$ - hence on $(q_0,p_0)$ - and the scaled time. So, the solutions in this final segment also satisfy the claim of the lemma.

It remains to establish the reflection law (\ref{impp}). For a given initial condition, we choose the origin of coordinates to be the billiard impact point. The billiard reflection law then is that $p_\parallel$ remains the same and $p_\bot$ changes sign.
By (\ref{dvpo}), we have that
$$\tilde p=(p_{0\parallel}, p_{0\bot}) + O(K^{-\alpha})$$
before entering the inner layer and, taking into
account the change in $p$ in the inner layer, as given by (\ref{impoutp}) we find that in the limit
$\delta\to 0$
$$\tilde p=(p_{0\parallel}, - p_{0\bot}) + O(K^{-\alpha})$$
after exiting the inner layer (we have $q_{in}\to 0$ as $\delta\to 0$ because we put
the coordinate origin at the billiard impact point $q_{impact}=\lim_{\delta\to 0} q_{in}$).
Thus, except for the bounded interval of the rescaled time $t_s$ for which the orbit is in the inner layer, the deviation from the billiard reflection law is bounded by
$O(K^{-\alpha})$ as $\delta\to0$. Since $K$ can be chosen as large as we want, and the result cannot depend on \(K\), this means that the deviation from the billiard law in the limit $\delta=0$ is zero, i.e., the billiard reflection
law is approached indeed.
\end{proof}

\subsubsection{Flow-box coordinates in the boundary layer}
Consider a billiard trajectory
near a regular impact point  \(M\in\partial D\).
\begin{lem}\label{lem:fixedtau} One can choose cross-sections  \(S_{\delta}^{-}\) and
\(S_{\delta}^+\) in the phase space such that the billiard trajectory intersects, transversely,
the cross-section \(S^{-}_0\) before the impact (at a point $(q^\prime_0,p^\prime_0)$) and
the cross-section \(S^{+}_0\) after the impact, and the following holds.
For all \((q_{0},p_{0})\in S_{\delta}^{-}\) the impact time \(t_{in}(q_{0},p_{0};\delta)\)
of Lemma \ref{lem:qnearimpact} is constant and equal to
$t_{in}(q^\prime_{0},p^\prime_{0};\delta)>0\),
and the flight time
from \( S_{\delta}^-\) to  \( S_{\delta}^+\) is constant for the orbits of system (\ref{eq:singlepartbiliarddel}) and equals to  \(2t_{in}(q^\prime_{0},p^\prime_{0};\delta)\). Such cross-sections are bounded away from the billiard boundary and depend continuously on
\(\delta\). The Poincar\'e map \( S_{\delta}^- \to S_{\delta}^+\) by the orbits of the system
tends, with all derivatives, to the Poincar\'e map \( S_0^- \to S_0^+\) by the billiard flow.
\end{lem}
\begin{proof} For the billiard flow, the function \( t_{in}(q_{0},p_{0};0)\) satisfies
\(Q(q_{0}+p_{0}\cdot t_{in})=0 \). Since the impact at point $M$ is regular, we have
\(p_{0}\cdot\nabla Q(q_{0})\neq 0\), hence the  equation
 \( t_{in}(q,p;0)= t_{in}(q^\prime_0,p^\prime_0;0) \) defines a smooth hypersurface, which is the  cross-section
\(S^-_0\). By the continuous dependence of  \(t_{in}(q,p;\delta)\) and its derivatives
on \(\delta\), the equation
\(t_{in}(q,p;\delta) = t_{in}(q^\prime_0,p^\prime_0;\delta)\) defines a smooth hypersurface, which is the cross-section \(S^-_\delta\).

To satisfy the lemma, we must choose \(S^+_\delta\) as the image of \(S^-_\delta\) by the
time-$2t_{in}$ flow map. Therefore, by Lemma \ref{lem:qnearimpact}, the cross-section
\(S^+_\delta\) consists of the points $(q,p)$ satisfying
\begin{equation}\label{sdeltaplus}
q =q_{impact}
+\delta^{1/\alpha} \tilde q(t_{in}\delta^{-1/\alpha},q_0,p_0;\delta),\qquad
p = \tilde p(t_{in}\delta^{-1/\alpha},q_0,p_0;\delta),
\end{equation}
where $(q_0,p_0)$ is in \(S^-_\delta\). Formula (\ref{sdeltaplus}) defines the Poincar\'e map
\( S_{\delta}^- \to S_{\delta}^+\). By (\ref{formtq}),
$$q =q_{impact} + p t_{in}+O(\delta^{1/\alpha}) +\delta^{1/\alpha}\int_0^{t_{in}\delta^{-1/\alpha}} (\tilde p(u,q_0,p_0;\delta)-p) du.$$
As in Lemma \ref{lem:qnearimpact}, it follows  (similar to (\ref{qvpo})) that the last term in this formula tends to zero
as $\delta\to 0$, along with all derivatives with respect to $(q_0,p_0)$. Therefore, $q$ and $p$
in (\ref{sdeltaplus}) have a well-defined limit as $\delta\to0$, and \(S^+_\delta\) tends to
$S_0^+$, the time-$2t_{in}$ image of $S_0^-$ by the billiard flow. For regular impacts,
$S_0^+$ is a well-defined smooth hypersurface and is bounded away from the boundary. By continuity, the same is true for  \(S^+_\delta\) for all small $\delta\geqslant\ 0$.
\end{proof}

Next, we introduce flow-box coordinates \cite{abraham2008foundations} for the union of trajectory segments that are close to the billiard trajectory $(q^*,p^*)$ in the boundary layer. Precisely, let
\(\mathcal{U}_\delta\) denote the union of the segments of trajectories of system (\ref{eq:singlepartbiliarddel}) that cross \(S^{-}_\delta\) at $t=0$ and correspond to an open time-interval containing the flight-time interval
\([0,2t_{in}] \) (at $\delta=0$, we take
\(\mathcal{U}_0\) as the union of the corresponding billiard trajectories).
\begin{lem} \label{lem:energytimecoor}
In  \(\mathcal{U}_\delta\), one can make a symplectic change of coordinates
\((q,p)\rightarrow(\tau,\mathcal{E},\mathcal{P})\in
\mathbb{R}^{1}\times\mathbb{R}^{1}\mathbb{\times R}^{d-2}\),
such that \((\tau,\mathcal{E})\) is the symplectic time-energy pair,
\(\tau=0 \) at  \(S_{\delta}^{-}\) and
\(\tau(q,p;\delta)=2t_{in}\) at \(S_{\delta}^{+}\),
and the system in \(\mathcal{U}_\delta\) acquires the form
\begin{equation}\label{tauep}
\dot \tau=1,  \qquad \mathcal{\dot E}=0, \qquad \mathcal{\dot P}=0.
\end{equation}
The coordinate transformation is smooth and depends continuously on $\delta$, along with the derivatives, for $\delta>0$.
Moreover, uniformly for all $\delta\geqslant 0$, the variable $q$ depends $C^\infty$-smoothly on $(\mathcal{E},\mathcal{P})$ for every fixed $\tau$, and it depends continuously on $\tau$ and $\delta$, along with all derivatives with respect to
$(\mathcal{E},\mathcal{P})$.

Away from the billiard boundary (in particular, near \(S_{\delta}^{-}\) and  \(S_{\delta}^{+}\) ), the transition maps between the coordinates $(q,p)$ and $(\tau,\mathcal{E},\mathcal{P})$ are smooth, and depend on $\delta$ continuously, in $C^\infty$, for all
\(\delta\geqslant 0\).
\end{lem}
\begin{proof} We use the standard flow-box construction: Given a point in $\mathcal{U}_\delta$, we take the value of the
Hamiltonian (\ref{eq:singlepartbiliarddel}) at this point and define $\mathcal{E}$ as this value minus $\frac{1}{2}$; we define $\tau$ as the time it takes a trajectory of system (\ref{eq:singlepartbiliarddel}) to reach this point from \(S_{\delta}^-\). The
coordinates \(\mathcal{P}\) are taken constant along the trajectory and, hence,
equal to the coordinates \(\mathcal{P}\) of the intersection point of the trajectory
with \(S_{\delta}^-\).  One can choose coordinates
\(\mathcal{P}\) on \(S_{\delta}^-\) such that the resulting coordinate system
\((\tau,\mathcal{E},\mathcal{P})\) is symplectic in $\mathcal{U}_\delta$  \cite{abraham2008foundations}. Formula (\ref{tauep})
is immediate from the construction (recall that the energy $\mathcal{E}$ is conserved).

Let us examine the regularity of the coordinate transformation in the limit $\delta\to 0$.
Away from the billiard boundary, the flow of
(\ref{eq:singlepartbiliarddel}) has a regular limit (the billiard flow). It follows that
if we remove from $\mathcal{U}_\delta$ the points whose $q$-component belongs to a given small
neighborhood of the boundary, then the resulting set will have two connected components
for all sufficiently small $\delta$. One component is comprised by trajectory pieces that intersect
$S_{\delta}^-$, and the other by pieces that intersect $S_{\delta}^+$.
It is immediate by construction that in the first connected component
the flow-box coordinate transformation is regular up to $\delta= 0$, as required.

For each point of the second component, one has well-defined (for all $\delta\geqslant 0$)
correspondence between the $(q,p)$ coordinates of the point and the flow-box coordinates $(\tau', \mathcal{E}',\mathcal{P}')$,
where $(\mathcal{E}',\mathcal{P}')$ are the coordinates of the intersection of the orbit of the point with $S_{\delta}^+$ and
$\tau'$ is the time
the orbit needs to arrive to $S_{\delta}^+$. We have $\tau = 2t_{in}-\tau'$, so since $t_{in}$ is a constant,
it follows that to establish the regularity of the transformation $(q,p)\to (\tau,\mathcal{E},\mathcal{P})$ in the second connected component, we just need to show the regularity of the transformation
$(\mathcal{E}',\mathcal{P}')\to (\mathcal{E},\mathcal{P})$ for the points on the cross-section $S_{\delta}^+$ only.

It remains to note that the coordinates $(\mathcal{E},\mathcal{P})$ stay constant along a trajectory, so for any point on
$S_{\delta}^+$ the correspondence $(\mathcal{E},\mathcal{P})\to (\mathcal{E}',\mathcal{P}')$ defines the
Poincar\'e map \(S_{\delta}^- \to S_{\delta}^+\) whose regularity  for all $\delta\geqslant 0$, is established by Lemma \ref{lem:fixedtau}.

We have shown the regularity of the flow-box coordinates away from the billiard boundary. Now, to finish the lemma, we
discuss the dependence of $q$ on the flow-box coordinates. Since $\tau$ is the time variable and $(\mathcal{E},\mathcal{P})$
give the initial conditions on the cross-section $S^-_{\delta,j}$, we can write (\ref{eq:scaledq}) as
\begin{equation}\label{eq:qhatEPtau}
q = q_{impact}+
\delta^{1/\alpha} \tilde q(\frac{\tau - t_{in}}{\delta^{1/\alpha}},\mathcal{E},\mathcal{P};\delta).
\end{equation}
Since the value of  $t_{in}$ is independent of
$(\mathcal{E},\mathcal{P})\in S^-_{\delta,j}$, the required regularity of $q$
follows immediately from
Lemma \ref{lem:qnearimpact}.
\end{proof}

\subsubsection{\label{sec:periodofL}The periodic orbit \(L^*_\delta\)}
Choose, near each regular impact point \(M^j\), $j=1,\ldots, k^*$, of the billiard orbit $L^*$,
the cross-sections
\(S_{\delta,j}^{-}\)  and  \(S_{\delta,j}^{+}\) as in  Lemma \ref{lem:fixedtau}. The cross-sections depend continuously on $\delta$ and stay bounded away from the billiard boundary.
Therefore, in a neighborhood of the piece of $L^*$ between the consecutive cross-sections
\(S_{\delta,j}^+\) and $S^-_{\delta,j+1}$, the potential tends to zero as $\delta\to 0$
and the motion tends to the constant speed motion, as in the billiard. Hence, for the orbits of
system (\ref{eq:singlepartbiliarddel}) the flight time from $S_{\delta,j}^+$ to
$S^-_{\delta,j+1}$ tends to the billiard flight time and the Poincar\'e map
\(S_{\delta,j}^{+}\rightarrow S^{-}_{\delta,j+1}\) tends to the corresponding Poincar\'e map of the billiard flow,
along with the derivatives with respect to the initial conditions.

By Lemma \ref{lem:fixedtau}, the same is true for the Poincar\'e map and the flight time
between \(S_{\delta,j}^-\) and \(S_{\delta,j}^+\) (the flight time does not depend on initial conditions and equals to \(2t_{in}\), so it tends to the billiard flight time by Lemma \ref{lem:qnearimpact}). Thus, we obtain that the return map
to $S_{\delta,1}^-$  for the flow (\ref{eq:singlepartbiliarddel}) is \(C^r\) close, for all \(r\),
to the return map of the billiard flow (the return map is the composition of the Poincar\'e maps
\(S_{\delta,j}^-\to S_{\delta,j}^+\) and
\(S_{\delta,j}^{+}\rightarrow S^{-}_{\delta,j+1 \mod k^*}\) for $j=1,\ldots, k^*$). Also, the return time to $S_{\delta,1}^-$ is
\(C^r\) close to the return time for the billiard flow.

By Assumption BD2, the intersection point of \(L^*\) with $S_{0,1}^-$ is a non-degenerate
non-resonant (up to order 4) elliptic fixed point of the billiard return map at the energy level
$H=\frac{1}{2}$ (this corresponds to the motion with the speed $1$). Such fixed points persist at small perturbations, so the return map to $S_{\delta,1}^-$ also has a KAM-non-degenerate elliptic fixed point at the energy level to $H=\frac{1}{2}$ for all small $\delta$.

This gives us an elliptic periodic orbit \(L^*_\delta\) of system (\ref{eq:singlepartbiliarddel})
such that \(L^*_\delta\cap S_{\delta,1}^-\)
tends to $L^*\cap S_{0,1}^-$ as $\delta\to 0$. By the continuous dependence of the return time on $\delta$, the period of
\(L^*_\delta\) tends to the period of the billiard orbit $L^*$. Recall that we use the notation
$L^*=(q^*(\omega_0 t), p^*(\omega_0 t)$,
where $\omega_0=\frac{2\pi }{|L^{*}|}$ and $(q^*,p^*)$ are $2\pi$-periodic functions,
$q^*$ is continuous and piece-wise linear, and $p^*$ is discontinuous and piece-wise constant.
We denote $L^*_\delta=(q^*_\delta(\omega_0(\delta) t), p_\delta^*(\omega_0(\delta) t)$
for some smooth $2\pi$-periodic functions $(q^*_\delta, p^*_\delta)$. Here,
$\omega_0(\delta)\to \omega_0$ as $\delta\to 0$, and
$(q^*_\delta, p_\delta^*)\to (q^*,p^*)$, with derivatives, if $q^*$ stays away from the billiard boundary. Since $q^*_\delta$ is uniformly Lipshitz and the time spent in the boundary layer is small, it follows that $q^*_\delta\to q^*$ in $C^0$ for all $t$.

The same holds true in every energy level close to $H=\frac{1}{2}$. The billiard has a periodic
orbit $L^*(E)$ in the energy level \(H=E\) which follows the same path as $L^*$ in the $q$- space, with the speed
\(\|p\|=\sqrt{2E}\) and period  \(T(E)=\frac{|L^{*}|}{\sqrt{2E}}\). By the same arguments as above, the system (\ref{eq:singlepartbiliarddel}) for all sufficiently small
\(\delta\) has an elliptic periodic orbit \(L_{\delta}^*(E)\) in the energy level \(H=E\), and
the family of the orbits \(L_{\delta}^*(E)\) approaches \(L^*(E)\) as $\delta\to 0$. In particular,
the period \(T_{\delta}(E)\) tends to \(T(E)\) along with derivatives with respect to $E$
(because the return time to $S_{\delta,1}^-$ tends to the billiard return time with derivatives
with respect to the initial conditions).

%\subsubsection{Action-angle coordinates near the periodic orbit}
For every $\delta>0$, for each impact point $M^j$, $j=1,\dots, k^*$, let us choose the region
$\mathcal{U}_{\delta,j}$ as in Lemma \ref{lem:energytimecoor} (i.e., this region consists of the
orbits of (\ref{eq:singlepartbiliarddel}) that connect the cross-sections
\(S_{\delta,j}^-\) and \(S_{\delta,j}^+\)). Since \(L^{*  }\) is regular periodic orbit, its impacts are distinct, so the regions $\mathcal{U}_{\delta,j}$ do not overlap for different $j$. Let
$\mathcal{U}^0$ be a sufficiently small, yet independent of $\delta$, open neighborhood of
the part of $L^*_\delta$ which is not covered by
the union of $\mathcal{U}_{\delta,j}$.

We use $(q,p)$ as coordinates in $\mathcal{U}^0$; note that for the points
$(q,p)\in\mathcal{U}^0$, the $q$-component is bounded away from the billiard boundary.
In $\mathcal{U}_{\delta,j}$, we use the flow-box coordinates
$(\tau,\mathcal{E},\mathcal{P})_j$ given by Lemma \ref{lem:energytimecoor}. We restrict the freedom in the choice of the
flow-box coordinates by the requirement that $(\mathcal{E},\mathcal{P})_j=0$ on the periodic orbit $L_{\delta}^*$.

Thus, we have covered a neighborhood of \(L_{\delta}^*\) by a system of coordinate charts
for all $\delta>0$.
The overlap region between $\mathcal{U}^0$ and $\mathcal{U}_{\delta,j}$
near the cross-sections \(S_{\delta,j}^{-}\) and  \(S_{\delta,j}^{+}\) stays bounded away from
the billiard boundary. Therefore, by Lemma \ref{lem:energytimecoor} the transition map between the $(q,p)$-coordinates in
$\mathcal{U}^0$ and the flow-box coordinates in
$\mathcal{U}_{\delta,j}$ in the overlap region is symplectic, depends continuously on $\delta$,
and has a well-defined limit, in $C^\infty$, as $\delta\to 0$.

In other words, we have introduced the structure of a smooth symplectic manifold in the neighborhood of
\(L_\delta^*\) for \(\delta>0\), and this structure has a regular limit at $\delta=0$. The equations of motion also have a regular limit in these coordinates: in $\mathcal{U}^0$ the equations converge to
$$\dot q=p, \qquad \dot p=0$$
at $\delta=0$, and in $\mathcal{U}_{\delta,j}$ the equations are the same for all $\delta$:
$$\dot \tau=1, \qquad \dot{\mathcal{E}}=0, \qquad \dot{\mathcal{P}}=0.$$

The elliptic periodic orbit $L^*_\delta$ tends, as $\delta\to0$, to the billiard periodic orbit
$L^*=L^*_0$, which is a smooth curve in these  coordinates (for every $\delta\geqslant 0$,
the curve $L^*_\delta$ has the same equation $(\mathcal{E}, \mathcal{P})=0$ in the flow-box coordinates). As in Section \ref{sec:setupsinglepart}, we can introduce
action-angle coordinates near \(L^*_\delta\) for all $\delta\geqslant 0$,
with the resulting Hamiltonian as in (\ref{eq:birkhoffsingle}), yet here the coefficients depend (continuously) on \(\delta\):
\begin{equation} \label{eq:birkhoffsinglebilliard}
H_0(I_{0},\theta,z;\delta)=\frac{1}{2}+
\omega (\delta )I+\frac{1}{2}I^\top A(\delta )I+g(I_{0},\theta,z;\delta),
\end{equation}
where $g=g_0(I_0;\delta)+g_1(I_0;\delta)\hat I+O(\|z\|^4 |I_{0}|+\|z\|^{5})$
with $g_0=O(I_0^3)$, $g_1 = O(I_0^2)$ (see notations after (\ref{eq:birkhoffsingle})).
We stress that the symplectic transformation between the action-angle coordinates \((I_{0},\theta,z)\) and the energy-time coordinates   \((\tau,\mathcal{E}, \mathcal{P})\)  defined in the near-impact regions,  $\mathcal{U}_{\delta,j}$, is smooth for all \(\delta\geqslant0\).

Similar to Section \ref{sec:setupsinglepart}, we denote the relation between the action-angle coordinates and the $(q,p)$-coordinates as $(q,p)=(\hat q(I_{0},\theta,z),\hat p(I_{0},\theta,z))$.
Away from the billiard boundary, i.e., in $\mathcal{U}^0$, this coordinate transformation is well-defined in the limit $\delta=0$. However, near the impacts (in the regions
$\mathcal{U}_{\delta,j}$) the relation between
$(q,p)$ and the action-angle coordinates acquires singularities at $\delta=0$.

The periodic orbit $L^*_\delta$ corresponds to $(I_0,z)=0$. We have
$\dot \theta = \omega_0(\delta)$ on $L^*_\delta$, i.e., $\theta=\omega_0(\delta) t$.
Thus, in our notations, $L^*_\delta=(q^*_\delta(\theta), p_\delta^*(\theta))=(\hat q(0,\theta,0),\hat p(0,\theta,0))$.
On \(L^*_\delta\),   equation (\ref{eq:qhatEPtau}) near the $j$-th impact point, where for \(\delta=0\) the impact occurs at \(\theta_{j}\), becomes
\begin{equation}\label{eq:qdeltabound}
q^*_\delta(\theta)= M^j +
\delta^{1/\alpha} \tilde q_{j,\delta}\left(\frac{\theta-\theta_{j}}{\omega_0(\delta)\delta^{1/\alpha}}\right),
\end{equation}
where $\tilde q_{j,\delta}$ has all derivatives bounded, see Lemma \ref{lem:qnearimpact}. So, near impacts, the \(k\)-th derivative of
\(q^*_\delta(\theta) \) is of order \(\delta^{-(k-1)/\alpha}\).

Like in Section \ref{sec:proofmainthm}, we will also perform the scaling
\(z\rightarrow \delta^{1/4} z\), \(I_0\rightarrow \delta^{1/2}I_0\). In the scaled variables, the motion in
a small neighborhood of $L^*_\delta$ is described by the scaled single-particle Hamiltonian
\begin{equation} \label{eq:birkhoffsinglebilliardscaled}
H_{0,scal}=\delta^{-1/2}\left(H_0(\delta^{1/2}I_0,\theta,\delta^{1/4}z;\delta)-\frac{1}{2}\right)
=\omega (\delta )I+\frac{1}{2} \delta^{1/2}I^\top A(\delta )I+O(\delta^{3/4}).
\end{equation}

Note that away from the billiard boundary, we have
\begin{equation}\label{qdeltaaway}
\hat q (\delta^{1/2}I_{0},\theta,\delta^{1/4}z)- q^*_\delta(\theta) = O(\delta^{1/4}),
\end{equation}
with all derivatives with respect to $\theta$ and the scaled variables $z$ and $I_0$. Note also
that the scaling of $I_0$ and $z$ induces the scaling
\((\mathcal{E},\mathcal{P})\rightarrow ( \delta^{1/2}\mathcal{E},\delta^{1/4} \mathcal{P})\) for the flow-box coordinates near impacts; the scaled energy $\mathcal{E}$ equals to $H_{0,scal}$ in (\ref{eq:birkhoffsinglebilliardscaled}).
Recall that  the symplectic transformation between the unscaled \((I_{0},\theta,z)\) and \((\tau,\mathcal{E}, \mathcal{P})\)  is smooth for all \(\delta\geqslant0\). By   (\ref{eq:birkhoffsinglebilliardscaled}), scaling \(I_{0}\) and \(\mathcal{E}\) by $\delta^{1/2}$ remains a smooth transformation in the limit \(\delta=0\). Similarly, since \( \mathcal{P}\) depends smoothly on \((\delta^{1/2}I_{0},\delta^{1/4}z)\) and vanishes at the origin, it follows that its scaling by \(\delta^{1/4}\)    also remains smooth in this limit. Therefore, we conclude that the symplectic transformation between the scaled \((I_{0},\theta,z)\) and the scaled \((\tau,\mathcal{E}, \mathcal{P})\)  is smooth  for all \(\delta\geqslant0\).

\subsubsection{Conditions SP1-SP3 for the orbit \(L^*_\delta\).}
 \begin{lem}\label{lem:billasum1to3}
Under Assumptions BD1 and BD2 on the billiard orbit, the periodic orbit  \(L^*_\delta\) satisfies Assumptions SP1-SP3  for all sufficiently small \(\delta\geqslant 0\).
\end{lem}
\begin{proof}
For the family $L^*(E)$ of billiard periodic orbits that follow, with the speed \(\|p\|=\sqrt{2E}\), the same path as $L^*_0$ in the $q$- space, the period  \(T(E)\) equals to $\frac{|L^{*}|}{\sqrt{2E}}$. It decreases with the energy $E$.
Since for small \(\delta\) the period of  \(L^*_\delta(E)\) is close to the period of  \(L^*(E)\)  with the derivatives with respect to $E$ (see Section \ref{sec:periodofL}), the period of
\(L^*_\delta(E)\) also decreases with $E$, i.e.,
Assumption SP1 is verified.

By the closeness of the return maps near \(L^*_\delta\) and $L^*_0$, the multipliers of
\(L^*_\delta\) are close to the multipliers of $L^*_0$. Hence, by assumption BD2, there are no low-order resonances for all small $\delta$, i.e., assumption SP2 is satisfied.

The twist assumption SP3 for the orbit \(L^*_\delta\) is an open condition on
coefficients of system (\ref{eq:birkhoffsinglebilliard}). Because of the continuity in $\delta$,
it is enough to check this condition at \(\delta=0\). We evaluate the necessary coefficients of
(\ref{eq:birkhoffsinglebilliard}) at \(\delta=0\) by analyzing the return map near
$L^*_0=(I_0=0, z=0)$.

The system of differential equations defined by the Hamiltonian (\ref{eq:birkhoffsinglebilliard})
is, in restriction to the energy level $H_0=E$ close to $E=\frac{1}{2}$, given by
\begin{equation}\label{system00b} \begin{array}{l}
 \omega_{0}I_{0} = E -\frac{1}{2}-\hat \omega\hat I+O(\hat I^{2}),\\
\dot \theta =\omega_{0}+ (aI_0+b \hat I)+O(I^2),\\
\dot x_j =  (\omega_j+b_{j}I_{0}+(\hat A \hat I)_j + O(I_0^2))y_j+O(\|z\|^3|I_0|+\|z\|^4),\\
\dot y_j = -(\omega_j+b_{j}I_{0}+(\hat A \hat I)_j + O(I_0^2))x_j +O(\|z\|^3|I_0|+\|z\|^4) \qquad (j=1,\ldots, d-1),
\end{array}
\end{equation}
where all the coefficients are taken at $\delta=0$.
Recall the notation: $I=(I_0,\hat I)$, $z_j=(x_j,y_j)$, $I_j = \frac{1}{2} z_j^2$,
$\omega=(\omega_0,\hat\omega)$,  $A=\begin{pmatrix}a &  b \\  b^\top &   \hat A \\ \end{pmatrix}$, see Section \ref{sec:setupsinglepart}.

For each value of $E$, the line $(x,y)=0$ (equivalently \(\hat I=0\)) corresponds to a periodic orbit. Its period
equals to the travel time from $\theta=0$ to $\theta=2\pi$, i.e.,
$$T(E)=\frac{2\pi}{\omega_0 + a\frac{E-\frac{1}{2}}{\omega_0} +O((E-\frac{1}{2})^2)}=\frac{2\pi}{\omega_0 }(1- a\frac{E-\frac{1}{2}}{\omega_0^2} +O((E-\frac{1}{2})^2))$$
The linear part of the system for $(x_j,y_j)$ is rotation with the frequency
$\omega_j(E)=\omega_j+b_j\frac{E-\frac{1}{2}}{\omega_0}+O((E-\frac{1}{2})^2)$.  It follows that
the linearization of the  return map
$(x,y)\mapsto (\bar x,\bar y)$ at the fixed point $(x,y)=0$ is the rotation
of $(x_j,y_j)$, $j=1,\ldots,d-1$, by the angle
\begin{equation}\label{omgprime}
\Delta_j =T(E)\omega_j(E)= \frac{2\pi}{\omega_0}(\omega_j+
\frac{E-\frac{1}{2}}{\omega_0}(b_{j}-\frac{\omega_j a}{\omega_0})+O((E-\frac{1}{2})^2)).
\end{equation}

System (\ref{eq:birkhoffsinglebilliard}) at $\delta=0$ describes the billiard flow, so its return map
is the return map of a billiard flow. The billiard flow is invariant with respect to
the energy and time scaling, in particular the return map at the energy level \(E=\frac{1}{2}\) is  conjugate to the return map at any \(E\) by the scaling \((q,p)\rightarrow(q,p/\sqrt{2E})\).
The conjugacy implies that the rotation angles $\Delta_j$ in (\ref{omgprime}) are independent of $E$. Therefore,
\begin{equation}\label{bawbil}
b_j=\frac{\omega_j a}{\omega_0}, \qquad j=1, \ldots, d-1.
\end{equation}

Now, at $E=\frac{1}{2}$ we rewrite system (\ref{system00b}) as
$$\begin{array}{l}
\dot \theta =\omega_{0} + O(\|z\|^4),\\
\dot x_j =  (\omega_j-\frac{b_j}{\omega_0}\hat\omega\hat I+(\hat A \hat I)_j)y_j+
O(\|z\|^4),\\
\dot y_j = -(\omega_j-\frac{b_j}{\omega_0}\hat\omega\hat I+(\hat A \hat I)_j)x_j +
O(\|z\|^4) \qquad (j=1,\ldots, d-1),
\end{array}
$$
The return map $(\theta=0)\to(\theta=2\pi)$ for this system
coincides, up to $O(\|z\|^4)$-terms
with the (nonlinear) rotation of $(x_j,y_j)$, $j=1\ldots,d-1$, to the angles
$$\frac{2\pi}{\omega_0}(\omega_j-\frac{b_j}{\omega_0}\hat\omega\hat I+(\hat A \hat I)_j).$$

Thus, the map coincides, up to $O(\hat I^2)$-terms, with the normal form (\ref{eq:birkhoffbillmap}) where the action-angle coordinates $(I_j,\Phi_j)$
are introduced such that
$(x_j,y_j)=\sqrt{2I_j}(\cos(\Phi_j),\sin(\Phi_j))$, and the matrix $\Omega$
is given by
$$
\Omega_{jk} =\frac{2\pi}{\omega_0}(\hat A_{jk} - \frac{\omega_k b_j }{\omega_0})=
\frac{2\pi }{\omega_0}(\hat A_{jk}  -\frac{a\omega_j \omega_k}{\omega^2_0}).
$$
We have \(\det\Omega \ne 0\) by Assumption BD2. This implies that conditions
(\ref{twistel}) and (\ref{twistelw}) of Assumption SP3 are satisfied. Indeed,
using (\ref{bawbil}), we find
$$
\det A=a \det (\hat A - \frac{1}{a} b^\top b)=a\det(\frac{\omega_0}{2\pi}\Omega)\neq 0,
$$
and
$$
\det A_\omega=\det \left(\begin{array}{ccc} 0 & \omega_0 & \hat\omega \\  \omega_0 & a &  \frac{a}{\omega_{0}}\hat \omega \\ \hat\omega^\top & \frac{a}{\omega_{0}}\hat \omega^\top & \hat A \end{array}\right)=\det \left(\begin{array}{ccc} -\frac{\omega_{0}^2}{a} & 0 & 0 \\  \omega_0 & a &  \frac{a}{\omega_{0}}\hat \omega \\ \hat\omega^\top & \frac{a}{\omega_{0}}\hat \omega^\top & \hat A \end{array}\right)= -\frac{\omega_{0}^2}{a}\det A\neq 0.
$$
\end{proof}

\subsection{Multi-particle system\label{sec:billiardmultipart}}
Now we analyze the multi-particle system (\ref{eq:multparsysybill0}) near the minimal line (\ref{eq:lineminbill}) of the averaged interaction potential  \(U\) of (\ref{eq:averagedpotuD}).
We have established that for all small \(\delta\geqslant0\) the single-particle Hamiltonian can be brought to the form (\ref{eq:birkhoffsinglebilliard}). In these variables the multi-particle system is of the form
\begin{equation}\label{eq:multparwsingle}
H=\sum_{n=1}^N\left[\omega (\delta )I^{(n)}+\frac{1}{2}{I^{(n)}}^\top A(\delta )I^{(n)}+
g(I_{0}^{(n)},\theta^{(n)},z^{(n)};\delta )\right]+
\delta\mathop{\sum_{n,m=1,\dots,N}}_{n\neq m} W(q^{(n)}-q^{(m)}),
\end{equation}
where, as \(\delta \rightarrow 0\),
the dependence \(q^{(n)}=\hat q(I_0^{(n)},\theta^{(n)},z^{(n)})\) becomes singular near impacts,
as described in Section \ref{sec:singlepartbil}. This singularity is quite mild in the flow-box coordinates
$(\tau^{(n)},\mathcal{E}^{(n)},\mathcal{P}^{(n)})$: we have, uniformly for all $\delta\geqslant 0$, continuous dependence of $q^{(n)}$ on $(\tau^{(n)},\mathcal{E}^{(n)},\mathcal{P}^{(n)})$ and, for every fixed value of $\tau^{(n)}$, smooth dependence on
$(\mathcal{E}^{(n)},\mathcal{P}^{(n)})$, see Lemma \ref{lem:energytimecoor}.

Let us make the scaling
\(z^{(n)}\rightarrow \delta^{1/4} z^{(n)}\), \(I_0^{(n)}\rightarrow \delta^{1/2}I_0^{(n)}\).
By (\ref{eq:birkhoffsinglebilliardscaled}), we obtain the scaled version of (\ref{eq:multparwsingle}):
\begin{equation}\label{eq:wperhaatbil}\begin{array}{ll}
H_{scal} &= \delta^{-1/2}\; H\!\left((\delta^{1/2}I_0^{(n)},\theta^{(n)},\delta^{1/4}z^{(n)})_{_{n=1,\dots,N}}\right)
\\& \displaystyle
=\sum_n \left[\omega  (\delta )I^{(n)} + \delta^{1/2} \frac{1}{2} {I^{(n)}}^\top A  (\delta )I^{(n)}+O(\delta^{3/4})\right]\\& \displaystyle
\qquad\qquad+\delta^{1/2}\sum_{n\neq m} W(\hat q(\delta^{1/2}I_0^{(n)},\theta^{(n)},\delta^{1/4}z^{(n)})-\hat q(\delta^{1/2}I_0^{(m)},\theta^{(m)},\delta^{1/4}z^{(m)})).  \end{array}
\end{equation}
As we consider a neighborhood of the minimal line (\ref{eq:lineminbill}) which satisfies the collision-free Assumption IP2, we can assume that \(W\) is bounded with its derivatives (as in Theorem \ref{thm2}). Therefore, away from the impacts,
$$
W(\hat q(I_0^{(n)},\theta^{(n)},z^{(n)})-\hat q(I_0^{(m)},\theta^{(m)},z^{(m)}))=
W(q^*_\delta(\theta^{(n)};\delta)-q^*_\delta(\theta^{(m)};\delta))+O(\delta^{1/4}),$$
see (\ref{qdeltaaway}).

Thus, when all the particles stay away from the billiard boundary, the Hamiltonian (\ref{eq:wperhaatbil}) is $O(\delta^{3/4})$-close,
with all derivatives, to the truncated Hamiltonian
\begin{equation}\label{eq:nparttrunbillnew}
H_{trun} =\sum_{n} \omega (\delta)I^{(n)} + \delta^{1/2} \sum_n \frac{1}{2} {I^{(n)}}^\top A(\delta) I^{(n)}+
\delta^{1/2}\sum_{n\neq m} W(q_\delta^*(\theta^{(n)};\delta)-q^*_\delta(\theta^{(m)};\delta)).
\end{equation}
Note that we cannot claim the same near impacts, so there we need a more accurate comparison of systems
(\ref{eq:wperhaatbil}) and (\ref{eq:nparttrunbillnew}).

We consider the Hamiltonians (\ref{eq:wperhaatbil}) and (\ref{eq:nparttrunbillnew})  close to the line
$$
\mathcal{L}_{min}^{*}=\{(I_0^{(n)}=0,\theta^{(n)}=\theta^{(n)}_{min}+c,z^{(n)}=0)_{n=1,..N},c\in S^{1}\},
$$
the phase space image of the minima line (\ref{eq:lineminbill}), i.e. the solution curve  (\ref{uncoupledch})  in  action angle coordinates).
Take the codimension-one hypersurface
\begin{equation}\label{eq:sectionthetan}
\Sigma_{0}:\quad\sum_{n=1}^N\theta^{(n)}=0 \mod2\pi.
\end{equation}
By (\ref{eq:sumtminD}), this surface intersects the  line \(\mathcal{L}_{min}^{*}\) at
\(\theta^{(n)}=\theta^{(n)}_{min}\). By (\ref{eq:noimpactthetnD}),  in the neighborhood of this intersection point,  all \(q^{(n)}\) are bounded away from the billiard boundary.
As \(\dot \varphi=\frac{1}{N}\sum_{n=1}^N \dot \theta^{(n)}= \omega _{0}(\delta)+O( \delta^{1/2})\neq0\ \), the hyper-surface (\ref{eq:sectionthetan}) is transverse to the flows of both  Hamiltonians (\ref{eq:wperhaatbil}) and  (\ref{eq:nparttrunbillnew}).
Now we show that despite the fact that we can guarantee the closeness of these Hamiltonians to each other only away from the billiard boundary, their Poincar\'e return maps to the cross-section  (\ref{eq:sectionthetan}) are sufficiently close to each other.

\begin{lem}\label{lem:poincaremapclosebilnew} When Assumption IP3 (non-simultaneous impacts) is satisfied, the Poincar\'e return map  of  system (\ref{eq:wperhaatbil}) to the cross-section \(\sum_{n=1}^N \theta^{(n)}=0\)\ is
 $o(\delta^{1/2})$-close, along with all derivatives, to the corresponding return map of the truncated system
(\ref{eq:nparttrunbillnew}).
\end{lem}
\begin{proof}
We take several cross-sections \(\Sigma_{c_k}:\frac{1}{N}\sum_{n=1}^N\theta^{(n)}=c_{k},\quad0=c_{1}<c_{2}<\ldots<c_K=2\pi\), and consider the return map of the flow of (\ref{eq:wperhaatbil})
to \(\Sigma_{0}\) as a composition of the maps from each consecutive cross-section to the next one (intermediate Poincar\'e maps). We show that each of these
maps is well defined in the neighborhood of \(\mathcal{L}_{min}^{*}\) and is
sufficiently close to the corresponding map for (\ref{eq:nparttrunbillnew}).

In a small neighborhood of \(\mathcal{L}_{min}^{*}\), by the non-simultaneous impacts Assumption IP3, when $q^{(n)}$ is close to an impact with the billiard boundary for particle $n$, all the other particles stay away from the billiard boundary. Therefore, we can choose the cross-sections \(\Sigma_{c_k}\) in such a way that when the phase point travels between
\(\Sigma_{c_k}\) and \(\Sigma_{c_{k+1}}\),
either all the particles are bounded away from the billiard boundary, or one particle gets close to a regular impact while all the other particles stay away from the boundary. In the first case, systems (\ref{eq:wperhaatbil})  and  (\ref{eq:nparttrunbillnew}) are uniformly \(O(\delta^{3/4})\)-close in \(C^{r}\) for any $r$ in the region of the phase space between these cross-sections.
Since the flight time between the cross-sections is bounded, this implies that the corresponding intermediate Poincar\'e maps are also
\(O(\delta^{3/4})\)-close, as required.

It remains to consider the second case, when one of the particles gets close to an impact. Let it be particle $n_{0}$. Let the cross-sections before and after the impact be \(\Sigma_{c_k}\) and
\(\Sigma_{c_{k+1}}\), and let the value of $q^{(n_{0})}$ on these cross-sections stay at a distance of order
$\eta>0$
from the impact for all small $\delta$. Let us show that the map \(\Sigma_{c_k}\to \Sigma_{c_{k+1}}\) by system
(\ref{eq:wperhaatbil}) is $O(\eta\delta^{1/2})$ close, in \(C^r\) for any $r$, to the corresponding map by the
system
\begin{equation}\label{eq:wperhaatbilnnn}\begin{array}{l}\displaystyle
\hat H_{scal} =\sum_n \left[\omega(\delta)I^{(n)} +  \delta^{1/2} \frac{1}{2} {I^{(n)}}^\top\! A(\delta) I^{(n)}+
O(\delta^{3/4})\right]\\ \displaystyle \qquad\qquad
+\; \delta^{1/2}\mathop{\sum_{n,m\neq n_0}}_{n\neq m} W(\hat q(\delta^{1/2}I_0^{(n)},\theta^{(n)},\delta^{1/4}z^{(n)})-\hat q(\delta^{1/2}I_0^{(m)},\theta^{(m)},\delta^{1/4}z^{(m)}))\end{array}
\end{equation}
(i.e., we switched off the interaction with the particle $n_0$), and, similarly, the map \(\Sigma_{c_k}\to \Sigma_{c_{k+1}}\) by system (\ref{eq:nparttrunbillnew} ) is $O(\eta\delta^{1/2})$-close, in \(C^r\) for any $r$, to the corresponding map by the
system
\begin{equation}\label{eq:switchedoffnnn}
\hat H_{trun}  =\sum_n \left[\omega  (\delta )I^{(n)} + \delta^{1/2} \frac{1}{2} {I^{(n)}}^\top A  (\delta )I^{(n)})\right]
+\delta^{1/2}\mathop{\sum_{n,m\neq n_0}}_{n\neq m} W(q^*_\delta(\theta^{(n)})-q^*_\delta(\theta^{(m)}).
\end{equation}
Indeed, since the particle $n_0$ between \(\Sigma_{c_k}\) and \(\Sigma_{c_{k+1}}\) is close to impact, we can use the scaled flow-box coordinates $(\tau^{(n_{0})},\mathcal{E}^{(n_{0})},\mathcal{P}^{(n_{0})})$ for the $n_{0}$-th particle. As explained at the end of section \ref{sec:periodofL}, the transformation between the scaled $(\tau^{(n_{0})},\mathcal{E}^{(n_{0})},\mathcal{P}^{(n_{0})})$ coordinates and the scaled $(\theta^{(n_{0})},I_0^{(n_{0})}, z^{(n_{0})})$
is non-singular in the limit $\delta=0$.
We use the notation
 \((\tau=\tau^{(n_{0})},\mathcal{E}=\mathcal{E}^{(n_{0})},\mathcal{P}=\mathcal{P}^{(n_{0})},
 X = \{(I_0^{(n)},\theta^{(n)},z^{(n)}) \}_{n\neq n_{0}})\) in the layer between \(\Sigma_{c_k}\) and \(\Sigma_{c_{k+1}}\).
The Hamiltonian (\ref{eq:wperhaatbil}) in these scaled coordinates can be schematically written as
\begin{equation}\label{eq:hbl86}
H_{scal} = \mathcal{E}+H_{1}(X;\delta)+\delta^{1/2} H_{2}(\tau, \delta^{1/2}\mathcal{E}, \delta^{1/4}\mathcal{P},X;\delta),
\end{equation}
and the Hamiltonian (\ref{eq:wperhaatbilnnn}) is given by
\begin{equation}\label{eq:hbl86nn}
\hat H_{scal}  = \mathcal{E}+H_{1}(X;\delta).
\end{equation}
Here $H_1$ is the Hamiltonian of the subsystem corresponding to all the particles but $n_0$:
$$\begin{array}{l}\displaystyle
H_1=\sum_{n\neq n_0} \left[\omega(\delta)I^{(n)} +  \delta^{1/2} \frac{1}{2} {I^{(n)}}^\top\! A(\delta) I^{(n)}+
O(\delta^{3/4})\right]\\ \displaystyle \qquad\qquad
+\; \delta^{1/2}\mathop{\sum_{n,m\neq n_0}}_{n\neq m} W(\hat q(\delta^{1/2}I_0^{(n)},\theta^{(n)},\delta^{1/4}z^{(n)})-\hat q(\delta^{1/2}I_0^{(m)},\theta^{(m)},\delta^{1/4}z^{(m)})).\end{array}
$$
Since these particles stay away of impacts, $H_1$ is $C^r$ for any fixed $r$, uniformly
for all small $\delta\geqslant 0$. The term $H_2$ describes the interaction of the particle $n_0$ with the rest of the particles\footnote{Recall that  for the convenience of notation each term $W(q^{(n)}-q^{(m)})$ appears twice in the double sum in (\ref{eq:multparwsingle}).}:
$$H_2=2\sum_{m\neq n_0} W(q(\tau,\delta^{1/2}\mathcal{E},\delta^{1/4}\mathcal{P})-\hat q(\delta^{1/2}I_0^{(m)},\theta^{(m)},\delta^{1/4}z^{(m)}))$$
where we denote
$q(\tau,\delta^{1/2}\mathcal{E},\delta^{1/4}\mathcal{P})=
\hat q(\delta^{1/2}I_0^{(n_0)},\theta^{(n_0)},\delta^{1/4}z^{(n_0)})$. By Lemma \ref{lem:energytimecoor},
\(H_2\) is a $C^\infty$ function of $(\delta^{1/2}\mathcal{E},\delta^{1/4}\mathcal{P},X)$ and depends continuously, in $C^\infty$, on $\tau$ and
$\delta\geqslant 0$. Moreover, as long as $q^{(n_0)}$ stays away from the billiard boundary, $H_2$ depends smoothly on $\tau$
as well, for all $\delta\geqslant 0$.

Since \(\Sigma_{c_k}\) and \(\Sigma_{c_{k+1}}\) stay bounded away from the impact, one obtains that near these cross-sections
$$\frac{d}{dt}\tau = \partial_{\mathcal{E}} H_{scal}\   = 1+O(\delta)$$
uniformly for all $\delta\geqslant 0$, i.e., it stays bounded away from zero. Hence, hypersurfaces of constant
\(\tau\) are cross-sections to the flow of \(H_{scal}\) for all
$\delta\geqslant 0$;  the same is true for the flow of \(\hat H_{scal}\).
Thus, we choose constant-\(\tau\) cross-sections,  \(\Sigma^{\tau}_{k}\)  and
\(\Sigma^{\tau}_{k+1}\) close to \(\Sigma_{c_k}\) and,
respectively, \(\Sigma_{c_{k+1}}\). Because all these cross-sections are bounded away from impacts, the flows of \(H_{scal}\)  and \(\hat H_{scal}\) are
\(O(\delta^{1/2})\)-close   in $C^r$ in between
\(\Sigma_{c_k}\) and \(\Sigma^{\tau}_{k}\) and in between \(\Sigma^{\tau}_{k+1}\) and
 \(\Sigma_{c_{k+1}}\), for any $r$.
Since the flight time from \(\Sigma_{c_k}\) and \(\Sigma_{c_{k+1}}\) is  $O(\eta)$, the flight times between the cross-sections
\(\Sigma_{c_k}\) and \(\Sigma^{\tau}_{k}\) and between \(\Sigma^{\tau}_{k+1}\) and  \(\Sigma_{c_{k+1}}\) are  also bounded by   $O(\eta)$, so the corresponding maps for \(H_{scal}\)  and \(\hat H_{scal}\) between these cross-sections are   $O(\eta\delta^{1/2})$-close.

We are left to show that the maps from \(\Sigma^{\tau}_{k}\) to \(\Sigma^{\tau}_{k+1}\)  for system (\ref{eq:hbl86}) (\(H_{scal}\)) and system (\ref{eq:hbl86nn})
(\(\hat H_{scal}\)) on any given energy level are $O(\eta \delta^{1/2})$-close in $C^r$ for any $r$.
The equations of motion defined by (\ref{eq:hbl86}) on the energy level \(H_{scal}=E\) are given by
$$\begin{array}{l}\displaystyle
\dot \tau=1+\delta^{1/2} \partial_{\mathcal{E}} H_{2},\\
\mathcal{\dot P} =\delta^{1/2} \{\mathcal{P}, H_{2}\},\\
 \dot X =\{X,H_{1}\}+\delta^{1/2} \{X,H_{2}\},\end{array}
$$
where \(\{\cdot,\cdot\}\) are the Poisson brackets in the \((\mathcal{P},X)\) space. To have the equations in a closed form, we substitute the energy $\mathcal{E}$ of particle \((n_{0})\) by its expression which can be found from (\ref{eq:hbl86}) by the implicit function theorem:
$$\mathcal{E}=E-H_{1}(X;\delta)-
\delta^{1/2} H_2(\tau,\delta^{1/2}\mathcal{E},\delta^{1/4}\mathcal{P},X;\delta) = E-H_{1}(X;\delta)-
\delta^{1/2} H_3(\tau,\delta^{1/4}\mathcal{P},X;\delta),$$
where \(H_3\) is  continuous in \(\tau\) and $C^\infty$-smooth in \(\delta^{1/4}\mathcal{P},X\) uniformly for all $\tau$
and \(\delta\geqslant 0\).

Choosing $\tau$ as the time variable, we obtain a non-autonomous system of the form
$$\frac{d}{d\tau} \mathcal{\dot P} =O(\delta^{3/4}), \qquad \frac{d}{d\tau} X =\{X,H_{1}\}+O(\delta^{1/2}),$$
where the right-hand sides are continuous in the new time $\tau$ and $C^\infty$ in \(\mathcal{P},X\) for all small \(\delta\geqslant 0\).
The right-hand side of this system is \(O(\delta^{1/2} )\)-close with all derivatives  with respect to
\(\mathcal{P},X\), to
$$\frac{d}{d\tau} \mathcal{\dot P} =0, \qquad \frac{d}{d\tau} X =\{X,H_{1}\},$$
which is the system (\ref{eq:hbl86nn}) restricted to any constant energy level. The constant-\(\Delta\tau\) maps of these systems are therefore $O(\Delta\tau\delta^{1/2})$-close in \(C^r\) for any \(r\). This is what we need since \(\Delta\tau\), the flight time from  \(\Sigma^{\tau}_{k}\) to \(\Sigma^{\tau}_{k+1}\), is  \(O(\eta) \).

This proves that the Poincar\'e map \(\Sigma_{c_k}\to \Sigma_{c_{k+1}}\) for the Hamiltonian
\(H_{scal}\)   of
(\ref{eq:wperhaatbil}) is $O(\eta\delta^{1/2})$-close, in \(C^r\) for any $r$, to the corresponding map for the Hamiltonian  \(\hat H_{scal}\) of (\ref{eq:wperhaatbilnnn}).

Next we show that the corresponding maps for systems (\ref{eq:nparttrunbillnew})
and (\ref{eq:switchedoffnnn}) are also $O(\eta\delta^{1/2})$-close in \(C^r\) for any $r$. Here, we do not use the flow-box coordinates and instead of showing the $O(\eta\delta^{1/2})$-closeness of the Poincar\'e maps between constant-$\tau$
cross-sections we show the $O(\eta\delta^{1/2})$-closeness of the Poincar\'e maps between
constant-$\theta^{(n_0)}$
cross-sections, \(\Sigma^{\theta}_{k}\)  before impact and  \(\Sigma^{\theta}_{k+1}\) after impact. Recall that only the particle \(n_{0}\) is near impact for the flow between  \(\Sigma_{c_k}\to \Sigma_{c_{k+1}}\).
Hence the flows of   (\ref{eq:nparttrunbillnew})
and (\ref{eq:switchedoffnnn}) between  \(\Sigma_{c_k}\) and \(\Sigma^{\theta}_{k}\) and between
\(\Sigma^{\theta}_{k+1}\) and  \(\Sigma_{c_{k+1}}\) are \(O(\delta^{1/2})\)-close, with derivatives, for all
$\delta\geqslant 0$. Since the flight time between these cross-sections is bounded by \(\eta\), the difference between the corresponding Poincar\'e maps for systems  (\ref{eq:nparttrunbillnew})
and (\ref{eq:switchedoffnnn})  is  $O(\eta\delta^{1/2})$ small. So we need to examine only the flow between \(\Sigma^{\theta}_{k}\)  and  \(\Sigma^{\theta}_{k+1}\).   As explained in Section \ref{sec:periodofL}, the terms \(W(q^*_\delta(\theta^{(n)})-q^*_\delta(\theta^{(m)}))\) in $H_{trun}$ of (\ref{eq:nparttrunbillnew}) and $\hat H_{trun}$ of (\ref{eq:switchedoffnnn}) are  $C^\infty$ functions of all the variables  if \(n,m\neq n_0\). The terms
\(W(q^*_\delta(\theta^{(n_{0})})-q^*_\delta(\theta^{(m)}))\) with \(m\neq n_0\) in (\ref{eq:nparttrunbillnew}) are $C^\infty$ functions of \(\theta^{(m)}\)  uniformly  for all $\theta^{(n_0)}$ and
$\delta\geqslant 0$. Restricting the flow to a constant energy level and choosing
$\theta^{(n_0)}$ as a new time variable, the right-hand sides of the corresponding non-autonomous systems obtained from
(\ref{eq:nparttrunbillnew}) and (\ref{eq:switchedoffnnn}) are $O(\delta^{1/2})$-close with derivatives with all the variables
but $\theta^{(n_0)}$. This implies the $O(\eta\delta^{1/2})$-closeness of their constant-$\theta^{(n_0)}$ maps, which implies
the required $O(\eta\delta^{1/2})$-closeness of the maps \(\Sigma_{c_k}\to \Sigma_{c_{k+1}}\) for the Hamiltonians $H_{trun}$  and $\hat H_{trun}$.

Because there is no dependence on the $q$-coordinates of the particle $n_0$ in equations (\ref{eq:wperhaatbilnnn}) and
(\ref{eq:switchedoffnnn}),
and this is the only source of singularity at $\delta=0$ (all the other particles are away from the impacts), the Hamiltonians $\hat H_{scal}$ and $\hat H_{trun}$ are $O(\delta^{3/4})$-close in $C^r$ for any $r$. Therefore, their Poincar\'e maps are also
$O(\delta^{3/4})$-close. This, finally, proves that the maps \(\Sigma_{c_k}\to \Sigma_{c_{k+1}}\) by system
(\ref{eq:wperhaatbil}) and system (\ref{eq:nparttrunbillnew}) are $O(\eta\delta^{1/2})$-close.

Taking the composition of the Poincar\'e maps from each consecutive cross-section to the next one, we obtain that
the return maps to \(\Sigma_{0}\) for systems (\ref{eq:wperhaatbil}) and (\ref{eq:nparttrunbillnew}) are
$O(\eta\delta^{1/2})$-close, in \(C^r\) for any $r$. Since $\eta$ can be taken as small as we want, we can also allow it to go
sufficiently slowly to zero as $\delta\to 0$, thus completing the proof of the lemma.
\end{proof}

We now evaluate the Poincar\'e map for the truncated system (\ref{eq:nparttrunbillnew}).
We average the truncated Hamiltonian \(H_{trun }\) with respect to the motion along the periodic orbit
\(L^*_\delta\) to establish an analogue of Lemma \ref{Lemma1.1}, but, due to the loss of smoothness at impacts, the proof is different -- we utilize explicitly the pairwise structure of the interaction potential instead of using the Fourier expansion of Lemma \ref{Lemma1.1}.

First, we introduce the following terminology:

\begin{defn}\label{weaklyreg} A function $G$ of
\(\left((I_0^{(1)},\theta^{(1)},z^{(1)}), \ldots, (I_0^{(N)},\theta^{(N)},z^{(N)})\right)\) is called {\em weakly regular} if it is $C^\infty$ at $\delta>0$ and satisfies the following conditions (i) and (ii) in the limit
$\delta\to 0$:

(i) When all $q^{(n)}(I_0^{(n)},\theta^{(n)},z^{(n)})$ stay bounded away from the billiard boundary, the function $G$ is of class $C^\infty$ for all $\delta\geqslant\ 0$ and depends continuously, in $C^\infty$, on
$\delta\geqslant 0$;

(ii) When exactly one of the particles is in the billiard boundary layer, i.e., for some $n_0$,
$q^{(n_{0})}(I_0^{(n_{0})},\theta^{(n_{0})},z^{(n_{0})})$ is near an impact point,
the function $G$ is $C^\infty$ with respect to
$(I^{(n_{0})}, z^{(n_{0})})$ and $(\{I_0^{(n)},\theta^{(n)},z^{(n)}\}_{n\neq n_{0}})$ for each
fixed $\theta^{(n_{0})}$ and all $\delta\geqslant 0$, and it depends continuously, as a $C^\infty$ function of
$(I^{(n_{0})}, z^{(n_{0})})$ and $(\{I_0^{(n)},\theta^{(n)},z^{(n)}\}_{n\neq n_{0}})$,
on $\theta^{(n_{0})}$ and $\delta\geqslant 0$.
\end{defn}

In particular, since we showed that if the particle \(n_{0}\) is close to impact \(q^*_\delta(\theta^{(n_{0})})\)  depends continuously on \(\theta^{(n_{0})},\)  the functions $W(q^*_\delta(\theta^{(n)})-q^*_\delta(\theta^{(m)}))$
in (\ref{eq:nparttrunbillnew}) are weakly regular. Note that we do not consider the case when there are more than
one particle in the boundary layer because it is not possible by the non-simultaneous impacts assumption IP3.

\begin{lem}\label{Lemmaavbil}
There exists a weakly regular, symplectic change of coordinates which brings the truncated Hamiltonian (\ref{eq:nparttrunbillnew}) to the form
\begin{equation}\label{eq:npartavbill}\begin{array}{l}\displaystyle
H =\sum_{n=1}^N \omega(\delta) I^{(n)} + \delta^{1/2} \sum_{n=1}^N \frac{1}{2} I^{(n)} A(\delta) I^{(n)}+
\delta^{1/2} U(\theta^{(1)},\dots, \theta^{(N)};\delta)+\delta^{3/4} G. \end{array}
\end{equation}
Here \( G\) is a weakly regular function, and the averaged potential is given by:
\begin{equation}\label{eq:averagedpotub}
U(\theta^{(1)}, \dots,\theta^{(N)};\delta)=\mathop{\sum_{n,m=1,\dots,N}}_{n\neq m}
 W_{avg}(\theta^{(n)}-\theta^{(m)};\delta),
\end{equation}
where
\begin{equation}\label{eq:averagedwb}
 W_{avg}(\theta^{(n)}-\theta^{(m)};\delta)=
\frac{1}{2\pi}\int_0^{2\pi} W(q^*_\delta(s+\theta^{(n)})-q^*_\delta(s+\theta^{(m)}))ds.
\end{equation}
All the derivatives of  \(W_{avg}\) with respect to  \(\theta\)  depend continuously on
\(\delta\) for small \(\delta\geqslant0.\)
\end{lem}
\begin{proof}
Let  \((u,v)\ =(\frac{1}{2}(\theta^{(n)}+\theta^{(m)}),\frac{1}{2}(\theta^{(n)}-\theta^{(m)}))\).
Define
\begin{equation}\label{eq:psinmdefbil}
\Psi_0(u,v;u_0,\delta)= \frac{1}{\omega_0(\delta)}\left[
\int_{u_{0}}^{u}W(q^*_\delta(s+v)-q^*_\delta(s-v))  ds-(u-u_{0})W_{avg}(2v;\delta)\right],
\end{equation}
so \(\Psi_0(u,-v;u_0,\delta)=\Psi_0(u,v;u_0,\delta)\). It follows from (\ref{eq:averagedwb}) and the \(2\pi\)-periodicity of
\(q^*_\delta\) that
\begin{equation}\label{psip2p}
\Psi_0(u+2\pi,v;u_0,\delta)=\Psi_0(u,v;u_0,\delta).
\end{equation}
Let
\begin{equation}\label{eq:Psiall}
\Psi(\theta^{(1)}, \ldots,\theta^{(N)};u_{0},\delta)=\sum_{m_{1}\neq m_2}\Psi_0(\frac{1}{2}(\theta^{(m_{1})}+\theta^{(m_{2})}),\frac{1}{2}(\theta^{(m_{1})}-\theta^{(m_{2})} );u_{0},\delta).
\end{equation}
For positive \(\delta\), it is identical to the function  (\ref{eq:Psismoothdef}) of Section \ref{sec:proofmainthm}. Consider the symplectic coordinate change (its smoothness properties in the limit
\(\delta\rightarrow0\) are discussed below)
\begin{equation}\label{subst1bil}
I_{0}^{(n)} \rightarrow  I_{0}^{(n)} - \delta^{1/2}\partial_{\theta^{(n)}}\Psi, \qquad n=1,\dots, N.
\end{equation}
By (\ref{psip2p}), the right-hand side does not change when we add $2\pi$ to all phases $\theta^{(1)},\dots, \theta^{(N)}$,
i.e., this is indeed a well-defined coordinate transformation in a neighborhood of \(\mathcal{L}_{min}^{*}\).
Differentiating (\ref{eq:psinmdefbil}), we obtain
$$\omega_{0}(\delta)  \partial_{u}\Psi  =\omega_{0}(\delta)  \partial_{\theta^{(n)}}\Psi+\omega_{0}(\delta) \partial_{\theta^{(m)}}\Psi= W(q^*_\delta(\theta^{(n)})-q^*_\delta(\theta^{(m)}))- W_{avg}(\theta^{(n)}-\theta^{(m)};\delta).$$
As \(\partial_{\theta^{(n)}}\Psi=  \sum_{m\neq n}\partial_{u}\Psi_0(\frac{1}{2}(\theta^{(n)}+\theta^{(m)}),\frac{1}{2}(\theta^{(n)}-\theta^{(m)} );u_0,\delta) \),
the transformation (\ref{subst1bil}) makes
$$
 \sum_{n=1}^N \omega_{0}(\delta) I_{0}^{(n)}\rightarrow\sum_{n=1}^N \omega_{0}(\delta)( I_{0}^{(n)}- \delta^{1/2}  \sum_{m\neq n}\partial_{u}\Psi_0(\frac{1}{2}(\theta^{(n)}+\theta^{(m)}),\frac{1}{2}(\theta^{(n)}-\theta^{(m)} );u_0,\delta)
$$
so it brings the Hamiltonian (\ref{eq:nparttrunbillnew}) to the required form (\ref{eq:npartavbill}).

We need to show that one can choose $u_0$ such that the expression
$$\partial_{\theta^{(n)}} \Psi_0(\frac{1}{2}(\theta^{(n)}+\theta^{(m)}),\frac{1}{2}(\theta^{(n)}-\theta^{(m)}); u_0,\delta)$$
in the formula for the coordinate transformation (\ref{subst1bil}) is weakly regular for all $n$ and $m$.
Let
$$ \Psi_{1}(u,v;u_0,\delta)=\int_{u_{0}}^{u}W(q^*_\delta(s+v)-q^*_\delta(s-v))  ds.$$

\begin{lem}\label{lem:psi1}
 For  \(u_{0}\) chosen such that \(({u_{0}}\pm\frac{1}{2}(\theta_{min}^{(n)}-\theta_{min}^{(m)}))/\omega_{0}\) are bounded away from the impact moments \(t^{j}\) for all $m$, $n$, and $j$, the function
 \(\Psi_{1}(\frac{\theta^{(n)}+\theta^{(m)}}{2},\frac{\theta^{(n)}-\theta^{(m)}}{2};u_0,\delta)\) and
 its derivatives with respect to \(\theta^{(n)},\theta^{(m)}\) are weakly regular.
\end{lem}
\begin{proof}
Let us divide the interval of integration to subintervals \([u_{j},u_{j+1}],j=0,\ldots,J\) such that the following holds:
the last interval is \([u_J,u_{J+1}=\frac{\theta^{(n)}+\theta^{(m)}}{2}]\);
 on each sub-interval at most one of the particles is in the boundary layer (i.e., for all \(s\) in each sub-interval, either \(q^*_\delta(s+\frac{1}{2}(\theta^{(n)}-\theta^{(m)}))\), or
\(q^*_\delta(s-\frac{1}{2}(\theta^{(n)}-\theta^{(m)})) \), or both are outside the boundary layer);
and all end points but the last one are at a finite distance from the impact points: \(({u_{j}}\pm\frac{1}{2}(\theta^{(n)}_{min}-\theta^{(m)}_{min}))/\omega_{0}\) are bounded away from the impact moments for \(j=0,\ldots,J\).

Such choice of sufficiently small intervals is possible: because we consider $\theta^{(n)}$ and $\theta^{(m)}$ close to
$\theta^{(n)}_{min}$ and $\theta^{(m)}_{min}$, the non-simultaneous impacts assumption IP3 implies that
$\theta^{(n)}-\theta^{(m)}$ can not get close to $\omega_0(t^j-t^k)$ for $n\neq m$, hence it is impossible that
$s_1+\frac{1}{2}(\theta^{(n)}-\theta^{(m)})$  and  $s_2-\frac{1}{2}(\theta^{(n)}-\theta^{(m)})$
get simultaneously close to the impact phases $\omega_0t^j$ and $\omega_0 t^k$ if $s_1$ and $s_2$  belong to the
same small interval.

By construction,
\begin{equation}\label{eq:95}
\Psi_1(\frac{\theta^{(n)}+\theta^{(m)}}{2},\frac{\theta^{(n)}-\theta^{(m)}}{2};u_0,\delta)=\sum_{j=0}^J\Psi_{1}(u_{j+1},\frac{\theta^{(n)}-\theta^{(m)}}{2};u_j,\delta).
\end{equation}

For all \(j<J\),  we show that \(\Psi_{1}(u_{j+1},\frac{\theta^{(n)}-\theta^{(m)}}{2};u_j,\delta)\) is $C^\infty$ for all small
$\delta\geqslant 0$ (hence it is weakly regular, and its derivatives are weakly regular as well). Indeed, if both \(q^*_\delta(s\pm \frac{\theta^{(n)}-\theta^{(m)}}{2})\) are not in the  boundary layer,
for all $s$ from the integration interval $[u_j,u_{j+1}]$, then the integrand $W(q^*_\delta(s+\frac{\theta^{(n)}-\theta^{(m)}}{2})-q^*_\delta(s-\frac{\theta^{(n)}-\theta^{(m)}}{2};\delta))$ is a $C^\infty$ function for all $\delta\geqslant0$ and the claim follows.
If, say, the first term, \(q^*_\delta(s+\frac{\theta^{(n)}-\theta^{(m)}}{2})\) is in the boundary layer for some value of
\(s\in(u_{j},u_{j+1})\), then we shift the integration interval by \(\frac{1}{2}(\theta^{(n)}-\theta^{(m)})\) to establish  \begin{equation}\label{eq:94}
\Psi_{1}(u_{j+1},\frac{\theta^{(n)}-\theta^{(m)}}{2};u_j,\delta)=
\int_{u_{j}+\frac{1}{2}(\theta^{(n)}-\theta^{(m)})}^{u_{j+1}+\frac{1}{2}(\theta^{(n)}-\theta^{(m)})}W(q^*_\delta(s')-q^*_\delta(s'-(\theta^{(n)}-\theta^{(m)}))) ds'.
\end{equation}
The term \(q^*_\delta(s'-(\theta^{(n)}-\theta^{(m)}))\) is away from the billiard boundary for the integration interval,
\(q^*_\delta(s')\) is away from the boundary layer at the limits of integration. Therefore, the integrand
$W(q^*_\delta(s')-q^*_\delta(s'-(\theta^{(n)}-\theta^{(m)})))$ is a
\(C^{\infty}\)-function of $\theta^{(n)}$ and $\theta^{(m)}$, uniformly continuous in $C^\infty$ for all $s'$ and all
$\delta\geqslant 0$. Moreover, it is $C^\infty$ in $s$ also near the limits of integration. It immediately follows that $\Psi_1$
given by (\ref{eq:94}) is \(C^{\infty}\) as required. Similarly, if  the second term, \(q^{*}(s-\frac{\theta^{(n)}-\theta^{(m)}}{2})\) is near an impact, we shift the integration interval by  \(-\frac{1}{2}(\theta^{(n)}-\theta^{(m)}) \) and establish the same smoothness result.

For the last segment, since \(u_{J+1}=\frac{\theta^{(n)}+\theta^{(m)}}{2}\), and \(v=\frac{\theta^{(n)}-\theta^{(m)}}{2} \), if both \(\theta^{(n)}/\omega_{0}\)  and \(\theta^{(m)}/\omega_{0}\)  are bounded away from the impact moments, the same arguments as above show that \(\Psi_{1}(\frac{\theta^{(n)}+\theta^{(m)}}{2},\frac{1}{2}(\theta^{(n)}-\theta^{(m)});u_J,\delta)\)  are smooth as required. On the other hand, if, say, the particle \(n\) is in the boundary layer (and hence the particle $m$ is not in the boundary layer) we write the integral in the form of (\ref{eq:94}):
\begin{equation}\label{eq:96}
\Psi_{1}(u_{J+1},\frac{\theta^{(n)}-\theta^{(m)}}{2};u_J,\delta)=
\int_{u_{J}+\frac{1}{2}(\theta^{(n)}-\theta^{(m)})}^{\theta^{(n)}}W(q^*_\delta(s')-q^*_\delta(s'-(\theta^{(n)}-\theta^{(m)})))ds'.
\end{equation}
As above, we have that $W(q^*_\delta(s')-q^*_\delta(s'-(\theta^{(n)}-\theta^{(m)})))$  is a
\(C^{\infty}\)-function of $\theta^{(n)}$ and $\theta^{(m)}$, uniformly continuous in $C^\infty$ for all $s'$ and all
$\delta\geqslant 0$. Moreover, it is $C^\infty$ in $s$ also near the lower limit of integration. The upper limit of integration does not depend on $\theta^{(m)}$, so we conclude that the integral is $C^\infty$ with respect to $\theta^{(m)}$, i.e., $\Psi_1$ and
$\partial_{\theta^{(m)}}\Psi_1$ given by (\ref{eq:96}) are weakly regular, and
$$\begin{array}{l}
\partial_{\theta^{(n)}}\Psi_1(u_{J+1},\frac{\theta^{(n)}-\theta^{(m)}}{2};u_J,\delta)  =W(q^*_\delta(\theta^{(n)})-q^*_\delta(\theta^{(m)}))\\
\qquad \qquad\qquad- \frac{1}{2}
W(q^*_\delta(u_J\!+\!\frac{1}{2}(\theta^{(n)}\!-\!\theta^{(m)}))\!-\!q^*_\delta(u_J\!-
\!\frac{1}{2}(\theta^{(n)}\!-\!\theta^{(m)}))) \\
\qquad\qquad\qquad-
\int_{u_{J}+\frac{1}{2}(\theta^{(n)}-\theta^{(m)})}^{\theta^{(n)}}W^\prime(q^*_\delta(s')-q^*_\delta(s'-(\theta^{(n)}-\theta^{(m)}))) \partial_{\theta^{(n)}} q^*_\delta(s'-(\theta^{(n)}-\theta^{(m)}))ds'
\end{array}$$
is also weakly regular since \(s'-(\theta^{(n)}-\theta^{(m)})\) is away from impact along this integration interval.
The similar weakly regularity results hold true if the particle \(m\), and not the particle $n$, is in the boundary layer.

Thus, all the terms in (\ref{eq:95}) are weakly regular, along with the derivatives with respect to $\theta^{(n)}$ and
$\theta^{(n)}$, which gives the lemma.
\end{proof}

Note that by the periodicity of $q^*_\delta$,
\begin{equation}\label{eq:wavgispsi1}
W_{avg}(\theta^{(n)}-\theta^{(m)};\delta) =
 \frac{1}{2\pi}\Psi_1(u_0+2\pi, u_0, \frac{\theta^{(n)}-\theta^{(m)}}{2};\delta),
\end{equation}
 and so the right-hand side  does not depend on $u_0$. In particular we can choose \(u_{0}\)  as needed for  Lemma \ref{lem:psi1},  so the average potential $W_{avg}$ is weakly regular.  Since we can always shift both \(\theta^{(n)}\) and \(\theta^{(m)}\) away from the impacts, the weak regularity of the average potential $W_{avg}$ means that it is
is $C^\infty$ for all small $\delta\geqslant\ 0$.

Now, by Lemma \ref{lem:psi1}, we obtain that since
$$\Psi_0(u,v;u_{0},\delta)=\frac{1}{\omega_0} \Psi_1(u,u_0,v;\delta) - \frac{u-u_0}{\omega_{0}} W_{avg}(2v;\delta),$$
the  transformation (\ref{subst1bil})  is also weakly regular.  \end{proof}

Like in Lemma \ref{lem:poincaremapclosebilnew}, omitting the weakly regular $O(\delta^{3/4})$-term in the Hamiltonian
 (\ref{eq:npartavbill}) results only in $O(\delta^{3/4})$ corrections to the return map to an interior cross-section near
 \(\mathcal{L}_{min}^{*}\).
Thus, the return map for the truncated averaged Hamiltonian
\begin{equation}\label{eq:nparttrunbill}\begin{array}{l}\displaystyle
H (I,\theta)=\sum_{n=1}^N \omega(\delta) I^{(n)} + \delta^{1/2} \sum_{n=1}^N \frac{1}{2} I^{(n)} A(\delta) I^{(n)}+
\delta^{1/2} U(\theta^{(1)},\dots, \theta^{(N)};\delta) \end{array}
\end{equation}
near \(\mathcal{L}_{min}^{*}\) is $o(\delta^{1/2})$-close the return map for the scaled Hamiltonian (\ref{eq:wperhaatbil}).

This Hamiltonian has the same form as the truncated Hamiltonian of (\ref{eq:npartav}). The parameters  $\omega(\delta)$,
$A(\delta)$, and the averaged potential \(U(\theta^{(1)},\dots, \theta^{(N)};\delta)\) depend continuously on $\delta$ and
satisfy, for all small $\delta\geqslant 0$,
the non-degeneracy assumptions as in Theorem \ref{thm:mainbounded}.  So we finish the proof of Theorem \ref{thm:billiardsystem}
in the same way as in Theorem \ref{thm:mainbounded}. Namely,
applying Lemma \ref{lem:poinc1}, we find that the Poincar\'e return map for system (\ref{eq:nparttrunbill}) is $O(\delta^{3/4})$-close, with all derivatives, to the time-\(\frac{2\pi}{\omega_0}\) map for system (\ref{hamtr}).
Hence, the return map for (\ref{eq:wperhaatbil}) is $o(\delta^{1/2})$-close to the time-\(\frac{2\pi}{\omega_0}\) map for system (\ref{hamtr}). Now, the same arguments as in Lemma \ref{Lemma1.2} show that if we make $O(\delta^{-1/2})$ iterations  of the Poincar\'e map of (\ref{eq:wperhaatbil}), the result is $o(1)_{\delta\to 0}$-close to
the time-1 map for the Hamiltonian (\ref{eq:timeoneham}). The latter map has, by Lemma \ref{lem:positivekam}, a positive
measure set of invariant KAM tori, hence so does the rescaled system (\ref{eq:wperhaatbil}), as well as the original system (\ref{eq:multparwsingle}).

\section{\label{sec:boxsection}Gas in a rectangular box}

Here we study the motion near the  family of fast vertical periodic orbits in a box, proving Theorem \ref{thm:billiardbox1} (the non-simultaneous impacts case, see Sections  \ref{sec:aabox},   \ref{sec:asynchproof} and \ref{sec:proofsmoothpot}) and Theorem  \ref{thm:billiardbox2} (the simultaneous impacts case, Sections   \ref{sec:aabox},  \ref{sec:proofsmoothpot} and \ref{sec:synchronproof}).

\subsection{\label{sec:aabox}Action-angle coordinates near vertical periodic orbits.}
Since the single-particle degrees of freedom decouple, and only the vertical direction dynamics are fast and billiard-like, the single-particle theory developed in Section \ref{sec:singlepartbil} applies to the one-dimensional vertical motion,  where $q=q_{d}\in R^1$ and the billiard corresponds to
a particle bouncing between the end points of the interval \([0,\pi]\). At positive $\delta$, the vertical motion of a single particle is described by the Hamiltonian
\begin{equation}\label{hsdofsp}
H_d=\frac{p^2}{2} + \delta V_d(q) = \frac{p^2}{2} + \frac{\delta}{Q_d(q)^\alpha},
\end{equation}
see (\ref{eq:vinbox}), (\ref{eq:potentialboundarylayerbox}). One can introduce action-angle variables
$(I,\theta)$ for this system. The periodic orbit that lies in the energy level $H=\frac{1}{2}$
is denoted as $q=q^*_\delta(\theta)$;  it tends to the saw-tooth (\ref{eq:potsawtooth}) as $\delta\to 0$.
The vertical action $I$ is a one-to-one, smooth function of the  vertical energy and is a constant of motion. Thus, the single-particle vertical motion is governed by
\begin{equation}\label{eq:H1ofIdlta}
H_{d}(I;\delta)=\frac{(\hat p_{d}(I,\theta;\delta))^2}{2}+ \delta V_{d}(\hat q_{d}(I,\theta;\delta))-\frac{1}{2}=\omega_{0}(\delta )I+\frac{1}{2}a(\delta )I^2+O(I^3), \qquad I\in\mathbb{R}^1.
\end{equation}
This is analogous to formula (\ref{eq:birkhoffsinglebilliard}) of
Section \ref{sec:singlepartbil} but, contrary to the multidimensional single-particle theory, there are no $z$-variables nor the corresponding $\mathcal{P}$-variables.

Using the action-angle coordinates in the vertical direction, the multi-particle Hamiltonian (\ref{eq:multparsysybox}) takes the following form:
\begin{equation}\label{orih}
\begin{split}
H&=\sum^N_{n=1 }H_d(I^{(n)};\delta)+
\delta\sum_{n=1,\dots, N}\sum_{i=1}^{d-1}\left[\frac{(p_{\xi ,i}^{(n)})^2}{2}+ V_{i}(\xi_{i}^{(n)})\right]\\
& \qquad +
\delta\mathop{\sum_{n,m=1,\dots,N}}_{n\neq m} \ W(\hat q_{d}(I^{(n)},\theta^{(n)} ;\delta)-\hat q_{d}(I^{(m)},\theta^{(m)} ;\delta),\xi^{(n)}-\xi^{(m)}));
\end{split}
\end{equation}
recall that\( (\xi,p_{\xi})\) denote the non-vertical coordinates:
\(\xi_{i}^{(n)}=q_{i}^{(n)},p_{\xi ,i}^{(n)}=p_{i}^{(n)}, i=1,\dots,d-1,\ n=1,\dots N\).

Scaling the  vertical actions as  \(\delta^{1/2}I\) and dividing the Hamiltonian by  \(\delta^{1/2}\), we obtain
\begin{equation}\label{eq:Hrectexpanded}
\begin{split}
H&_{scal}=\sum^N_{n=1 }\omega_{0}(\delta )I^{(n)}+\delta^{1/2}\sum^N_{n=1 }\left(\frac{1}{2}a(\delta )(I^{(n)})^2 +\sum_{i=1 }^{d-1}\left[\frac{(p_{\xi ,i}^{(n)})^2}{2}+ V_{i}(\xi_{i}^{(n)})\right]\right)
+O(\delta)\\
& \qquad +\delta^{1/2}
\mathop{\sum_{n,m=1,\dots,N}}_{n\neq m} \ W(\hat q_{d}(\delta^{1/2} I^{(n)},\theta^{(n)} ;\delta)-\hat q_{d}(\delta^{1/2} I^{(m)},\theta^{(m)} ;\delta),\xi^{(n)}-\xi^{(m)})).
\end{split}\end{equation}

Below, we establish the existence of an elliptic periodic orbit for system (\ref{eq:Hrectexpanded}) by comparing it with the averaged system defined by the Hamiltonian
\begin{equation}\label{eq:Hrectavrg}
\begin{split}
H_{avg}&=\sum^N_{n=1 }\omega_{0}(\delta )I^{(n)}+\delta^{1/2}\sum^N_{n=1 }\left(\frac{1}{2}a(\delta )(I^{(n)})^2 +\sum_{i=1 }^{d-1}\left[\frac{(p_{\xi ,i}^{(n)})^2}{2}+ V_{i}(\xi_{i}^{(n)})\right]\right)\\
& \qquad +\delta^{1/2}
\mathop{\sum_{n,m=1,\dots,N}}_{n\neq m} W_{avg}(\theta^{(n)} -\theta^{(m)} ,\xi^{(n)}-\xi^{(m)};\delta)+\delta \tilde  G(\theta,I,\xi,p;\delta).
\end{split}\end{equation}
This is an analogue of the averaged system (\ref{eq:npartavbill}) (just the higher order terms are of order \(\delta\) and not of order \(\delta^{3/4}\)  as in (\ref{eq:npartavbill})). Here
\(W_{avg}\) is the pairwise interaction potential  averaged over the vertical oscillations in the system where
the interaction between particles is switched off (see (\ref{orih})):
$$
H=\sum^N_{n=1 }H_d(I^{(n)};\delta)+
\delta\sum_{n=1,\dots, N}\sum_{i=1}^{d-1}\left[\frac{(p_{\xi ,i}^{(n)})^2}{2}+ V_{i}(\xi_{i}^{(n)})\right].$$
The vertical oscillations at $I^{(n)}=0$ correspond to the choreographic solution
(cf. (\ref{uncoupledch})): \begin{equation}\label{uncoupledchbox}
\mathbf{L}^{*}(\theta,\xi)=\{q_{1,\ldots d-1}^{(n)}=\xi_{i}^{(n)},q^{(n)}_{d}=q_\delta^{*}(\omega _{0}t+\theta^{(n)}), \;\; p_{1,\ldots d-1}^{(n)}=0,p_{d}^{(n)}=p_{\delta}^*(\omega _{0}t+\theta^{(n)})\}^{N}_{n=1}.
\end{equation}

 Thus,
\begin{equation}\label{eq:averagedpotboxdelta}
 W_{avg}(\theta^{(n)} -\theta^{(m)} ,\xi^{(n)}-\xi^{(m)};\delta)=\frac{1}{2\pi}\int_0^{2\pi} W(q^{*}_{\delta}(s+\theta^{(n)})-q^{*}_{\delta}(s+\theta^{(m)}),\xi^{(n)}-\xi^{(m)})ds.
\end{equation}
We denote the potential  of the averaged system (\ref{eq:Hrectavrg})
by
\begin{equation}\label{eq:averagedparalpotential}
U_\delta(\theta,\xi)=
\mathop{\sum_{n,m=1,\dots,N}}_{n\neq m} W_{avg}(\theta^{(n)} -\theta^{(m)} ,\xi^{(n)}-\xi^{(m)};\delta)+\sum^N_{n=1 }\sum_{i=1}^{d-1}  V_{i}(\xi_{i}^{(n)}),
\end{equation}
so \(U_0(\theta,\xi)=U(\theta,\xi)\) of (\ref{eq:averagedpotuD}).
The regularity properties of the potential \(U_\delta\) and of the correction term $\delta G$ in (\ref{eq:Hrectavrg}) in the limit
$\delta\to 0$, and their influence on the dynamics, are evaluated differently in the case of non-simultaneous and simultaneous impacts.

\subsection{\label{sec:asynchproof}The non-simultaneous impacts case}
The proof of Theorem \ref{thm:billiardbox1} is, essentially, the same as for Theorem \ref{thm:billiardsystem} - we just do not have here the fast variables $z$ but, instead, have slow variables $\xi$ (which, in fact, makes the situation simpler). As in Lemma \ref{lem:poincaremapclosebilnew} of Section \ref{sec:billiardmultipart}, in the case where the impacts are non-simultaneous, the Poincar\'e map for the Hamiltonian (\ref{eq:Hrectexpanded}) is
\(o(\delta^{1/2})\)-close to the  Poincar\'e map for the truncated Hamiltonian  \(H_{trun}\):
\begin{equation}\label{eq:Hrectexpandedtruncated}
\begin{split}
H_{trun}&=\sum^N_{n=1 }\omega_{0}(\delta )I^{(n)}+\delta^{1/2}\sum^N_{n=1 }\left(\frac{1}{2}a(\delta )(I^{(n)})^2 +\sum_{i=1 }^{d-1}\left[\frac{(p_{\xi ,i}^{(n)})^2}{2}+ V_{i}(\xi_{i}^{(n)})\right]\right)\\
& \qquad +\delta^{1/2}
\mathop{\sum_{n,m=1,\dots,N}}_{n\neq m} W( q^*_\delta(\theta^{(n)} )- q^*_\delta(\theta^{(m)} ),\xi^{(n)}-\xi^{(m)}).
\end{split}\end{equation}
As in Lemma \ref{Lemmaavbil}, we  average the truncated Hamiltonian \(H_{trun}\) by performing the symplectic coordinate  transformation defined by the \(\delta^{1/2}\)-time map for the Hamiltonian
\(\Psi\) defined as in (\ref{eq:Psiall}), where  \(\Psi_{0}\) is defined as in (\ref{eq:psinmdefbil}), i.e.,
\begin{equation}\label{eq:Psiall00}
\Psi(\theta^{(1)}, \ldots,\theta^{(N)},\xi^{(1)},\ldots, \xi^{(N)};u_{0},\delta)=\sum_{m_{1}\neq m_2}\Psi_0(\frac{1}{2}(\theta^{(m_{1})}+\theta^{(m_{2})}),\frac{1}{2}(\theta^{(m_{1})}-\theta^{(m_{2})}), \xi^{(m_1)}-\xi^{(m_2)};u_{0},\delta),
\end{equation}
where
$$\Psi_0(u,v,w;u_0,\delta)= \frac{1}{\omega_0(\delta)}\left[
\int_{u_{0}}^{u}W(q^*_\delta(s+v)-q^*_\delta(s-v),w;\delta)  ds-(u-u_{0})W_{avg}(2v,w;\delta)\right].$$
The only difference with Section \ref{sec:billiardmultipart} is that now $\Psi$ also depends on the $\xi$ variables.
The same computations as in Lemma \ref{Lemmaavbil} show that the resulting near-identity transformation
\((I^{(n)},p_{\xi}^{(n)})\rightarrow (I^{(n)},p_{\xi}^{(n)})-\delta^{1/2}(\partial_{\theta^{(n)}},\partial_{\xi^{(n)}})\Psi\)
is weakly regular and brings the truncated Hamiltonian (\ref{eq:Hrectexpandedtruncated}) to the averaged form (\ref{eq:Hrectavrg}), where the error term $\delta G$ is weakly regular (see Definition \ref{weaklyreg}); the potential $W_{avg}$ is a $C^\infty$ function and depends continuously on $\delta$, along with all derivatives, for al
$\delta\geq 0$.

Like in Theorem \ref{thm:billiardsystem}, since no more than one particle can be near the box boundary at any given moment of time (this is the non-simultaneous impacts assumption), the weak regularity of $G$ implies that omitting
the $\delta G$ terms in the averaged system (\ref{eq:Hrectavrg}) results only in $O(\delta)$-corrections to the Poincar\'e return map. Thus, we obtain that the Poincar\'e return map of the original system (\ref{eq:Hrectexpanded}) is
\(o(\delta^{1/2})\)-close to the Poincar\'e return map for the system
\begin{equation}\label{eq:averagebox}
\begin{split}
H &=\sum^N_{n=1 }\omega_{0}(\delta )I^{(n)}+\delta^{1/2}\sum^N_{n=1 }\left(\frac{1}{2}a(\delta )(I^{(n)})^2 +\sum_{i=1 }^{d-1}\frac{(p_{\xi ,i}^{(n)})^2}{2}\right)+\delta^{1/2} U_\delta(\theta,\xi)
\end{split}
\end{equation}where  \(U_\delta(\theta,\xi)\) is defined  by (\ref{eq:averagedparalpotential}). Since \(U_\delta(\theta,\xi)\) and all its derivatives depend on \(\delta\) continuously,  the non-degeneracy assumption Box2 implies that the minimal line at \(\delta=0\) is non-degenerate and persists for small \(\delta\). So, we introduce local normal coordinates near this line:
$$(\theta^{(n)},I^{(n)})\rightarrow(\varphi,\psi,P,J),$$
in the same way as in (\ref{eq:phipsitrans}),(\ref{eq:pjdef}). By the translation invariance of $U_\delta$, it
is independent of $\varphi$:
 \begin{equation}\label{eq:uparhatbox}
 U_{\delta}(\theta,\xi)=\hat  U_\delta(\psi,\xi).
\end{equation}
The potential $\hat U_\delta$ has a non-degenerate minimum at $(\psi=0,\xi=\xi_{min})$. The Hamiltonian
 \(H\) of (\ref{eq:averagebox}) becomes
\begin{equation}\label{eq:Hrectavrg1}
H(P,\varphi ,J,\psi,p_{\xi},\xi;\delta)=\omega_{0}(\delta)P+\delta^{1/2}\left(\frac{1}{2N}a(\delta)P^{2}+\frac{1}{2N}a(\delta)J^{2} + \frac{p_{\xi}^2}{2} +\hat U_\delta(\psi,\xi)\right).
\end{equation}
This Hamiltonian is similar to that in (\ref{eq:hamlitafterphipsi}) but it is simpler, as there is only one fast degree of freedom
$(\varphi, P)$. Moreover, $H$ is independent of $\varphi$, so \(P\) is an integral  which controls the period of the fast motion.

With this simplification in mind, let us follow the same procedure as is applied in Section \ref{sec:proofmainthm}
to Hamiltonian (\ref{eq:hamlitafterphipsi}) and compute  the Poincar\'e return map from \(\varphi=0\) to
\(\varphi=2\pi\). Since the flight time depends only on \(P\), the restriction of this map onto a fixed level of \(P\) is a constant-time map for the Hamiltonian  (\ref{eq:Hrectavrg1}) or, equivalently, an \(O(\delta^{1/2})\)-time map for the Hamiltonian
\begin{equation}\label{eq:Hrectavrg2}
\begin{array}{ll}
H_{P}(J,\psi,p_{\xi},\xi;\delta)&=\frac{1}{2N}a(\delta)J^{2} +\ \frac{p_{\xi }^2}{2} +\hat U_{\delta}(\psi,\xi).  \end{array}
\end{equation}
Thus, as in Lemma \ref{lem:poincaremapclosebilnew}, the Poincar\'e map of  (\ref{eq:Hrectexpanded}) is
\(o(\delta^{1/2})\)-close to the  \(O(\delta^{1/2})\)-time map for  the Hamiltonian  (\ref{eq:Hrectavrg2}).
Since \(\hat U_{\delta}(\psi,\xi)\) depends continuously on \(\delta\) with all derivatives, it follows from the non-degeneracy assumption Box2 and the KAM non-degeneracy assumption Box3 that the Hamiltonian  (\ref{eq:Hrectavrg2})  has a KAM-non-degenerate elliptic fixed point near \((J,\psi,p_{\xi},\xi)=(0,0,0,\xi_{min})\) with a positive measure set of KAM tori around it. Hence the Poincar\'e map of  the Hamiltonian (\ref{eq:Hrectexpanded}) also has such elliptic fixed point,  proving  Theorem \ref{thm:billiardbox1}.

\subsection{\label{sec:proofsmoothpot}Smoothness properties of the averaged potential}
Before proceeding to the proof of Theorem \ref{thm:billiardbox2} (the case of simultaneous impacts) we investigate regularity properties of the averaged potential $W_{avg}$ in (\ref{eq:averagedpotboxdelta}).
We start with the case $\delta=0$, i.e., let the function $q^{*}$ in (\ref{eq:averagedpotboxdelta}) be the billiard's saw-tooth solution (\ref{eq:sawtoothsol}).
\begin{lem}\label{lem:udelt0isc2} Provided \(||\xi^{(n)}-\xi^{(m)}||>\rho \), the averaged potential along the saw-tooth solution,  \(W_{avg}\) of (\ref{eq:potsawtooth}), is  $C^\infty$ for $\theta^{(n)}-\theta^{(m)} \neq 0 \mod \pi$.
Under the parity assumption Box4, at the singularities $\theta^{(n)}-\theta^{(m)} = 0 \mod \pi$,
the potential $W_{avg}$ is \(C^{2}\)-smooth, yet it is not \(C^{3}\) smooth in general.
\end{lem}

\begin{proof}
Setting the shorthand notation \(\vartheta=\theta^{(m)}-\theta^{(n)}\) and \( \zeta=\xi^{(n)}-\xi^{(m)}\), and using the periodicity of \(q^{*}\),  the averaged potential of (\ref{eq:potsawtooth}) becomes
$$W_{avg}(\vartheta ,\zeta) = \frac{1}{2\pi}\int_0^{2\pi}\ W(q^{*}(s)-q^{*}(s+\vartheta),\zeta)ds.$$
By (\ref{eq:sawtoothsol}), for \(\vartheta\in(0,\pi)\),
\begin{equation}\label{wbila}
\begin{array}{ll}
2\pi  W_{avg}(\vartheta,\zeta)
 & =\int_{0}^{\pi-\vartheta}\ W(-\vartheta,\zeta)ds+\int_{\pi-\vartheta}^{\pi}\ W(2s-2\pi+\vartheta,\zeta)ds \\
 & \qquad \qquad\qquad\qquad
 +\int_{\pi}^{2\pi-\vartheta}W(\vartheta,\zeta)ds+\int_{2\pi-\vartheta}^{2\pi}W(4\pi-2s-\vartheta,\zeta)ds
 \\
 & =(\pi-\vartheta)(W(-\vartheta,\zeta)+W(\vartheta,\zeta))\\
 & \qquad\quad+\int_{\pi-\vartheta}^{\pi} (W(2s-2\pi+\vartheta,\zeta)+W(-2s+2\pi-\vartheta,\zeta))ds
\end{array}
\end{equation}
so for \(\|\zeta\|>\rho\) it is \(C^{\infty}\) for \(\vartheta\in(0,\pi)\). Now, recall that \( W_{avg}(\vartheta,\zeta)\) is an even function, so its extension to \(\vartheta\in(-\pi,0)\) is simply  \(W_{avg} (\vartheta,\zeta)= W_{avg}(-\vartheta,-\zeta) \), hence it is $C^\infty$ for all \(\vartheta\) away from the matching points \(\vartheta=0 \mod \pi\).

It remains to prove the \(C^{2}\)-smoothness at the matching points. By the periodicity and the parity assumption Box4,
$$W_{avg} (\vartheta,\zeta)= W_{avg}(-\vartheta,\zeta)=W_{avg}(2\pi-\vartheta,\zeta),$$
so it is enough to verify only that the first derivatives of $W_{avg}$ in (\ref{wbila}) vanish at $\vartheta=0,\pi$.

By the parity assumption the interaction potential $W$ is even in \(\vartheta\), so  (\ref{wbila}) becomes\begin{equation}
\begin{array}{ll}
2\pi  W_{avg}(\vartheta,\zeta)
 & =2(\pi-\vartheta)W(\vartheta,\zeta)+2\int_{\pi-\vartheta}^{\pi} W(2s-2\pi+\vartheta,\zeta)ds\\
 &=2\pi W(\vartheta,\zeta)+2\int_0^{\vartheta} (W(-2u+\vartheta,\zeta)-W(\vartheta,\zeta))du.
\end{array}
\end{equation}
Therefore,
$$
\frac{\partial}{\partial\vartheta}W_{avg}(\vartheta,\zeta) = \frac{\partial}{\partial\vartheta}W(\vartheta,\zeta)+\frac{1}{\pi}\int_{0}^{\vartheta} \frac{\partial}{\partial\vartheta}(W(-2u+\vartheta,\zeta)-W(\vartheta,\zeta))du
=(1-\frac{\vartheta}{\pi} )\frac{\partial}{\partial\vartheta}W(\vartheta,\zeta),
$$
so \(\frac{\partial}{\partial\vartheta}W_{avg}(\vartheta,\zeta)|_{\vartheta=0,\pi}=0\), as required.

Differentiating further, we find
\begin{equation}\label{avgbil2der}
\frac{\partial^{2}}{\partial\vartheta^{2}}W_{avg}(\vartheta,\zeta)  =(1-\frac{\vartheta}{\pi} )\frac{\partial^{2}}{\partial\vartheta^{2}}W(\vartheta,\zeta)-\frac{1}{\pi}\frac{\partial}{\partial\vartheta}W(\vartheta,\zeta).
\end{equation}
Next, we obtain
$$ \frac{\partial^{3}}{\partial\vartheta^{3}}W_{avg}(\vartheta,\zeta)  =(1-\frac{\vartheta}{\pi} )\frac{\partial^{3}}{\partial\vartheta^{3}}W(\vartheta,\zeta)-\frac{2}{\pi}\frac{\partial^{2}}{\partial\vartheta^{2}}W(\vartheta,\zeta),
$$
so  \(\frac{\partial^{3}}{\partial\vartheta^{3}}W_{avg}(\vartheta,\zeta)|_{\vartheta\rightarrow0,\pi}=-\frac{2}{\pi}\frac{\partial^{2}}{\partial\vartheta^{2}}W(\vartheta,\zeta)|_{\vartheta\rightarrow0,\pi}\).
Generically, these values  do not vanish, so the even extension  of \(W_{avg}\) to negative $\vartheta$ cannot be \(C^{3}\)
(i.e., it is only piecewise smooth). \end{proof}

Let us now consider the case of $\delta>0$. As in the lemma above, we use the notation
  \(\vartheta=\theta^{(n)}-\theta^{(m)} \) and   \( \zeta=\xi^{(n)}-\xi^{(m)}\).
\begin{lem}\label{lem:wtildc2inthet}
Away from   \(\vartheta=0 \mod \pi\), the averaged potential  \( W_{avg}(\vartheta,\zeta;\delta)\) is \(C^{\infty}\)-close, for small $\delta$ to the saw-tooth averaged potential
\(W_{avg}(\vartheta,\zeta;0)\).

When the parity assumption Box4 is satisfied, the averaged potential near \(\vartheta = 0 \mod \pi\) is  \(C^2\)-close
to \(W_{avg}(\vartheta,\zeta;0)\), along with the derivatives with respect to \(\zeta\).
The higher order derivatives with respect to \(\vartheta\) do not, in general, have a continuous limit near the singular values
\(\vartheta= 0\mod \pi\) as $\delta\to 0$.  The following estimates hold true:
$$\begin{array}{l}
\frac{\partial^{3}}{\partial\vartheta^{3}}W_{avg}(\vartheta,\zeta;\delta)=O(1),\\
\frac{\partial^{4}}{\partial\vartheta^{4}}W_{avg}(\vartheta,\zeta;\delta)=O(\delta^{-\frac{1}{\alpha}}),\\
\frac{\partial^{5}}{\partial\vartheta^{5}}W_{avg}(\vartheta,\zeta;\delta)=O(\delta^{-\frac{2}{\alpha}});
\end{array}$$
differentiation with respect to $\zeta$ does not affect these estimates:
$$
\frac{\partial^{i+k}}{\partial\vartheta^{i}\partial\zeta^k}W_{avg}(\vartheta,\zeta;\delta)=O(\delta^{-\frac{i-3}{\alpha}}),
\qquad i=3,4,5; \ \ k\geqslant 0.
$$
\end{lem}
\begin{proof} In order to establish the regularity for $\vartheta\neq 0 \mod \pi$ (i.e., for non-simultaneous impacts), we take
$\vartheta\in (0,\pi)$ and let $\eta>0$ be a sufficiently small number so that \(\eta<\min(\vartheta,\pi-\vartheta)\). By the periodicity of $q^*_\delta$, we write
(\ref{eq:averagedpotboxdelta}) as$$\begin{array}{ll}2\pi  W_{avg}(\vartheta,\zeta;\delta) &=
\int_{-\eta}^{2\pi-\eta} W(q_{\delta}^{*}(s)-q_{\delta}^{*}(s+\vartheta),\zeta)ds \\&=(
\int_{-\eta}^{\pi-\vartheta-\eta}
+\int_{\pi-\vartheta-\eta}^{\pi-\vartheta+\eta} +
\int_{\pi-\vartheta+\eta}^{2\pi-\vartheta-\eta} +
\int_{2\pi-\vartheta-\eta}^{2\pi-\eta} )W(q_{\delta}^{*}(s)-q_{\delta}^{*}(s+\vartheta),\zeta)ds\\& =
\int_{-\eta}^{\pi-\vartheta-\eta} W(q_{\delta}^{*}(s)-q_{\delta}^{*}(s+\vartheta),\zeta)ds
+\int_{\pi-\eta}^{\pi+\eta} W(q_{\delta}^{*}(s-\vartheta)-q_{\delta}^{*}(s),\zeta)ds\\ & +
\int_{\pi-\vartheta+\eta}^{2\pi-\vartheta-\eta} W(q_{\delta}^{*}(s)-q_{\delta}^{*}(s+\vartheta),\zeta)ds +
\int_{2\pi-\eta}^{2\pi+\vartheta-\eta} W(q_{\delta}^{*}(s-\vartheta)-q_{\delta}^{*}(s),\zeta)ds.
\end{array}$$
The function $q^*_\delta$ is $C^\infty$ for all $\delta\geqslant 0$ when its argument is bounded away from
$0 \mod \pi$, so every integral in this sum is $C^\infty$ function of \(\vartheta\) for all $\delta\geqslant 0$ (this is an explicit version of a similar statement in Lemma \ref{Lemmaavbil}). Similarly, one proves the regularity
of $W_{avg}$ for $\vartheta\in (\pi,2\pi)$.

Let us now examine the case where $\vartheta$ is close to $0$ or $\pi$. We have
\begin{equation}\label{eq:wavgsandr}
2\pi  W_{avg}(\vartheta,\zeta;\delta)\ =\int_0^{2\pi}\ W(q_{\delta}^{*}(s)-q_{\delta}^{*}(s+\vartheta),\zeta)ds
:= S(\vartheta,\zeta;\delta,\eta)+R(\vartheta,\zeta;\delta,\eta),
\end{equation}
where \(S\), the ``singular part'', corresponds to the integration intervals with both particles  \(\eta\)-close to impacts:
\begin{equation}\label{seqwav}
S(\vartheta,\zeta;\delta,\eta)= \int_{-\eta}^{\eta}W(q_{\delta}^{*}(s)-q_{\delta}^{*}(s+\vartheta),\zeta)ds+\int_{\pi-\eta}^{\pi+\eta}W(q_{\delta}^{*}(s)-q_{\delta}^{*}(s+\vartheta ),\zeta)ds,
\end{equation}
and \(R\), the ``regular part'', corresponds to the integration intervals  for which both particles are at a distance larger than
\(\eta\) from impacts:
\begin{equation}\label{reqwav}
R(\vartheta,\zeta;\delta,\eta)=\int_{\eta}^{\pi-\eta} W(q_{\delta}^{*}(s)-q_{\delta}^{*}(s+\vartheta),\zeta)ds +\int_{\pi+\eta}^{2\pi-\eta}W(q_{\delta}^{*}(s)-q_{\delta}^{*}(s+\vartheta),\zeta)ds.
\end{equation}

For small \(\eta\), and \(\vartheta\) close to \(0\) (both particles have the same phase) or \(\pi\) (the anti-phase state), the term \(R(\vartheta,\zeta;\delta,\eta)\) is \(C^{\infty}\) for all \(\delta\geqslant 0\). Thus, we need to evaluate the singular term, \(S\), and its derivatives. Recall  that
\(q_{\delta}^{*}(s)\rightarrow q_{0}^{*}(s)\) in \(C^0\), so \(S\) converges in  \(C^0\) to the billiard limit \(S(\vartheta,\zeta;0,\eta)\). We will show that up to order 3 the derivatives of \(S\) are uniformly bounded for all small $\delta$. This proves, by compactness argument, the \(C^{2}\)-closeness of \(S\) and, hence, of
\(W_{avg}\), to their
\(C^{0}\) limits at \(\delta=0\), as claimed.

So, to prove the lemma, we only need to estimate the derivatives of $S$.
In order to do this, we use formula (\ref{eq:qdeltabound}) for $q^*_\delta$, where
the first impact is at $M^1: q=0$ (this corresponds to $\theta_1=0$) and the second impact is at $M^2: q=\pi$
(this corresponds to $\theta_2=\pi$).
Recall that $q^*_\delta$ is the periodic orbit in the
energy level $H=\frac{1}{2}$ of the system (\ref{hsdofsp}).
Since this system has one degree of freedom and
is reversible, the periodic orbit $q^*_\delta(\theta)$ is even and $2\pi$-periodic. Moreover, by the symmetry of
the potential $V_d(q)$ (see Assumption Box1), we have $q^*_\delta(\theta)=q^*_\delta(\pi-\theta)$.
Thus, we may write (\ref{eq:qdeltabound}) as
\begin{equation}\label{qdeltabo}
\begin{array}{l}\displaystyle
q^*_\delta(\theta)=
\delta^{1/\alpha} \tilde q_\delta\left(\frac{\theta}{\omega_0(\delta)\delta^{1/\alpha}}\right) \;\mbox{ near }\; \theta=0,\\
\displaystyle
q^*_\delta(\theta)= \pi -
\delta^{1/\alpha} \tilde q_\delta\left(\frac{\theta-\pi}{\omega_0(\delta)\delta^{1/\alpha}}\right) \;\mbox{ near }\; \theta=\pi.
\end{array}
\end{equation}
where $\tilde q_\delta(\cdot)$ is an even function with bounded derivatives (uniformly for all $\delta\geqslant 0$).
By (\ref{hsdofsp})
$$\frac{d}{dt} q^*_\delta = p^*_\delta, \qquad \frac{d}{dt} p^*_\delta = - \delta V_d'(q^*_\delta).$$
Therefore, if we denote
\begin{equation}\label{tilpj}
\tilde p_{\delta}\left(\frac{\theta}{\omega_0(\delta)\delta^{1/\alpha}}\right) = \tilde q_{\delta}' = \frac{d}{dt} q^*_{\delta}
\end{equation}
(see (\ref{eq:scaledq}) and (\ref{eq:qdeltabound})), then
\begin{equation}\label{eq:pprimebl}
\tilde p_{\delta}'=-\frac{\alpha\delta^{1+1/\alpha}}{Q_d(\delta^{1/\alpha} \tilde q)^{\alpha+1}}Q_d'(\delta^{1/\alpha} \tilde q)=- \; \frac{1}{\tilde q_{\delta}^{\alpha+1}}\left(\frac{\alpha}{(Q_d'(0))^{\alpha}}+
O(\delta^{1/\alpha})\right).
\end{equation}
 Notice that by
Lemma \ref{lem:qnearimpact}, the behavior of
\(\tilde q\) is asymptotically linear at large \(u\):
\begin{equation}\label{behtlq}
\tilde q_\delta(u) = |u| + o(u),
\end{equation}
so
\begin{equation}\label{derpkal1}
\tilde p_{\delta}^{(k)}(u) = O(|u|^{-\alpha -k}), \qquad k\geqslant 1.
\end{equation}
Note also that the frequency $\omega_0$ of the vertical oscillations tends to that of the billiard motion, i.e.,
\begin{equation}\label{omega0box5}
\omega_0(0)=1.
\end{equation}

Now we can return to analyze the behavior of the singular term \(S\). For \(\vartheta\) close to $0$,  using formulas (\ref{qdeltabo}) and that \(W\) is even, we write (\ref{seqwav}) as
\begin{equation}\nonumber
\begin{array}{l}\displaystyle \!\!\!\!\!\!\!\!\!\!
S(\vartheta,\zeta;\delta,\eta) =(\int_{-\eta}^{\eta} +\int_{\pi-\eta}^{\pi+\eta} )\ W(q_{\delta}^{*}(s)-q_{\delta}^{*}(s+\vartheta),\zeta)ds =2\int_{-\eta}^{\eta} W\left(\delta^{1/\alpha} \tilde q_{\delta}(\frac{s}{\omega_0\delta^{1/\alpha}})-\delta^{1/\alpha} \tilde q_{\delta}(\frac{s+\vartheta}{\omega_0\delta^{1/\alpha}}),\ \zeta\right) ds
 \\ \\ \displaystyle
 \qquad\qquad=2\omega_0\delta^{1/\alpha}\int_{-\eta/(\omega_0\delta^{\frac{1}{\alpha}})}^{\eta/(\omega_0\delta^{\frac{1}{\alpha}})} W\left(\delta^{1/\alpha} \tilde q_{\delta}(u)-\delta^{1/\alpha} \tilde q_{\delta}(u+\frac{\vartheta}{\omega_0\delta^{1/\alpha}}),\ \zeta\right) du.
\end{array}
\end{equation}
Similarly, near \(\vartheta=\pi\), we have
\begin{equation}\nonumber
\begin{array}{l}\displaystyle \!\!\!\!\!\!\!\!\!
S(\vartheta,\zeta;\delta,\eta) =(\int_{-\eta}^{\eta} + \int_{\pi-\eta}^{\pi+\eta})   W(q_{\delta}^{*}(s)-q_{\delta}^{*}(s+\vartheta),\zeta)ds
\\   \!\!\!\!\!\!\!\!\!\!\!\!\!\!\!\!
\displaystyle =\int_{-\eta}^{\eta}\!\!W(\delta^{1/\alpha} \tilde q_\delta(\frac{s}{\omega_0\delta^{1/\alpha}})+
\delta^{1/\alpha} \tilde q_\delta(\frac{s+\vartheta-\pi}{\omega_0\delta^{1/\alpha}}) -\pi,\zeta)ds
 +\int_{\pi-\eta}^{\pi+\eta}\!\!W(\pi-\delta^{1/\alpha} \tilde q_\delta(\frac{s-\pi}{\omega_0\delta^{1/\alpha}})-\delta^{1/\alpha} \tilde q_\delta(\frac{s-2\pi+\vartheta}{\omega_0\delta^{1/\alpha}}),\zeta)ds
\\ \displaystyle
=2\omega_0\delta^{1/\alpha}\int_{-\eta/(\omega_0\delta^{\frac{1}{\alpha}})}^{\eta/(\omega_0\delta^{\frac{1}{\alpha}})} W\left(\pi-\delta^{1/\alpha} \tilde q_\delta(u)-\delta^{1/\alpha} \tilde q_\delta(u +\frac{\vartheta-\pi}{\omega_0\delta^{1/\alpha}}),\ \zeta\right)\ ds.
\end{array}
\end{equation}
These two formulas can be written in a unified way:
\begin{equation}\nonumber
S(\vartheta,\zeta;\delta,\eta) =2\omega_0\delta^{1/\alpha}\int_{-\eta/(\omega_0\delta^{\frac{1}{\alpha}})}^{\eta/(\omega_0\delta^{\frac{1}{\alpha}})} W\left(\sigma \pm \delta^{1/\alpha} \tilde q_{\delta}(u)-\delta^{1/\alpha} \tilde q_{\delta}(u+\frac{\hat\vartheta}{\omega_0\delta^{1/\alpha}}),\ \zeta\right) du,
\end{equation}
where one chooses $\sigma=0$, $\hat\vartheta=\vartheta$ and the plus sign in front of $\delta^{1/\alpha} \tilde q_{\delta}(u)$ in the case of $\vartheta$ close to zero, and $\sigma=\pi$, $\hat\vartheta=\vartheta-\pi$ and the minus sign in front of $\delta^{1/\alpha} \tilde q_{\delta}(u)$ in the case of $\vartheta$ close to $\pi$.

The first derivative of $S$ with respect to $\vartheta$ is
\begin{equation}
\frac{\partial}{\partial\vartheta}S(\vartheta,\zeta;\delta,\eta)=-2\delta^{1/\alpha}\int_{-\eta/(\omega_0\delta^{\frac{1}{\alpha}})}^{\eta/(\omega_0\delta^{\frac{1}{\alpha}})} \frac{\partial W}{\partial q}\left(\sigma \pm \delta^{1/\alpha} \tilde q_{\delta}(u)-\delta^{1/\alpha} \tilde q_{\delta}(u+\frac{\hat\vartheta}{\omega_0\delta^{1/\alpha}}),\ \zeta\right) \tilde p_{\delta}(u+\frac{\hat\vartheta}{\omega_0\delta^{1/\alpha}})du,
\nonumber \end{equation}
where $\tilde p_{\delta}$ is the derivative of $\tilde q_{\delta}$ (see (\ref{tilpj})), so it is bounded with all derivatives by Lemma \ref{lem:qnearimpact}. Since the integrand is bounded (along with all derivatives with respect to $\zeta$), it follows that $\frac{\partial S}{\partial\vartheta}=O(\eta)$, along with its derivatives with respect to $\zeta$.

Next, we check the second derivative:
\begin{equation}\nonumber
\!\!\!\!\!\!\!\!\!\!\begin{array}{ll}
 \frac{\partial^{2}}{\partial\vartheta^{2}}S(\vartheta,\zeta;\delta,\eta)&=2\omega_0^{-1}\delta^{1/\alpha}\int_{-\eta/(\omega_0\delta^{\frac{1}{\alpha}})}^{\eta/(\omega_0\delta^{\frac{1}{\alpha}})} \frac{\partial^{2}W}{\partial q^{2}}(\sigma \pm \delta^{1/\alpha} \tilde q_\delta(u)-\delta^{1/\alpha} \tilde q_\delta(u+\frac{\hat\vartheta}{\omega_0\delta^{1/\alpha}}),\ \zeta)\ \tilde p_\delta(u+\frac{\hat\vartheta}{\omega_0\delta^{1/\alpha}})^{2}du\\
& -2\omega_0^{-1}\int_{-\eta/(\omega_0\delta^{\frac{1}{\alpha}})}^{\eta/(\omega_0\delta^{\frac{1}{\alpha}})} \frac{\partial W}{\partial q}(\sigma \pm \delta^{1/\alpha} \tilde q_\delta(u)-\delta^{1/\alpha} \tilde q_\delta(u+\frac{\hat\vartheta}{\omega_0\delta^{1/\alpha}}),\ \zeta)\ \tilde p_\delta'(u+\frac{\hat\vartheta}{\omega_0\delta^{1/\alpha}})du.
\end{array}
\end{equation}
As before, the first line of the right-hand side is \(O(\eta)\). Since $\tilde p'$ decays as $|u|^{-\alpha-1}$ (see (\ref{derpkal1})), the
integral in the second line is uniformly convergent. Thus, $\frac{\partial^2 S}{\partial\vartheta^2}$ is uniformly bounded for all small $\delta$. The same is true for its derivatives with respect to $\zeta$.

Differentiating further, we obtain that
\begin{equation}\label{eq:sthirdderv} \begin{array}{l}
 \frac{\partial^{3}}{\partial\vartheta^{3}}S(\vartheta,\zeta;\delta,\eta) =-2\omega_0^{-2}\delta^{\frac{1}{\alpha}}\int_{-\eta/(\omega_0\delta^{\frac{1}{\alpha}})}^{\eta/(\omega_0\delta^{\frac{1}{\alpha}})} \frac{\partial^{3}W}{\partial q^{3}}(\sigma \pm \delta^{1/\alpha} \tilde q_\delta(u)-\delta^{1/\alpha} \tilde q_\delta(u+\frac{\hat\vartheta}{\omega_0\delta^{1/\alpha}}),\ \zeta)\ \tilde p_\delta(u+\frac{\hat\vartheta}{\omega_0\delta^{1/\alpha}})^{3} du\\
\qquad+6\omega_0^{-2}\int_{-\eta/(\omega_0\delta^{\frac{1}{\alpha}})}^{\eta/(\omega_0\delta^{\frac{1}{\alpha}})} \frac{\partial^{2}W}{\partial q^{2}}(\sigma \pm \delta^{1/\alpha} \tilde q_\delta(u)-\delta^{1/\alpha} \tilde q_\delta(u+\frac{\hat\vartheta}{\omega_0\delta^{1/\alpha}}),\ \zeta)\ \tilde p_\delta(u+\frac{\hat\vartheta}{\omega_0\delta^{1/\alpha}}) \tilde p_\delta'(u+\frac{\hat\vartheta}{\omega_0\delta^{1/\alpha}}) du\\ \qquad
-2\omega_0^{-2}\delta^{-\frac{1}{\alpha}}\int_{-\eta/(\omega_0\delta^{\frac{1}{\alpha}})}^{\eta/(\omega_0\delta^{\frac{1}{\alpha}})} \frac{\partial W}{\partial q}(\sigma \pm \delta^{1/\alpha} \tilde q_\delta(u)-\delta^{1/\alpha} \tilde q_\delta(u+\frac{\hat\vartheta}{\omega_0\delta^{1/\alpha}}),\ \zeta)\ \tilde p_\delta''(u+\frac{\hat\vartheta}{\omega_0\delta^{1/\alpha}}) du.
\end{array}
\end{equation}
As above, the first term in the right-hand side is
\(O(\eta)\). Integrating the last term by parts, we obtain
\begin{equation}\nonumber
\begin{array}{l}
\frac{\partial^{3}}{\partial\vartheta^{3}}S_j(\vartheta,\zeta;\delta,\eta) = O(\eta)+\\ \quad\
+ 2\omega_0^{-2}\int_{-\eta/(\omega_0\delta^{\frac{1}{\alpha}})}^{\eta/(\omega_0\delta^{\frac{1}{\alpha}})}\!\!
\frac{\partial^{2}W}{\partial q^{2}}(\sigma \pm \delta^{1/\alpha} \tilde q_\delta(u)-\delta^{1/\alpha} \tilde q_\delta(u+\frac{\hat\vartheta}{\omega_0\delta^{1/\alpha}}),\ \zeta)\
(2\tilde p_\delta(u+\frac{\hat\vartheta}{\omega_0\delta^{1/\alpha}}) \pm \tilde p_\delta(u))\ \tilde p_\delta'(u+\frac{\hat\vartheta}{\omega_0\delta^{1/\alpha}}) du\\
\quad\ -2\delta^{-\frac{1}{\alpha}}[\frac{\partial W}{\partial q}(\sigma \pm \delta^{1/\alpha} \tilde q_\delta(u)-\delta^{1/\alpha} \tilde q_\delta(u+\frac{\hat\vartheta}{\omega_0\delta^{1/\alpha}}),\ \zeta)\ \tilde p_\delta'(u+\frac{\hat\vartheta}{\omega_0\delta^{1/\alpha}})]_{-\eta/(\omega_0\delta^{\frac{1}{\alpha}})}^{\eta/(\omega_0\delta^{\frac{1}{\alpha}})}
\end{array}
\end{equation}
Since $\tilde p_\delta'$ decays as $|u|^{-\alpha-1}$, all the terms here are uniformly bounded, i.e.,
\(\frac{\partial^{3}}{\partial\vartheta^{3}}S(\vartheta,\zeta;\delta,\eta)\) is uniformly bounded, with all its $\zeta$-derivatives.

Similarly, differentiating (\ref{eq:sthirdderv}) and integrating by parts, we obtain:
\begin{equation}\nonumber
\begin{array}{l}
\frac{\partial^{4}}{\partial\vartheta^{4}}S(\vartheta,\zeta;\delta,\eta) = 2\omega_0^{-3}\delta^{\frac{1}{\alpha}}\int_{-\eta/(\omega_0\delta^{\frac{1}{\alpha}})}^{\eta/(\omega_0\delta^{\frac{1}{\alpha}})} \frac{\partial^{4 }W}{\partial q^{4}}(\sigma \pm \delta^{1/\alpha} \tilde q_\delta(u)-\delta^{1/\alpha} \tilde q_\delta(u+\frac{\hat\vartheta}{\omega_0\delta^{1/\alpha}}),\  \zeta)\ \tilde p_\delta(u+\frac{\hat\vartheta}{\omega_0\delta^{1/\alpha}})^{4} du\\
\quad\ -12\omega_0^{-3}\int_{-\eta/(\omega_0\delta^{\frac{1}{\alpha}})}^{\eta/(\omega_0\delta^{\frac{1}{\alpha}})} \frac{\partial^{3} W}{\partial q^{3}}(\delta^{1/\alpha} \sigma \pm \tilde q_\delta(u)-\delta^{1/\alpha} \tilde q_\delta(u+\frac{\hat\vartheta}{\omega_0\delta^{1/\alpha}}),\ \zeta)\ \tilde p_\delta(u+\frac{\hat\vartheta}{\omega_0\delta^{1/\alpha}})^2 \tilde p_\delta'(u+\frac{\hat\vartheta}{\omega_0\delta^{1/\alpha}}) du\\
\quad\ +6\omega_0^{-3}\delta^{-\frac{1}{\alpha}}
\int_{-\eta/(\omega_0\delta^{\frac{1}{\alpha}})}^{\eta/(\omega_0\delta^{\frac{1}{\alpha}})}\!\! \frac{\partial^{2}W}{\partial q^{2}}(\sigma \pm \delta^{1/\alpha} \tilde q_\delta(u)-\delta^{1/\alpha} \tilde q_\delta(u+\frac{\hat\vartheta}{\omega_0\delta^{1/\alpha}}), \zeta)\ (\tilde p_\delta'(u+\frac{\hat\vartheta}{\omega_0\delta^{1/\alpha}}))^{2} du \\
\quad\ +2\omega_0^{-3}\delta^{-\frac{1}{\alpha}}\int_{-\eta/(\omega_0\delta^{\frac{1}{\alpha}})}^{\eta/(\omega_0\delta^{\frac{1}{\alpha}})}\!\! \frac{\partial^{2}W}{\partial q^{2}}(\sigma \pm \delta^{1/\alpha} \tilde q_\delta(u)-\delta^{1/\alpha} \tilde q_\delta(u+\frac{\hat\vartheta}{\omega_0\delta^{1/\alpha}}), \zeta)\
 (3\tilde p_\delta(u+\frac{\hat\vartheta}{\delta^{1/\alpha}}) \pm \tilde p_\delta(u))\ \tilde p_\delta''(u+\frac{\hat\vartheta}{\omega_0\delta^{1/\alpha}}) du\\
\quad\ -2\omega_0^{-3}\delta^{-\frac{2}{\alpha}}[\frac{\partial W}{\partial q}(\sigma \pm \delta^{1/\alpha} \tilde q_\delta(u)-\delta^{1/\alpha} \tilde q_\delta(u+\frac{\hat\vartheta}{\omega_0\delta^{1/\alpha}}), \zeta)\ \tilde p_\delta''(u+\frac{\hat\vartheta}{\omega_0\delta^{1/\alpha}})]_{-\eta/(\omega_0\delta^{\frac{1}{\alpha}})}^{\eta/(\omega_0\delta^{\frac{1}{\alpha}})}
\end{array}
\end{equation}
As before, the first term is \(O(\eta)\). By (\ref{derpkal1}), we have \( \tilde p_\delta'(u)\approx 1/|u|^{\alpha +1}\) and  \( \tilde p_\delta''(u)\approx1/|u|^{\alpha+2}\) at  large \(|u|\). This implies that the integrals in the 2-4 lines are uniformly bounded, so the second line is $O(1)$ and the third and fourth lines are
\(O(\delta^{-\frac{1}{\alpha}})\). The fifth line is of order O(\(\delta^{-\frac{2}{\alpha}}\delta^{\frac{\alpha+2}{\alpha}})=O(\delta)\). So the fourth order derivative with respect to $\vartheta$ diverges, along with its derivatives with respect to $\zeta$, at most as
\(O(\delta^{-\frac{1}{\alpha}})\), in agreement with the claim of the lemma. Namely
\begin{equation}\label{eq:sfourthderv}
\begin{array}{l}
\frac{\partial^{4}}{\partial\vartheta^{4}}S(\vartheta,\zeta;\delta,\eta) = 2\omega_0^{-3}\delta^{-\frac{1}{\alpha}}
\int_{-\eta/(\omega_0\delta^{\frac{1}{\alpha}})}^{\eta/(\omega_0\delta^{\frac{1}{\alpha}})} \frac{\partial^{2}W}{\partial q^{2}}(\sigma \pm \delta^{1/\alpha} \tilde q_\delta(u)-\delta^{1/\alpha} \tilde q_\delta(u+\frac{\hat\vartheta}{\omega_0\delta^{1/\alpha}}),\ \zeta) \\
\qquad\qquad \times( 3(\tilde p_\delta'(u+\frac{\hat\vartheta}{\omega_0\delta^{1/\alpha}}))^{2}+(3\tilde p_\delta(u+\frac{\hat\vartheta}{\delta^{1/\alpha}}) \pm \tilde p_\delta(u)) \tilde p_\delta''(u+\frac{\hat\vartheta}{\omega_0\delta^{1/\alpha}}))du+O(1).
\end{array}
\end{equation}
Finally, we evaluate the fifth derivative, using the same procedure as above (i.e., twice differentiating (\ref{eq:sthirdderv}), integrating by parts, and using the estimate (\ref{derpkal1}) for the decay of the derivatives of $\tilde p$ at large $|u|$):
 \begin{equation}\nonumber
\begin{array}{l}
 \frac{\partial^{5}}{\partial\vartheta^{5}} S(\vartheta,\zeta;\delta,\eta) =-2\omega_0^{-4}\delta^{\frac{1}{\alpha}}\int_{-\eta/(\omega_0\delta^{\frac{1}{\alpha}})}^{\eta/(\omega_0\delta^{\frac{1}{\alpha}})} \frac{\partial^{5}W}{\partial q^{5}}(\sigma \pm \delta^{1/\alpha} \tilde q_\delta(u)-\delta^{1/\alpha} \tilde q_\delta(u+\frac{\hat\vartheta}{\omega_0\delta^{1/\alpha}}),\zeta)\ \tilde p_\delta(u+\frac{\hat\vartheta}{\omega_0\delta^{1/\alpha}})^{5}du\\
\quad\ +20\omega_0^{-4}\int_{-\eta/(\omega_0\delta^{\frac{1}{\alpha}})}^{\eta/(\omega_0\delta^{\frac{1}{\alpha}})} \frac{\partial^{4}W}{\partial q^{4}}(\sigma \pm \delta^{1/\alpha} \tilde q_\delta(u)-\delta^{1/\alpha} \tilde q_\delta(u+\frac{\hat\vartheta}{\omega_0\delta^{1/\alpha}}),\zeta)\ \tilde p_\delta(u+\frac{\hat\vartheta}{\omega_0\delta^{1/\alpha}})^3 \tilde p_\delta'(u+\frac{\hat\vartheta}{\omega_0\delta^{1/\alpha}})du\\
\quad\ -30\omega_0^{-4}\delta^{-\frac{1}{\alpha}}\int_{-\eta/(\omega_0\delta^{\frac{1}{\alpha}})}^{\eta/(\omega_0\delta^{\frac{1}{\alpha}})} \frac{\partial^{3}W}{\partial q^{3}}(\sigma \pm \delta^{1/\alpha} \tilde q_\delta(u)-\delta^{1/\alpha} \tilde q_\delta(u+\frac{\hat\vartheta}{\omega_0\delta^{1/\alpha}}),\zeta)\ \tilde p_\delta(u+\frac{\hat\vartheta}{\omega_0\delta^{1/\alpha}}) (\tilde p_\delta'(u+\frac{\hat\vartheta}{\omega_0\delta^{1/\alpha}}))^2du\\
\quad\ -20\omega_0^{-4}\delta^{-\frac{1}{\alpha}}\int_{-\eta/(\omega_0\delta^{\frac{1}{\alpha}})}^{\eta/(\omega_0\delta^{\frac{1}{\alpha}})} \frac{\partial^{3}W}{\partial q^{3}}(\sigma \pm \delta^{1/\alpha} \tilde q_\delta(u)-\delta^{1/\alpha} \tilde q_\delta(u+\frac{\hat\vartheta}{\omega_0\delta^{1/\alpha}}),\zeta)\ \tilde p_\delta(u+\frac{\hat\vartheta}{\omega_0\delta^{1/\alpha}})^2 \tilde p_\delta''(u+\frac{\hat\vartheta}{\omega_0\delta^{1/\alpha}})du\\
\quad\ +20\omega_0^{-4}\delta^{-\frac{2}{\alpha}}
\int_{-\eta/(\omega_0\delta^{\frac{1}{\alpha}})}^{\eta/(\omega_0\delta^{\frac{1}{\alpha}})} \frac{\partial^{2}W}{\partial q^{2}}(\sigma \pm \delta^{1/\alpha} \tilde q_\delta(u)-\delta^{1/\alpha} \tilde q_\delta(u+\frac{\hat\vartheta}{\omega_0\delta^{1/\alpha}}),\zeta)\ \tilde p_\delta'(u+\frac{\hat\vartheta}{\omega_0\delta^{1/\alpha}}) \tilde p_\delta''(u+\frac{\hat\vartheta}{\omega_0\delta^{1/\alpha}})du \\
\quad\ +2\omega_0^{-4}\delta^{-\frac{2}{\alpha}}\int_{-\eta/(\omega_0\delta^{\frac{1}{\alpha}})}^{\eta/(\omega_0\delta^{\frac{1}{\alpha}})} \!\!\frac{\partial^{2}W}{\partial q^{2}}(\sigma \pm \delta^{1/\alpha} \tilde q_\delta(u)-\delta^{1/\alpha} \tilde q_\delta(u+\frac{\hat\vartheta}{\omega_0\delta^{1/\alpha}}),\zeta)
(4\tilde p_\delta(u+\!\frac{\hat\vartheta}{\omega_0\delta^{1/\alpha}}) \pm \tilde p_\delta(u)) \tilde p_\delta'''(u+\!\frac{\hat\vartheta}{\omega_0\delta^{1/\alpha}})du\\
\quad\ -2\omega_0^{-4}\delta^{-\frac{3}{\alpha}}[\frac{\partial W}{\partial q}(\sigma \pm \delta^{1/\alpha} \tilde q_\delta(u)-\delta^{1/\alpha} \tilde q_\delta(u+\frac{\hat\vartheta}{\omega_0\delta^{1/\alpha}}),\zeta)\ \tilde p_\delta'''(u+\frac{\hat\vartheta}{\omega_0\delta^{1/\alpha}})]_{-\eta/(\omega_0\delta^{\frac{1}{\alpha}})}^{\eta/(\omega_0\delta^{\frac{1}{\alpha}})}
 \\ \quad\ =O(\delta^{-\frac{2}{\alpha}})
\end{array}
\end{equation}
This completes the proof of the lemma. \end{proof}

\begin{lem}\label{lem:wavatorigin} Under the parity assumption, the second derivative of the averaged potential satisfies
\begin{equation}\label{wavato2}
\frac{\partial^{2}}{\partial\vartheta^{2}}W_{avg}(0,\zeta;0) =\frac{\partial^{2}W}{\partial q^{2}} (0,\zeta), \qquad\qquad
\frac{\partial^{2}}{\partial\vartheta^{2}}W_{avg}(\pi,\zeta;0) =-\frac{1}{\pi}\frac{\partial W}{\partial q} (\pi,\zeta).
\end{equation}
The fourth derivative satisfies
\begin{equation}\label{wavato4} \begin{array}{ll}
\frac{\partial^4}{\partial\vartheta^4}W_{avg}(0,\zeta;\delta) &
 =-\;\delta^{-\frac{1}{\alpha}}\ Q_d'(0)K(\alpha)\ \frac{\partial^{2}W}{\partial q^{2}} (0,\zeta) \ (1+o(1)_{\delta\to0}) ,\\ \\
\frac{\partial^4}{\partial\vartheta^4}W_{avg}(\pi,\zeta;\delta) &
 =  \;\;\;\delta^{-\frac{1}{\alpha}}\ Q_d'(0)  K(\alpha)\ \frac{\partial^{2}W}{\partial q^{2}} (\pi,\zeta) \ (1+o(1)_{\delta\to0} ), \end{array}
\end{equation}
where
\begin{equation}\label{eq:kofalpha}
K(\alpha)=\int_{2^{1/\alpha}}^{{\infty}}\ \frac{4\alpha^{2}}{ q^{2\alpha+2}\sqrt{1-\frac{2}{q^{\alpha}}}} \ dq \ > \ 0.
\end{equation}
\end{lem}
\begin{proof}
Since $W_{avg}(\vartheta,\zeta;\delta)$ tends to the billiard limit in $C^2$ as $\delta\to 0$, formulas (\ref{wavato2}) are just given by (\ref{avgbil2der}). So, in order to prove (\ref{wavato2}), we only need to calculate the fourth order derivative. As shown in Lemma \ref{lem:wtildc2inthet}, the derivatives of the regular part \((R\) in (\ref{eq:wavgsandr})) are uniformly bounded for all $\delta\geqslant\ 0$, and the fourth order derivative of \(W_{avg}\) at $0$ and $\pi$ is dominated by the derivative of \(S\).
 It follows from (\ref{eq:sfourthderv}) at $\vartheta=0$ (hence, $\hat\vartheta=0$ and $\sigma=0$)
 that
\begin{equation}\nonumber
\begin{array}{l}
\!\!\!\!\!\!\!\!\!\!\!\!
\frac{\partial^{4}}{\partial \vartheta^{4}}  S(\vartheta,\zeta;\delta,\eta)|_{\vartheta=0}=2\omega_0^{-3}\delta^{-\frac{1}{\alpha}} \frac{\partial^{2}W}{\partial q^{2}}(0,\zeta)
\int_{-\eta/(\omega_0\delta^{\frac{1}{\alpha}})}^{\eta/(\omega_0\delta^{\frac{1}{\alpha}})} (3\tilde p_\delta'(u)^{2}+
4 \tilde p_\delta(u) \tilde p_\delta''(u)) du+O(1) \\
\qquad\qquad =- 2\omega_0^{-3}\delta^{-\frac{1}{\alpha}}\ \frac{\partial^{2}W}{\partial q^{2}}(0,\zeta)
\int_{-\eta/(\omega_0\delta^{\frac{1}{\alpha}})}^{\eta/(\omega_0\delta^{\frac{1}{\alpha}})} (\tilde p_\delta'(u))^{2} du
+ 8\omega_0^{-3}\frac{\partial^{2}W}{\partial q^{2}}(0,\zeta)\delta^{-\frac{1}{\alpha}} [\tilde p_\delta(u) \tilde p_\delta'(u)]_{-\eta\omega_0\delta^{-\frac{1}{\alpha}}}^{\eta\omega_0\delta^{-\frac{1}{\alpha}}}+O(1).
\end{array}
\end{equation}
Since  \( \tilde p_\delta'\) decays sufficiently fast by (\ref{derpkal1}),
\begin{equation}\label{eq:s4near0fin}
\begin{array}{l}
\frac{\partial^{4}}{\partial \vartheta^{4}}  S(\vartheta,\zeta;\delta,\eta)|_{\vartheta=0}=- 2\omega_0^{-3}\delta^{-\frac{1}{\alpha}}\ \frac{\partial^{2}W}{\partial q^{2}}(0,\zeta)
\int_{-\infty}^{\infty} (\tilde p_\delta'(u))^{2} du
+O(1).
\end{array}
\end{equation}Using (\ref{eq:pprimebl})\begin{equation}
\int_{-\infty}^{\infty} (\tilde p_\delta'(u))^{2} du=\frac{\alpha^{2}\ }{Q_d'(0)^{2\alpha}}\int_{-\infty}^{+\infty} \tilde q_0(u)^{-2\alpha-2} \ du+o(1),
\end{equation}
where \((\tilde q_0(u),\tilde p_0(u))\)  is the solution of the limit of the Hamiltonian system (\ref{hsdofsp}), namely
$$H=\frac{\tilde p^{2}}{2}+\frac{1}{(Q_d'(0)\tilde q)^{\alpha}}$$
at \(H=\frac{1}{2}\). Choosing symmetric parameterization of the time $u$ (so \(\tilde p_0(0)=0\), hence
\(\tilde q_0(0)=2^{1/\alpha}/Q_d'(0)\)), we obtain
$$\int_{-\infty}^{+\infty} \tilde q_0(u)^{-2\alpha-2}\ du= 2\int_{2^{1/\alpha}/Q_d'(0)}^{+\infty}
\tilde q_0^{-2\alpha-2}\; \frac{d\tilde q_0}{\tilde p_0}=
2 Q_d'(0)^{2\alpha+1} \int_{2^{1/\alpha}}^{+\infty} q^{-2\alpha-2}\ \sqrt{1-\frac{2}{q^{\alpha}}} dq,
$$
so
\begin{equation}\nonumber
\int_{-\infty}^{\infty} (\tilde p_\delta'(u))^{2} du=2 Q_d'(0)\alpha^{2} \int_{2^{1/\alpha}}^{+\infty} q^{-2\alpha-2}\ \sqrt{1-\frac{2}{q^{\alpha}}} dq+o(1).
\end{equation}
Finally, by (\ref{eq:s4near0fin}), and since \(\omega_{0}(0)=1\),
\begin{equation}\nonumber
\frac{\partial^{4}}{\partial \vartheta^{4}}  S(\vartheta,\zeta;\delta,\eta)|_{\vartheta=0}
=-4\delta^{-\frac{1}{\alpha}}\ \frac{\partial^{2}W}{\partial q^{2}}(0,\zeta)\ Q_d'(0)\alpha^{2} \int_{2^{1/\alpha}}^{+\infty} q^{-2\alpha-2}\ \sqrt{1-\frac{2}{q^{\alpha}}} dq+o(\delta^{-\frac{1}{\alpha}}),
\end{equation}
and (\ref{wavato4}) follows at $\vartheta=0$.

When $\vartheta=\pi$ (hence, $\hat\vartheta=0$ and $\sigma=\pi$ in (\ref{eq:sfourthderv})),
equation (\ref{eq:sfourthderv}) gives
\begin{equation}\nonumber
\begin{array}{l}
\frac{\partial^{4}}{\partial\vartheta^{4}}S(\pi,\zeta;\delta,\eta)=2\omega_0^{-3}\delta^{-\frac{1}{\alpha}}
\int_{-\eta/(\omega_0\delta^{\frac{1}{\alpha}})}^{\eta/(\omega_0\delta^{\frac{1}{\alpha}})} \frac{\partial^{2}W}{\partial q^{2}}(\pi-2\delta^{1/\alpha} \tilde q_{\delta}(u),\zeta)\ (3(\tilde p'_{\delta}(u))^{2}+2\
 \tilde p_{\delta}(u)  \tilde p''_{\delta}(u))du  +O(1)\\
\qquad\qquad =2\omega_0^{-3}\delta^{-\frac{1}{\alpha}}
\int_{-\eta/(\omega_0\delta^{\frac{1}{\alpha}})}^{\eta/(\omega_0\delta^{\frac{1}{\alpha}})} \frac{\partial^{2}W}{\partial q^{2}}(\pi,\zeta)\ (3(\tilde p'_{\delta}(u))^{2}+2\
 \tilde p_{\delta}(u)  \tilde p''_{\delta}(u))du  +O(1)\\
\qquad\qquad =2\omega_0^{-3}\delta^{-\frac{1}{\alpha}}\frac{\partial^{2}W}{\partial q^{2}}(\pi,\zeta)
\int_{-\eta/(\omega_0\delta^{\frac{1}{\alpha}})}^{\eta/(\omega_0\delta^{\frac{1}{\alpha}})} \ (\tilde p'_{\delta}(u))^{2}du   +4\omega_0^{-3}\delta^{-\frac{1}{\alpha}} \frac{\partial^{2}W}{\partial q^{2}}(\pi,\zeta)\left[\
 \tilde p_{\delta}(u)  \tilde p'_{\delta}(u)\right]_{-\eta/(\omega_0\delta^{\frac{1}{\alpha}})}^{\eta/(\omega_0\delta^{\frac{1}{\alpha}})}
+O(1)
 \\
 \qquad\qquad =2\omega_0^{-3}\delta^{-\frac{1}{\alpha}}
\frac{\partial^{2}W}{\partial q^{2}}(\pi,\zeta)\int_{-\infty}^{\infty} \ (\tilde p'_{0}(u))^{2}du+O(1).
\qquad\qquad \end{array}
\end{equation}
The last line differs from (\ref{eq:s4near0fin}) only by a factor, so (\ref{wavato4}) follows at
$\vartheta=\pi$.
\end{proof}

The above results allow us to characterize the behavior of the averaged potential \(U(\theta,\xi;\delta)\)
defined by  (\ref{eq:averagedparalpotential}).
It is given by Lemma \ref{lem:wtildc2inthet} that \(U(\theta,\xi;\delta)\) depends on $\delta\geqslant 0$ continuously with the derivatives up to order 2. Therefore, by the non-degenerate minimum assumption Box2, it has a non-degenerate minimum line near the line (\ref{lnqbox}) for all small $\delta$:
$$\theta^{(n)}=\theta^{(n)}_{min}(\delta)+c, \qquad \xi^{(n)}=\xi^{(n)}_{min}(\delta).$$
\begin{lem} \label{lem:expansionupsi}The potential \(U(\theta,\xi;\delta)\) has the following expansion near the minimum line:
\begin{equation}\label{eq:uparhatpsinonscaled}
\begin{split}
U(\theta,\xi_{min}(\delta);\delta)&=
\sum^N_{n=2}\sum^{n-1}_{m=1}\ \gamma_{nm}(\delta)(\theta^{(n)}-\theta^{(m)}-\theta^{(n)}_{min}-\theta^{(m)}_{min})^2\\
&\qquad+\delta^{-\frac{1}{\alpha}} {K(\alpha)}{Q_d'(0)} \beta_{nm}(\delta)(\theta^{(n)}-\theta^{(m)}-\theta^{(n)}_{min}-\theta^{(m)}_{min})^{4}+\ldots,
\end{split}\end{equation}
where \(\gamma_{nm}(\delta)\) and \(\beta_{nm}(\delta)\) tend, as $\delta\to0$, to $\gamma_{nm}$
and, respectively, \(\beta_{nm}\) defined by  (\ref{eq:haveragedthetasyncinabox});  the positive coefficient \(K(\alpha)\) is given by (\ref{eq:kofalpha}), and the dots stand for sixth and higher order terms of the expansion in powers of
\((\theta^{(n)}-\theta^{(m)}-\theta^{(n)}_{min}-\theta^{(m)}_{min})\).
\end{lem}
\begin{proof}
By (\ref{eq:averagedparalpotential}), the derivatives of  \(U(\theta,\xi;\delta)\) with respect to \(\theta^{(n)}\) are given by the sum over the corresponding derivatives of \(W_{avg}(\theta^{(n)} -\theta^{(m)} ,\xi^{(n)}-\xi^{(m)};\delta)\). The second derivatives tend to those of \(W_{avg}(\theta^{(n)} -\theta^{(m)} ,\xi^{(n)}-\xi^{(m)};0)\)  by Lemma \ref{lem:wtildc2inthet},  so the values of \(\gamma_{nm}\) follow from Lemma \ref{lem:wavatorigin}. Since \(W_{avg}(\vartheta,\zeta;\delta)\) is even and \(2\pi\)-periodic in \(\vartheta\), the third order derivatives with respect to \(\vartheta\) vanish
at $\vartheta = 0 \mod \pi$, whereas the fourth order derivatives are given by Lemma \ref{lem:wavatorigin}, which determines the values of $\beta_{nm}$. \end{proof}

\subsection{\label{sec:synchronproof}The simultaneous impacts case}
Since the billiard choreographic solution in the case of simultaneous impacts satisfies, for any pair of particles,
\(\theta_{min}^{(n)}-\theta_{min}^{(m)}=0 \mod \pi\), the averaged potential in the limit $\delta=0$ is not, in general
\(C^{3}\) smooth, as explained in Section \ref{sec:proofsmoothpot}. Therefore, the limit \(\delta=0\) is singular
in this case.
Indeed, we estimated the derivatives of the averaged potential up to order 5 and showed that starting with order 4 they tend to infinity as $\delta\to 0$.

To deal with this difficulty, we prove Theorem \ref{thm:billiardbox2} by using  different arguments from those used to prove Theorems \ref{thm:billiardsystem} and  \ref{thm:billiardbox1}. In Sections \ref{sec:regughat} and  \ref{sec:regularavrg}  we bring the system to the averaged form  (\ref{eq:Hrectavrg}) and estimate the correction term  \(\delta\tilde  G\). In Section \ref{sec:normalformtrun} we bring the truncated averaged system to a Birkhoff normal form; the coefficients of the 4-th order terms in the normal form diverge as $\delta\to 0$, nevertheless we show that it has  a KAM non-degenerate elliptic periodic orbit for all small $\delta>0$.  In Section \ref{sec:normalformfull} we estimate the difference between the Birkhoff normal form for the full and truncated systems and, under the requirement
\(\alpha>6\) show that the full system has an elliptic periodic orbit surrounded by KAM tori.

\subsubsection{\label{sec:regughat}Expansion near the  orbit \(\mathbf{L}^{*}(\theta,\xi)\) of  the uncoupled system. }
\begin{lem} \label{lem:deltaghat}In the case of simultaneous-impacts, provided  \(\alpha>2\) and \(\|\xi^{(n)}-\xi^{(m)}\|>\rho\) for all \(n\neq m\),  the expansion of the scaled Hamiltonian   (\ref{eq:Hrectexpanded}) near \(\mathbf{L}^{*}(\theta,\xi)\)   of (\ref{uncoupledchbox}) is of the form\begin{equation}\label{eq:Hsimulter1}
\begin{split}
H&=\sum^N_{n=1 }\omega_{0}(\delta )I^{(n)}+\delta^{1/2}\sum^N_{n=1 }\left(\frac{1}{2}a(\delta )(I^{(n)})^2 +\sum_{i=1 }^{d-1}\left[\frac{(p_{\xi ,i}^{(n)})^2}{2}+ V_{i}(\xi_{i}^{(n)})\right]\right)\\
& \qquad +\delta^{1/2}
\mathop{\sum_{n,m=1,\dots,N}}_{n\neq m} \ W( q^{*}_{\delta}(\theta^{(n)} )- q^{*}_{\delta}(\theta^{(m)} ),\xi^{(n)}-\xi^{(m)})+\delta\hat  G(I,\theta,\xi;\delta),
\end{split}\end{equation}
 where  the derivatives of \(\delta\hat  G\) which include exactly    \(k\) differentiations with respect to \(\theta\) are bounded by \(O(\delta^{1-k/\alpha})\).
 \end{lem}

\begin{proof}
In the box case, for \(\alpha>2\),  the expansion (\ref{qdeltaaway}) can be improved to\begin{equation}\label{eq:q1deltaoned}
\hat q_{d}(\delta^{1/2}I,\theta ;\delta) = q^*_\delta(\theta) +
\delta^{1/2}G_{d}(\theta,I;\delta)
\end{equation}
where derivatives of \(  G_d\) which include exactly    \(k\) differentiations with respect to \(\theta\) are bounded by \(O(\delta^{-k/\alpha})\).
Indeed, since \(I=0\) corresponds to the periodic orbit $L^*_\delta$, we have that
$\hat q(0,\theta,0):=q^*_\delta(\theta)$. Away from impacts
\(\hat q_{d}(\delta^{1/2}I,\theta;\delta)\)  is a uniformly smooth function of \((\delta^{1/2}I,\theta)\), so (\ref{eq:q1deltaoned}) follows immediately. Near impacts, as the vertical dynamics are one-dimensional, equation (\ref{eq:qhatEPtau}) becomes, in the scaled coordinates (recall that here there are no \(\mathcal{P}\) variables):$$
\hat q _{d}(\delta^{1/2}I,\theta,\delta) = q_{impact}+
\delta^{1/\alpha} \tilde q(\frac{\tau(\delta^{1/2}I,\theta;\delta)\ - t_{in}(\delta)}{\delta^{1/\alpha}},\mathcal{E}
(\delta^{1/2}I);\delta),$$
where \(q_{impact}\) is either \(0\) or \(\pi\). Since \((\tau,\mathcal{E)}\) are smooth functions of \((\theta,I)\), we can expand the above expression at \(I=0:\)
$$
\hat q _{d}(\delta^{1/2}I,\theta,\delta) = q_{\delta}^*(\theta) +
\delta^{1/2} I \bar  G_{d}(\theta,\delta^{1/2-1/\alpha}I;\delta).$$
  Each derivative with respect to \(\theta\) generates a \(\delta^{-1/\alpha}\) factor whereas each derivative with respect to \( I\) generates a \(\delta^{1/2-1/\alpha}\) factor.
Hence, for \(\alpha>2\), the  derivatives of \(  G_d=I\bar G_{d}\) which include exactly    \(k\) derivatives with respect to \(\theta\) are bounded by \(O(\delta^{-k/\alpha})\), as claimed.

Since \(\|\xi^{(n)}-\xi^{(m)}\|>\rho\) for all \(n\neq m\),  the interaction potential \(W\) at any point along \(\mathbf{L}^*\) is smooth, so, plugging (\ref{eq:q1deltaoned})  in the interaction term of   (\ref{eq:Hrectexpanded}) gives
 $$
W(\hat q_{d}(\delta^{1/2}I^{(n)},\theta^{(n)} ;\delta)-\hat q_{d}(\delta^{1/2}I^{(m)},\theta^{(m)} ;\delta),\xi^{(n)}-\xi^{(m)}))=W(q_{\delta}^*(\theta^{(n)} )-q_{\delta}^*(\theta^{(m)} ),\xi^{(n)}-\xi^{(m)}))+$$
$$+\delta^{1/2}\hat G_{d}(\theta^{(n)},\theta^{(m)}, \delta^{1/2}I^{(n)}, \delta^{1/2}I^{(m)},\xi^{(n)}-\xi^{(m)};\delta),$$  where derivatives of \(\hat G_{d}\) which include exactly    \(k\) derivatives with respect to \(\theta^{(n)}\) or \(\theta^{(m)}\) are of order \(\delta^{-k/\alpha}\).
Hence,   (\ref{eq:Hrectexpanded}), becomes
of the required form (\ref{eq:Hsimulter1}), where \(\hat G\) is the sum of such functions  \(\hat G_{d}\) and  regular terms coming from the \(O((\delta^{1/2}I^{(n)})^3)\) terms of (\ref{eq:H1ofIdlta}).
   \end{proof}

\subsubsection{Regularity of averaging}\label{sec:regularavrg}

We average the Hamiltonian (\ref{eq:Hsimulter1}) by using the same transformation (\ref{eq:Psiall}) as in the non-simultaneous impacts case. To this aim, we  establish that \(\Psi\) of  (\ref{eq:Psiall}) and its derivatives have the correct regularity as \(\delta\rightarrow0\). We first show that\\  \begin{equation}\label{eq:psi1box}
 \Psi_{1}(u,u_0,v;\xi^{(n)}-\xi^{(m)};\delta)=\int_{u_{0}}^{u}W(q^*_{\delta}(s+v)-q^*_{\delta}(s-v),\xi^{(n)}-\xi^{(m)})  ds
\end{equation}
satisfies the following result (the weak regularity of Lemma \ref{lem:psi1} is replaced by weaker estimates on the derivatives, thus  allowing both particles to visit simultaneously the impact regions):
\begin{lem}\label{lem:psi1box} Provided \(\|\xi^{(n)}-\xi^{(m)}\|>\rho\), the derivatives  of the function   \(\Psi_{1}(\frac{\theta^{(n)}+\theta^{(m)}}{2},u_{0,}\frac{\theta^{(n)}-\theta^{(m)}}{2},\xi^{(n)}-\xi^{(m)};\delta)\)   which include    \(k>0\) differentiations  with respect to  \(\theta^{(n)},\theta^{(m)}\) are of order \(\delta^{-(k-1)/\alpha}\).
\end{lem}
\begin{proof}
Since \(\|\xi^{(n)}-\xi^{(m)}\|>\rho\), and the \(\xi\) variables are fixed along the integration interval,  \(W\) is  smooth and bounded function of its arguments. So all the derivatives with respect to \(\xi\) are bounded. When both particles are away from impacts (i.e. when the argument of  \(q^*_{\delta}\) is away from \(\{0,\pi\}\)),  all the derivatives of  \(q^*_{\delta}\) are bounded.  Otherwise, by    (\ref{eq:qdeltabound}), the \(k\)-th  derivative of \( q^*_{\delta}(\theta ) \) is of order \(\delta^{-(k-1)/\alpha}\), so, the same property is shared by   \(W(q^*_{\delta}(s+\frac{\theta^{(n)}-\theta^{(m)}}{2})-q^*_{\delta}(s-\frac{\theta^{(n)}-\theta^{(m)}}{2}),\xi^{(n)}-\xi^{(m)}) \) and its integral, \(\Psi_{1}\). \end{proof}\
By  (\ref{eq:wavgispsi1}), the function \(W_{avg}\) has the same regularity properties as \(\Psi_{1}\). Now perform the transformation defined by
  \(\Psi(;\delta)\) of (\ref{eq:Psiall}):  \((I^{(n)},p_{\xi}^{(n)})\rightarrow (I^{(n)},p_{\xi}^{(n)})-\delta^{1/2}(\partial_{\theta^{(n)} },\partial_{\xi^{(n)}})\Psi \). Notice that  here (as opposed to the non-simultaneous impacts case) we apply the transormation to  the  non-truncated Hamiltonian (\ref{eq:Hsimulter1})  (i.e. we do not drop  the correction term \(\delta\hat G\)). Though we  formally obtain the same averaged system (\ref{eq:Hrectavrg}), now the \(\delta \tilde G\) term includes the transformed \(\delta\hat G\) term.  The derivatives with respect to \(\xi\) and \(I\) of the \(\delta\tilde G\)  term  remain of order \(\delta \) whereas its derivatives which include exactly    \(k\) differentiations with respect to
  \(\theta^{(n)}\), \(n=1,\ldots,N\), are of order \(\delta^{1-k/\alpha}\).
  In particular, for \(\alpha>6\) the  correction term  \(\delta\tilde G\) in (\ref{eq:Hrectavrg}) and its derivatives up to order 3 are \(o(\delta^{1/2})\).

\subsubsection{\label{sec:normalformtrun}Normal form of the truncated averaged system.}
Let us study now the truncation of system (\ref{eq:Hrectavrg}),  namely the system (\ref{eq:averagebox}) for the case of simultaneous impacts.
This system is translation-invariant (the Hamiltonian does not change when the same constant is added to all $\theta^{(n)}$).
By Lemma \ref{lem:wtildc2inthet}, the potential  \(U_{\delta}(\theta,\xi)\) in   (\ref{eq:averagebox})  has, even for the simultaneous impacts case, a \(C^{2}\) limit at \(\delta=0\).  By the non-degeneracy assumption Box2,  the potential $U$ has a non-degenerate line of minima for all small $\delta$. This line corresponds to an elliptic periodic orbit\footnote{This orbit is translation invariant, i.e., it is a relative equilibrium.} of system (\ref{eq:averagebox}). Our goal here is to establish the
KAM-nondegeneracy of it. This amounts to bringing the Poincar\'e map near the periodic orbit to the 4-th order Birkhoff normal form (Lemma \ref{lem:hnfdiag} below) and verification of the twist condition (Lemma \ref{lem:nondegcombinednf}).

We introduce local normal coordinates
$(\theta^{(n)},I^{(n)})\rightarrow(\varphi,\psi,P,J)$
near the periodic orbit
in the same way as in (\ref{eq:phipsitrans}),(\ref{eq:pjdef}). Here
$\varphi=\frac{1}{N}(\theta^{(1)}+\ldots+\theta^{(N)})$, and $P$ is symplectically conjugate to
$\varphi$. The variables $\psi$ vary near zero and are linear combinations of $((\theta^{(n)} -\theta^{(n)}_{min}) - (\theta^{(m)} -\theta^{(m)}_{min}))$, ~ $n,m = 1,\ldots, N$; the variables $J$ are symplectically conjugate to $\psi$ and also vary near zero.
As in Section \ref{sec:asynchproof}, the  Poincar\'e return map of (\ref{eq:averagebox})  at the level of fixed \(P\), equals to the \(O(\delta^{1/2})\)-map of the flow defined by the Hamiltonian
\begin{equation}\label{eq:Hrectavrg2simul}
H_{P}(J,\psi,p_{\xi},\xi;\delta)=\frac{1}{2N}a(\delta)J^{2} +\ \frac{p_{\xi }^2}{2} +\hat U_{\delta}(\psi,\xi),
\end{equation}
where  $\hat U_\delta$ is the reduced potential of (\ref{eq:uparhatbox}).
This Hamiltonian is exactly of the same form as \(H_{P}\) of  (\ref{eq:Hrectavrg2}), yet, the regularity properties of  the reduced potential
\(\hat U_{\delta}(\psi,\xi)\) are different in the simultaneous impacts case.

By the parity assumption Box4,  \(W_{avg}(\theta^{(n)}-\theta^{(m)},\xi^{(n)}-\xi^{(m)};\delta)\) is even and $2\pi$-periodic in $\theta$, so the phase differences along the minimum line remain locked at \(\theta^{(n)}= 0 \mod \pi\), ~$n=1,\ldots, N$, for all small \(\delta\).
Therefore,  as \(\psi\) is a linear function of \(\theta^{(n)}-\theta^{(m)}-\theta^{(n)}_{min}-\theta^{(m)}_{min}\), the reduced potential is even in $\psi$:
\begin{equation}\label{hredhatsym}
\hat U_\delta(\psi,\xi)=\hat \hat U_\delta(-\psi,\xi).
\end{equation}
By this symmetry,
\begin{equation}\label{zeroderhatu}
\frac{\partial^{j+l}}{(\partial \psi)^{j}(\partial  \xi)^l}\hat U_\delta(\psi,\xi)|_{\psi=0,\xi=\xi_{min}}=0\quad \mbox{for odd}\;\; j.
\end{equation}
 In particular,    \( \frac{\partial^{2}}{\partial \psi\partial  \xi}\hat U_\delta(\psi,\xi)|_{\psi=0,\xi=\xi_{min}}=0\),  so  at \(\psi=0,\xi=\xi_{min}\) the quadratic part of the Hamiltonian (\ref{eq:Hrectavrg2simul}) is block-diagonal. Therefore,   the expansion of the Hamiltonian (\ref{eq:Hrectavrg2simul}) up to order four terms is of the form
\begin{equation}\label{155}
H_{4jet}=H^z_{2}(z,p;\delta)+H^\theta_{2}(\psi,J;\delta)+H_{3}^{\theta z}(\psi ,z;\delta)+H^{z}_3(z;\delta)+H^{z}_4(z;\delta)+H_{4}^{\theta z}(\psi ,z;\delta)+\delta^{-\frac{1}{\alpha}} H_{4}^{\theta }(\psi;\delta),
\end{equation}
where,  hereafter, \(z=\xi-\xi_{min},p=p_\xi\), and
\(H_{j}\) denotes a homogeneous polynomial of order \(j\). By Lemma \ref{lem:expansionupsi}, the coefficients of
\(H_{j}\) are uniformly bounded for all \(\delta\geqslant0\).

Define the \(z\)-Hamiltonian:
\begin{equation}\label{eq:hzlem}
H^{z}(z,p;\delta)= \frac{p^2}{2}+\hat  U_{\delta}(0,\xi_{min}+z)=H^z_{2}(z,p;\delta)+H^{z}_3(z;\delta)+H^{z}_4(z;\delta)+O(z^{5}),\qquad
\end{equation}
and the  \(\theta\)-Hamiltonian:
\begin{equation}\label{eq:hpsilem}
H^{\theta}(\psi ,J;\delta)=\frac{1}{2N}aJ^{2} +\hat  U_{\delta}(\psi,\xi_{min})=H^\theta_{2}(\psi,J;\delta)+\delta^{-\frac{1}{\alpha}} H_{4}^{\theta }(\psi;\delta)+\ldots,\qquad
\end{equation}
where,  hereafter, the dots stand for terms of order 5 and higher (with coefficients that  may diverge as
\(\delta\rightarrow0\)). Then the Hamiltonian (\ref{eq:Hrectavrg2simul}) is of the form
\begin{equation}\label{eq:HPzthet}
H_{P}=H^{z}(z,p_{};\delta)+H^{\theta}(\psi,J;\delta)+H_{3}^{\theta z}(\psi ,z;\delta)+H_{4}^{\theta z}(\psi ,z;\delta)+ \ldots,
\end{equation}
where, by the symmetry \(\psi \rightarrow-\psi\) (see (\ref{hredhatsym})), the polynomials \(H_{3,4}^{\theta z}(\psi ,z;\delta)\) have only monomials which are quadratic in $\psi$.

Denote the  fourth order Birkhoff normal form of the Hamiltonian (\ref{eq:hzlem})  near the fixed point at the origin by
\(H^{z}_{NF}(z,p;\delta)\) and  the  Birkhoff normal form of the Hamiltonian  (\ref{eq:hpsilem}) by \(H^{\theta}_{NF}(\psi,J;\delta)\). Recall that in such normal forms all the terms up to order 4 are resonant -- all non-resonant terms up to order 4 are eliminated by a sequence of symplectic coordinate transformations. Each of these transformations is a time-1 map of a certain polynomial Hamiltonian whose coefficients depend smoothly on the coefficients of the 4-jet of the original system, see e.g. Chapter 7 of \cite{Arnold2007CelestialMechanics}. It is a well-known fact that the normalizing transformations can be chosen such that the linear symmetries of the system are preserved. In particular, the symmetry $(\psi,J)\to -(\psi,J)$ survives the transformations we describe below.

\begin{lem} \label{lem:normalformtheta}There exists a  symplectic  transformation which depends continuously on
\(\delta>0\) and brings
 the Hamiltonian \(H^{\theta}\) of (\ref{eq:hpsilem}) to  the normal form
 \begin{equation}\label{eq:HNFtheta}
H^{\theta}_{NF}(\psi,J;\delta)=\tilde H^{\theta}_{2}(\psi,J;\delta)+\delta^{-\frac{1}{\alpha}} K(\alpha)\tilde H^{\theta}_{4}(\psi,J;\delta)+\dots \ ,
\end{equation} where \(\tilde H^{\theta}_{k}\) are homogeneous polynomials of \((\psi,J)\) of degree \(k\) ($k=2,4$) with bounded coefficients depending continuously on \(\delta\). The polynomials  \(\tilde H^{\theta}_{2,4}\)  have a limit at  \(\delta=0\) such that the  normal form of the  Hamiltonian   \(H^{\theta,0}\) of Assumption Box5 is
\begin{equation}\label{eq:NFthetadelta0}
H^{\theta,0}_{NF}(\psi,J)=\tilde H^{\theta}_{2}(\psi,J;0)+\tilde H^{\theta}_{4}(\psi,J;0)+\ldots ~.
\end{equation}
\end{lem}
\begin{proof}  The normal form transformation we are going to describe is a composition of several symplectic transformations. First, we do a linear transformation which diagonalizes the  quadratic part of the  \(\theta\)-Hamiltonian (\ref{eq:hpsilem}). By Lemma \ref{lem:expansionupsi},  the quadratic part of the Hamiltonian (\ref{eq:hpsilem}) is, in the limit $\delta=0$, identical to the quadratic part of \(H^{\theta,0}\). Therefore, by Assumption Box2, all the frequencies are distinct for small $\delta$; hence the quadratic part of (\ref{eq:hpsilem}) is indeed symplectically diagonalizable and the limit of the diagonalizing  transformation at $\delta=0$ also diagonalizes the quadratic part of
\(H^{\theta,0}\). Since the transformation is linear, it does not introduce third order terms.

Thus, the Hamiltonian (\ref{eq:hpsilem}) becomes \(H^{\theta}(\psi,J;\delta)=\tilde  H^{\theta}_{2}(\psi,J;\delta)+\delta^{-\frac{1}{\alpha}}K(\alpha)\hat H^{\theta}_{4}(\psi,J;\delta)+\ldots\),  where  $\hat H^{\theta}_{4}(\psi,J;0)$ coincides with the 4-th order terms of \(H^{\theta,0}\) after the diagonalization.  Next, we do a symplectic  transformation
 in order to eliminate all non-resonant fourth order terms, i.e., to bring the \(\theta\)-Hamiltonian to its normal form
 \(H_{NF}^{\theta}\). It is well-known (see \cite{Arnold2007CelestialMechanics}, Chapter 7) that such normalizing transformation does not alter the resonant terms of order 4 or lower, hence, \(\tilde  H^{\theta}_{4}(\psi,J;\delta)\) is obtained from   $\hat H^{\theta}_{4}(\psi,J;\delta)$ by throwing away the non-resonant terms. It follows that  even though the normalizing  transformation does not have a limit\footnote{The transformation is the time-1 map, $(\psi, J) \to (\psi, J)+\delta^{-\frac{1}{\alpha}}K(\alpha) (\partial_J S, -\partial_\psi S)+O((\psi,J)^4)$, of a Hamiltonian flow defined by the Hamilton function $\delta^{-\frac{1}{\alpha}}K(\alpha) S(\psi,J;\delta)$, where $S$ is a fourth-degree homogeneous polynomial with bounded coefficients which solves the equation $\{\tilde H_2^{\theta},S\}=\tilde H^{\theta}_{4}-\hat H^{\theta}_{4}$. The normalizing transformation for \(H^{\theta,0}\) is the time-1 map for the Hamiltonian $S(\psi,J;0)$.} at \(\delta=0\), the term \(\tilde  H^{\theta}_{4}(\psi,J;\delta)\) has a limit, equal to the resonant part of
\(\hat H^{\theta}_{4}(\psi,J;0)\), so the lemma follows.
\end{proof}

\begin{lem}
\label{lem:hnfdiag}The Birkhoff normal form of  the reduced Hamiltonian (\ref{eq:Hrectavrg2simul}) is the sum of the normal forms of the Hamiltonians (\ref{eq:hzlem}) and (\ref{eq:hpsilem}) with bounded corrections of order 4 which vanish at \((J,\psi)=0\):
\begin{equation}\label{eq:Hnormalformlem}
H_{NF}=H^{z}_{NF}(z,p;\delta)+H^{\theta}_{NF}(\psi,J;\delta)+H_{2}^{\psi J}(\psi,J;\delta)H_2^{zp}(z,p;\delta)+H_4^{\psi J}(\psi,J;\delta)+\ldots \ .
\end{equation}\end{lem}
\begin{proof}
Consider the  symplectic transformation which brings the Hamiltonian  \(H^{z}=H^z_{2}(z,p;\delta)+H^{z}_3(z;\delta)+H^{z}_4(z;\delta)+ \dots \) to its Birkhoff normal form up to order 4:
\begin{equation}\label{eq:hnormalformz}
H^{z}_{NF}(z,p;\delta)=\tilde H^{z}_{2}(z,p;\delta)+\tilde H^{z}_{4}(z,p;\delta)+\ldots \ .
\end{equation}
Since, by Lemma \ref{lem:wtildc2inthet}, the coefficients of all monomials in \(H^{z}\) up to order 4 have well-defined finite limits as \(\delta\rightarrow0\), this transformation also has a well-defined finite limit.  Applying this transformation to \(H_{P}\) of  (\ref{eq:Hrectavrg2simul}) (equivalently, to (\ref{eq:HPzthet})), we obtain, in the new coordinates,
\begin{equation}\label{164}
\begin{array}{l}
 H_P = \tilde H^{z}_{2}(z,p;\delta)+\tilde H^{z}_{4}(z,p;\delta)+H^\theta_{2}(\psi,J;\delta)+\hat H_{3}^{\theta z}(\psi ,z,p;\delta)\\ \qquad\qquad \qquad+\hat H_{4}^{\theta z}(\psi ,z,p;\delta)+\delta ^{-1/\alpha}H_{4}^{\theta }(\psi ;\delta)+ \ldots,
\end{array}
\end{equation}
where all monomials in \(\hat H_{3,4}^{\theta z}(\psi ,z,p;\delta)\) are quadratic in $\psi$
(because of the symmetry $\psi\to -\psi$). Since all the terms in  \(\hat H_{3}^{\theta z}(\psi ,z,p;0)\)   are  non-resonant by Assumption Box5 (as at the minimum of \(\hat U\), the frequencies have no resonances of third order), they remain non-resonant for small \(\delta\). Therefore, there exists a symplectic  transformation  \((z,p)\rightarrow(z,p)+O((\psi,J)^{2}), (\psi,J)\rightarrow(\psi,J)+
O((z,p))\cdot(\psi,J) \) which eliminates the cubic terms and brings the Hamiltonian \( H_P\) to the form
\begin{equation}\label{165})
H_P = \tilde H^{z}_{2}(z,p;\delta)+\tilde H^{z}_{4}(z,p;\delta)+H^\theta_{2}(\psi,J;\delta)+ {\bar H}_{4}^{\theta z}(\psi ,J,z,p;\delta)+\bar H_{4}^{\theta}(\psi ,J;\delta)+\delta ^{-1/\alpha}H_{4}^{\theta }(\psi ;\delta)+\ldots,
\end{equation}
where \({\bar H}_{4}^{\theta z}(\psi ,J,z,p;\delta)\) is a sum of monomials  quadratic in \((\psi,J)^{2}\) and
\((z,p)^{2}\).

Now, we apply to $H_P$  the transformation of Lemma \ref{lem:normalformtheta}, which brings the Hamiltonian
\(H^\theta_{2}(\psi,J;\delta)+\delta ^{-1/\alpha}H_{4}^{\theta }(\psi ;\delta)+ \dots \) to its Birkhoff normal form.
This brings $H_P$ to the form (\ref{eq:Hnormalformlem}), with some of the 4th order terms in
$H_{2}^{\psi J}$, $H_2^{zp}$, and $H_4^{\psi J}$ possibly non-resonant. All such non-resonant terms are eliminated by a symplectic transformation without changing all other terms up of order four or less.
\end{proof}

Because there are  no resonances up to order 4 by Assumption Box5, the 4-jets of the Birkhoff normal forms depend only on actions. Namely,
we introduce actions
\begin{equation}\label{actionsIzpsi}
\begin{array}{l}
I_{z_{j}}=\frac{1}{2}(z_{j}^{2}+p_{j}^2), \qquad j=1,\ldots N(d-1),\\
I_{\psi_{k}}=\frac{1}{2}(\psi_{k}^{2}+J_{k}^2), \qquad k=1, \dots,N-1.
\end{array}
\end{equation}
Then
 \begin{equation}\nonumber\begin{array}{ll}
H^{z}_{NF}&=\omega_{z}(\delta)I_{z}+I_{z}^{T}B_{z}(\delta) I_{z} +\ldots,\\
 H^{\theta}_{NF}&=\omega_{\psi}(\delta)I_{\psi}+\delta ^{-1/\alpha}K(\alpha )I_{\psi}^{T}B_{\psi}(\delta)I_{\psi}+
 \ldots,
 \end{array}
\end{equation}
where, by Lemmas \ref{lem:wtildc2inthet} and \ref{lem:normalformtheta}  \(\omega_{z}(\delta)\), \(\omega_{\psi}(\delta)\), $B_{z}(\delta)$, $B_{\psi}(\delta)$ all are continuous functions with bounded limits at \(\delta=0\).
By Assumption Box5, the determinants of \(B_{z}(0)\) and \(B_{\psi}(0)\) are non-zero.

Moreover, by Lemma \ref{lem:normalformtheta}, the Hamiltonian \(H^{\theta,0}_{NF}\) is of the form
\begin{equation}\nonumber
 H^{\theta,0}_{NF}=\omega_{\psi}(0)I_{\psi}+I_{\psi}^{T}B_{\psi}(0)I_{\psi}+\ldots\ .
\end{equation}
Also, by Lemma \ref{lem:hnfdiag},
the Birkhoff normal form \(H_{NF}\) of  (\ref{eq:Hrectavrg2simul}) is given by
\begin{equation}\nonumber
H_{NF}=H^{z}_{NF}+H^{\psi}_{NF}+O(\|I_{z}\| +\|I_{\psi}\|)\cdot\ I_\psi + \ldots\ .
\end{equation}

\begin{lem}\label{lem:nondegcombinednf}
The Birkhoff normal form of the Hamiltonian  (\ref{eq:Hrectavrg2simul}) is  non-degenerate.
\end{lem}\begin{proof}
We need to check that the twist condition for \(H_{NF}\) is satisfied. Indeed,
\begin{equation}\begin{array}{ll}
\det(\frac{\partial ^{2}H_{NF}}{\partial I^{2}})&=\begin{vmatrix}B_{z}(\delta) & O(1) \\
O(1) & \delta ^{-1/\alpha}K(\alpha )B_{\psi}(\delta)+O(1) \\
\end{vmatrix}=\delta ^{-2N/\alpha}K(\alpha )^{2N}\begin{vmatrix}B_{z}(\delta) & O(1) \\
O(\delta ^{1/\alpha}) & B_{\psi}(\delta)+O(\delta ^{1/\alpha})
\end{vmatrix} \\ \\&=\delta ^{-2N/\alpha}K(\alpha )^{2N}(|B_{z}(\delta)|\cdot |B_{\psi}(\delta)|+O(\delta ^{1/\alpha})).
\end{array}
\end{equation}
Since the determinants \(|B_{z}(\delta)|\) and \(|B_{\psi}(\delta)|\) are bounded away from zero for all sufficiently small \(\delta\), we have \(\det(\frac{\partial ^{2}H_{NF}}{\partial I^{2}})\neq~0\), i.e. the twist condition is satisfied, so \(H_{NF}\) is non-degenerate. \end{proof}

\subsubsection{\label{sec:normalformfull}Normal form of the full system.}

We have shown that the truncation  of   system (\ref{eq:Hrectavrg}),  namely system (\ref{eq:averagebox}) for the case of simultaneous impacts, has a stable KAM non-degenerate elliptic periodic orbit for all \(\delta>0\).
To complete the proof of Theorem \ref{thm:billiardbox2}, we need to show that the periodic orbit persists in the full system  (\ref{eq:Hrectavrg}), which includes the error term  \(\delta \tilde  G\), and also remains KAM-non-degenerate. This is non-trivial, as both the normal form and the error term
\(\delta \tilde  G\) diverge in \(C^{4}\) as \(\delta\rightarrow0\). While the persistence of the periodic orbit follows from the \(C^{2}\) regularity of the normal form and the error term,
the KAM non-degeneracy is established by revisiting the normal form computation of Section \ref{sec:normalformtrun} applied to the full system.

Rewriting  (\ref{eq:Hrectavrg}) with the use of the coordinates \((P,\varphi ,J,\psi,p_{\xi},\xi)\) of   (\ref{eq:Hrectavrg1}), we obtain the perturbed Hamiltonian:
 \begin{equation}\label{eq:Hrectavrg1plg}
\begin{array}{ll}
H(P,\varphi ,J,\psi,p_{\xi},\xi;\delta)&=\omega_{0}P+\delta^{1/2}(\frac{1}{2N}aP^{2}+\frac{1}{2N}aJ^{2} + \frac{p_{\xi}^2}{2} +\hat  U_{\delta}(\psi,\xi))\\ &\qquad\qquad+\delta \tilde G(P,\varphi ,J,\psi,p_{\xi},\xi;\delta).
\end{array}
\end{equation}
The corresponding system is:
\begin{equation}\begin{array}{ll}
\dot \varphi &=\omega_{0}+\frac{1}{N}a\delta^{1/2}P+\delta \frac{\partial  }{\partial P  } \tilde G(P,\varphi ,J,\psi,p_{\xi},\xi;\delta), \\
\dot P
 &= -\delta \frac{\partial  }{\partial \varphi  } \tilde G(P,\varphi ,J,\psi,p_{\xi},\xi;\delta),\\
       \dot \psi &=\frac{1}{N}a\delta^{1/2}J+\delta \frac{\partial  }{\partial J  } \tilde G(P,\varphi ,J,\psi,p_{\xi},\xi;\delta), \\
\dot J
 &=- \delta^{1/2} \frac{\partial  }{\partial \psi\  }\hat  U_{\delta}(\psi,\xi)-\delta \frac{\partial  }{\partial \psi\  } \tilde G(P,\varphi ,J,\psi,p_{\xi},\xi;\delta),\\
       {\dot \xi} &=\delta^{1/2}\ p_{\xi}+\delta \frac{\partial  }{\partial p_{\xi}  } \tilde G(P,\varphi ,J,\psi,p_{\xi},\xi;\delta),\\{\dot p_{\xi}} &=- \delta^{1/2} \frac{\partial  }{\partial \xi\  }\hat  U_{\delta}(\psi,\xi)-\delta \frac{\partial  }{\partial \xi\  } \tilde G(P,\varphi ,J,\psi,p_{\xi},\xi;\delta).\\
       \end{array}
\end{equation}
Restricting to a given energy level (so \(P\) is determined by all other variables) and using \(\bar\varphi=\delta^{1/2}\varphi\) as a new time:\begin{equation}\label{eq:systemofreduced}\begin{array}{ll}
        \frac{d \psi}{d\bar\varphi} &=\frac{1}{N\omega_{0}}aJ+\delta^{1/2}   G_1(\bar\varphi ,J,\psi,p_{\xi},\xi;\delta), \\
\frac{d J}{d\bar\varphi}
 &=-\frac{1}{\omega_{0}}  \frac{\partial  }{\partial \psi\  }\hat  U_{\delta}(\psi,\xi)  -\delta ^{\frac{1}{2}-\frac{1}{\alpha}}  G_{2}(\bar\varphi ,J,\psi,p_{\xi},\xi;\delta),\\
       \frac{d\xi}{d\bar\varphi} &=\ \frac{1}{\omega_{0}}p_{\xi}+\delta^{1/2}   G_{3}(\bar\varphi ,J,\psi,p_{\xi},\xi;\delta),\\
\frac{d p_{\xi}}{d\bar\varphi}  &=-\frac{1}{\omega_{0}} \frac{\partial  }{\partial \xi\  }\hat  U_{\delta}(\psi,\xi)- \delta^{1/2}   G_{4}(\bar\varphi ,J,\psi,p_{\xi},\xi;\delta),
\end{array}\end{equation}
where  \(\frac{\partial  }{\partial \psi^{k} }  G_{i}(\bar\varphi ,J,\psi,p_{\xi},\xi;\delta)=O(\delta ^{-k/\alpha}) \),  see the end of Section  \ref{sec:regularavrg}. The reduced system  (\ref{eq:systemofreduced}) is a fast-oscillating non-autonomous  Hamiltonian system with the Hamilton function
\begin{equation}\label{eq:Hredg5del}
\begin{array}{ll}
H_{reduced}(\bar\varphi ,J,\psi,p_{\xi},\xi;\delta)&=\frac{1}{2N\omega_{0}}aJ^{2} +\ \frac{p^2}{2\omega_{0}} + \frac{1}{\omega_{0}}\hat  U_{\delta}(\psi,\xi)+ \delta^{1/2} G_5(\bar\varphi ,J,\psi,p_{\xi},\xi;\delta)\\
&=\frac{1}{\omega_{0}}H_{P}(J,\psi,p_{\xi},\xi;\delta)+ \delta^{1/2} G_5(\bar\varphi ,J,\psi,p_{\xi},\xi;\delta),
\end{array}
\end{equation}
where \(H_{P}\) is the Hamiltonian (\ref{eq:Hrectavrg2simul}). As above, \(G_5\) is a function with bounded derivatives with respect to \(J,p_{\xi},\xi\), whereas making \(k\) differentiations of \(G_5\) with respect to
\(\psi\) introduces a factor of \(O(\delta ^{-k/\alpha})\). Now, as apposed to  (\ref{eq:Hrectavrg2simul}),  $H_{reduced}$ is
a periodic function of the time variable $\bar\varphi$, with the period $2\pi \delta^{1/2}$.

\begin{lem}\label{lem:lastNF}
Provided that \(\alpha>6\), the system (\ref{eq:systemofreduced}) has a KAM-nondegenerate elliptic periodic orbit. \end{lem}
\begin{proof}
As we mentioned in Section \ref{sec:normalformtrun}, due to the non-resonance assumption of Box2 and the continuous dependence of the quadratic part of the Hamiltonian \(H_{P}(J,\psi,p_{\xi},\xi;\delta)\) on
\(\delta\), the truncated system (\ref{eq:Hrectavrg2simul}) has a non-resonant elliptic fixed point at the origin.  Since for \(\alpha>4\) the perturbation to (\ref{eq:Hrectavrg2simul}),
i.e. the term \( \delta^{1/2} G_5\) in the Hamiltonian (\ref{eq:Hredg5del}), is \(C^{2}\)-small, the system (\ref{eq:systemofreduced}) has, for all small \(\delta,\) an elliptic periodic orbit
\((J,\psi,p_{\xi},\xi)=\delta ^{\frac{1}{2}-\frac{1}{\alpha}}  X_{p}(\bar\varphi;\delta)\)
where \(X_{p}(\bar\varphi;\delta)\) is a continuous bounded function for all \(\delta\geqslant0\).
Let  \((J,\psi,p_{\xi},\xi)=\delta ^{\frac{1}{2}-\frac{1}{\alpha}}  X_{p}(\bar\varphi;\delta)+X\).
Then
\begin{equation}\label{eq:hreduced}
H(\bar\varphi ,X;\delta ):=H_{reduced}(\delta ^{\frac{1}{2}-\frac{1}{\alpha}}  X_{p}(\bar\varphi;\delta)+X;\delta)=H_{P}(X;\delta)+ G_{6}(\bar\varphi ,X;\delta),
\end{equation}
where the non-autonomous term \(G_{6}\) is given by
\[ G_{6}(\bar\varphi ,X;\delta)=H_{P}(\delta ^{\frac{1}{2}-\frac{1}{\alpha}}  X_{p}(\bar\varphi;\delta)+X;\delta)-H_{P}(X;\delta)+ \delta ^{\frac{1}{2}} G_{5}(\bar\varphi ,\delta ^{\frac{1}{2}-\frac{1}{\alpha}}  X_{p}(\bar\varphi;\delta)+X;\delta).\]
Next we show that  if \(\alpha>6\), the term \(G_{6}(\bar\varphi,X;\delta)\) and its derivatives with respect to \(X\) including up to 3 derivatives with respect to \(\psi\) are small, while its 4th  derivatives with respect to \(\psi\) are \(o(\delta ^{-1/\alpha})\):
\begin{equation}\label{g6est}
\frac{\partial^{k+l}}{\partial \psi^{k} \partial (J,p_{\xi},\xi)^l } G_6 = O(\delta ^{\frac{1}{2}-\frac{k}{\alpha}}),\quad\ k=0,1,2,3,4.
\end{equation}
Indeed, first recall that \( \frac{\partial}{\partial \psi^{k}} G_5\)  is
\(O(\delta^{-\frac{k}{\alpha}}),\) so for  \(\alpha>6\) the term \( \delta ^{\frac{1}{2}} G_{5}\)  gives a correct contribution to the estimates (\ref{g6est}). Next, by Lagrange formula,
\begin{equation}\label{lagrange}
H_{P}(\delta ^{\frac{1}{2}-\frac{1}{\alpha}}  X_{p}(\bar\varphi;\delta)+X;\delta)-H_{P}(X;\delta)=\delta ^{\frac{1}{2}-\frac{1}{\alpha}}X_{p}(\bar\varphi;\delta)\cdot\int_{0}^{1}\nabla H_{P}(X+s\delta ^{\frac{1}{2}-\frac{1}{\alpha}}  X_{p}(\bar\varphi;\delta);\delta)ds. \nonumber
\end{equation}
Recall that by Lemma \ref{lem:wtildc2inthet} the potential \(\hat U_{\delta}(\psi,\xi)  \), hence the Hamiltonian \(H_{P}\), has bounded derivatives up to order 3, yet its fourth derivatives with respect to
\(\psi\) are of order \(\delta ^{-\frac{1}{\alpha}}\) and its fifth derivatives are of order
\(\delta^{-\frac{2}{\alpha}}\). Hence, by (\ref{lagrange}),
\(\frac{\partial  }{\partial \psi^{k} }(H_{P}(\delta ^{\frac{1}{2}-\frac{1}{\alpha}}  X_{p}(\bar\varphi;\delta)+X;\delta)-H_{P}(X;\delta)))\) for \(k=0,1,2\) is of order  \(O(\delta ^{\frac{1}{2}-\frac{1}{\alpha}}) \), whereas for \(k=3\)   the derivatives are of order \(\delta ^{\frac{1}{2}-\frac{2}{\alpha}} \), and for  \(k=4  \) they are of order  \(\delta ^{\frac{1}{2}-\frac{3}{\alpha}}\), all in agreement with
(\ref{g6est}).

In order to align with the notations of Section \ref{sec:normalformtrun}, we denote hereafter the coordinates of \(X\) by \((J,\psi,p,z)\). Then, by (\ref{eq:HPzthet}) and (\ref{eq:hreduced})
\begin{equation}\label{eq:hpnormstep1}\begin{array}{l}
H(\bar\varphi ,X;\delta)=H^z_{2}(z,p;\delta)+H^{z}_3(z;\delta)+H^{z}_4(z;\delta)+H^\theta_{2}(\psi,J;\delta)+H_{3}^{\theta z}(\psi ,z;\delta) \\ \qquad\qquad\qquad+H_{4}^{\theta z}(\psi ,z;\delta)+\delta ^{-1/\alpha}H_{4}^{\theta }(\psi ;\delta)+
 G_6(\bar\varphi ,X;\delta),\end{array}
\end{equation}
where the functions \(H_{j}\) of  (\ref{eq:HPzthet}) are homogeneous polynomials of degree \(j \) with bounded coefficients (continuous in \(\delta\)) and  \(H_{3,4}^{\theta z}(\psi ,z;\delta)\) have only monomials quadratic in \(\psi\).

Recall  that  \(H_{P}\)  was brought to its normal form by a sequence of coordinate transformations described in Lemma \ref{lem:hnfdiag}. Next, we apply similar transformation to the full Hamiltonian
\(H(\bar\varphi ,X;\delta)\) of (\ref{eq:hreduced}). The difference is that $H$ is periodic in time
$\bar\varphi$, so the normalizing transformations are also periodic in $\bar\varphi$. It is well-known that
one can make the coordinate transformations such that the resulting Hamiltonian will become autonomous
up to any given order in $X$, i.e. such a transformation  corresponds to averaging to any given order \cite{Arnold2007CelestialMechanics}.   Our goal is to bring
$H(\bar\varphi, X;\delta)$ to an autonomous normal form up to order $4$ and compare it
with the 4-th order normal form $H_{NF}$ (see (\ref{eq:Hnormalformlem})) of the autonomous Hamiltonian $H_P$.
Since the period is small, of order \(\delta^{1/2}\), there are no additional resonant terms due to the dependence on  \(\bar \varphi\). Still, one needs to check that the singular terms in (\ref{eq:hpnormstep1}) do not destroy the twist condition.

Let us make the linear symplectic transformations which bring the Hamiltonians
\(H^\theta_{2}(\psi,J;\delta)\) and  \(H^z_{2}(z,p;\delta)\)  to the diagonal form \(\tilde H^\theta_{2}(\psi,J;\delta)\) and \(\tilde H^z_{2}(z,p;\delta) \) as in (\ref{eq:HNFtheta}) and (\ref{eq:hnormalformz}),
respectively. Applying these transformations to (\ref{eq:hpnormstep1}), we obtain a new Hamiltonian
\begin{equation}\label{eq:normalstep2}\begin{array}{ll}
 H (\bar\varphi ,X;\delta)=\tilde H^z_{2}(z,p;\delta) + H^{z}_{3}(z,p;\delta) + H^{z}_{4}(z,p;\delta) + \tilde H^\theta_{2}(\psi,J;\delta) + H_{3}^{\theta z}(\psi ,J,z,p;\delta)\\\qquad \qquad \qquad + H_{4}^{\theta z}(\psi ,J,z,p;\delta)+\delta ^{-1/\alpha}H_{4}^{\theta }(\psi,J;\delta ) + G_{7}(\bar\varphi ,X;\delta), \end{array}
\end{equation}
where \(H_{j}\)  are different functions from those in (\ref{eq:hpnormstep1}), yet they keep the same structure - they are homogeneous polynomials of degree \(j \) with bounded coefficients (continuous in
\(\delta\)), and \(H_{3,4}^{\theta z}\) have only monomials quadratic in \((\psi,J)\). Note also that since the linear normalizing transformations have bounded coefficients for all \(\delta\geqslant0\), the derivatives of the function \(G_{7}\) are of the same order as the derivatives of \(G_{6}\), as given by (\ref{g6est}).

Since \(X=0\) is a periodic orbit of the Hamiltonian \(H(\bar\varphi,X;\delta)\), it follows that the expansion of
\(G_7(\bar\varphi ,X;\delta)\) in powers of $X$ starts with quadratic terms. Moreover, the terms of the expansion which are independent of $(\psi,J)$ are of order $\delta^{1/2}$, the terms linear in $(\psi,J)$
are of order $\delta^{1/2-1/\alpha}$, and so on, e.g. the 4-th order terms in $(\psi,J)$ are
$O(\delta^{1/2-4/\alpha})$, i.e., $o(\delta^{-1/\alpha})$ since $\alpha>6$. Therefore, expanding \(G_{7}\) in powers of \(X\) we can rewrite (\ref{eq:normalstep2}) as
\begin{equation}\label{eq:normalstep2.5}\begin{array}{l}
 H (\bar\varphi ,X;\delta)=
\tilde H^z_{2}(z,p;\delta)+ \tilde H^\theta_{2}(\psi,J;\delta)+
\hat H_{2}^{\theta z}(\bar\varphi,\psi,J,z,p;\delta)\\\qquad \qquad +
\hat H^{z}_{3}(\bar\varphi,z,p;\delta)+\hat H_{3}^{\theta z}(\bar\varphi,\psi ,J,z,p;\delta)+\hat H^{z}_{4}(\bar\varphi ,z,p;\delta) +\hat H_{4}^{\theta z}(\bar\varphi,\psi ,J,z,p;\delta)
\\\qquad \qquad \qquad+\delta^{-1/\alpha}\hat H_{4}^{\theta }(\bar\varphi,\psi,J;\delta)+\ldots, \end{array}
\end{equation}
where, hereafter, the dots stand for the terms of order higher than $4$ (i.e., they are irrelevant for our purposes).
The terms $\hat H_j$ are homogeneous polynomials of $X$ of degree $j$ with periodic in $\bar\varphi$ coefficients, continuous and bounded for all $\delta\geqslant 0$. Moreover, \(\hat H_2^{\theta z}\) is
\(O(\delta ^{\frac{1}{2}-\frac{2}{\alpha}})\)-close to zero. The polynomials $\hat H_{3}^z$ and
$\hat H_{4}^z$ are  the cubic and, respectively quartic part of $H(\bar\varphi ,X;\delta)$ at $(\psi,J)=0$, so
they are $O(\delta^{1/2})$-close to  $H^{z}_{3,4}(z,p;\delta)$ of (\ref{eq:normalstep2}).
The  polynomial
$\hat H_{3}^{\theta z}$ is \(O(\delta ^{\frac{1}{2}-\frac{3}{\alpha}})\)-close to $H_{3}^{\theta z}$ of (\ref{eq:normalstep2}) (indeed, it has a linear in $(\psi,J)$ part which is  \(O(\delta ^{\frac{1}{2}-\frac{1}{\alpha}})\)-close to
zero, since \(H_{3}^{\theta z}\) has no linear part in \(\psi\), a quadratic in $(\psi,J)$ part which is
\(O(\delta ^{\frac{1}{2}-\frac{2}{\alpha}})\)-close to $H_{3}^{\theta z}$, and the cubic  in $(\psi,J)$ part is, again,
\(O(\delta ^{\frac{1}{2}-\frac{3}{\alpha}})\)-close to
zero). Similarly, the  polynomials  $\hat H_4^{\theta z}$ and $\hat H_4^\theta$ are \(O(\delta ^{\frac{1}{2}-\frac{3}{\alpha}})\)-close to $H^{\theta z}_4$ and $H_4^{\theta}$, respectively. Since $\alpha>6$, all these corrections are small, i.e.,
they vanish at $\delta=0$.

Since $\hat H_{2}^{\theta z}=O(\delta^{1/2-2/\alpha})$, the $\bar\varphi$-periodic, linear symplectic transformation of the variables $X$ which diagonailzes the quadratic part ($\tilde H^z_{2}+ \tilde H^\theta_{2}(\psi,J;\delta)+
\hat H_{2}^{\theta z}$) of $H$ and makes the quadratic part independent  of \(\bar\varphi\) is
\(O(\delta^{\frac{1}{2}-\frac{2}{\alpha}})\)-close to identity:
\(X\rightarrow(Id+\delta ^{\frac{1}{2}-\frac{2}{\alpha}}M(\bar\varphi;\delta))X\), for some periodic
in $\bar\varphi$ matrix $M$,
continuous and bounded for all $\delta\geqslant 0$.

After this transformation, the Hamiltonian (\ref{eq:normalstep2.5}) becomes
\begin{equation}\label{eq:normalstep3}\begin{array}{l}
 H (\bar\varphi ,X;\delta)=  \hat H^z_{2}(z,p;\delta)+ \hat H^\theta_{2}(\psi,J;\delta)+\hat H^{z}_{3}(\bar\varphi ,z,p;\delta)+\hat H_{3}^{\theta z}(\bar\varphi,\psi ,J,z,p;\delta)+\hat H^{z}_{4}(\bar\varphi ,z,p;\delta)\\\qquad \qquad \qquad +\hat H_{4}^{\theta z}(\bar\varphi,\psi,J,z,p;\delta)+\delta ^{-1/\alpha}\hat H_{4}^{\theta }(\bar\varphi,\psi,J;\delta ) + \ldots, \end{array}
\end{equation}
where \(\hat H_2^z\) and  $\hat H_2^\theta$ are \(O(\delta ^{\frac{1}{2}-\frac{2}{\alpha}})\)-close to
$\tilde H_2^z$ and, respectively, $\tilde H_2^\theta$; the new homogeneous third degree polynomials
$\hat H^{z}_{3}$ and $\hat H_{3}^{\theta z}$ acquire \(O(\delta ^{\frac{1}{2}-\frac{2}{\alpha}})\)-corrections in comparison with
$\hat H^{z}_{3}$ and $\hat H_{3}^{\theta z}$ of (\ref{eq:normalstep2.5}), and the new homogeneous fourth degree polynomials
$\hat H^{z}_4$, $\hat H_4^{\theta z}$, and $\hat H_4^{\theta}$ acquire \(O(\delta ^{\frac{1}{2}-\frac{3}{\alpha}})\)-corrections. It is only important for us that these corrections vanish at $\delta=0$.

Next, we follow the same steps as in Lemma \ref{lem:hnfdiag}. We make a periodic in \(\bar\varphi\), symplectic transformation of the $(z,p)$-coordinates which brings the Hamiltonian
\(\hat H^z_{2}(z,p;\delta)+\hat H^{z}_{3}(\bar\varphi, z,p;\delta)
+\hat H^{z}_{4}(\bar\varphi,z,p;\delta)\) to its Birkhoff normal form, independent of \(\bar\varphi\)
up to order 4 (recall that no new resonances can be created here by the \(\bar\varphi\)-dependence because the frequency is large). This Hamiltonian is close to the $z$-Hamiltonian of (\ref{eq:hzlem}), therefore the resulting normal form $\hat H_{NF}^z(z,p;\delta)$ is close to the normal form $H_{NF}^z$, see (\ref{eq:hnormalformz}); in particular, $\hat H_{NF}^z=\hat H_2(z,p;\delta) + \hat H_4(z,p;\delta)$ does not contain cubic terms, its quadratic and quartic terms $\hat H_{2,4}^z$ depend only on the actions  $I_z$
(see (\ref{actionsIzpsi})), and their limit as $\delta\to 0$ coincides with the limit of $\tilde H_{2,4}^z$.

After this transformation Hamiltonian (\ref{eq:normalstep2}) becomes
\begin{equation}\label{eq:normalstep3b}\begin{array}{l}
 H (\bar\varphi ,X;\delta)=  \hat H^z_{2}(z,p;\delta) + \hat H^{z}_{4}(z,p;\delta)+ \hat H^\theta_{2}(\psi,J;\delta)+\hat H_{3}^{\theta z}(\bar\varphi,\psi ,J,z,p;\delta)\\ \qquad \qquad \qquad +\hat H_{4}^{\theta z}(\bar\varphi,\psi,J,z,p;\delta)+\delta ^{-1/\alpha}\hat H_{4}^{\theta }(\bar\varphi,\psi,J;\delta) + \ldots, \end{array}
\end{equation}
where the modified functions $\hat H_{3,4}$ have the same structure as before and are close to their counterparts in (\ref{164}).  Since the third order terms in the Hamiltonian are all non-resonant
(by Assumption Box 5 and by the fact that the  period in   \(\bar\varphi\) is small), they are eliminated by a normalizing symplectic transformation of \(X\), which equals to the identity plus higher order    \(\bar\varphi\)-dependent terms and is close to that employed in Lemma
\ref{lem:hnfdiag} (the transformation from (\ref{164}) to (\ref{165})). The Hamiltonian (\ref{eq:normalstep3b}) becomes
\begin{equation}\label{eq:normalstep4}\begin{array}{l}
 H (\bar\varphi ,X;\delta)=  \hat H^z_{2}(z,p;\delta) + \hat H^{z}_{4}(\varphi,z,p;\delta)+ \hat H^\theta_{2}(\psi,J;\delta)\\ \qquad \qquad \qquad +\hat H_{4}^{\theta z}(\bar\varphi,\psi,J,z,p;\delta)+\delta ^{-1/\alpha}\hat H_{4}^{\theta }(\bar\varphi,\psi,J;\delta) + \ldots, \end{array}
\end{equation}
where $\hat H_4$, the fourth-degree polynoimials in $X$, are close to their counterparts in (\ref{165}) (note that applying the transformation  to the singular term, \(\delta ^{-1/\alpha}\hat H_{4}^{\theta }\)   of  (\ref{eq:normalstep3}),  introduces additional singularities but only to the terms of order 5  or higher in (\ref{eq:normalstep4})).

The last step is to bring the 4-th order terms to the autonomous normal form. This is done by a symplectic
transformation and the result is equivalent to throwing away the non-resonant 4-th order terms and taking the average of the resonant ones over $\bar\varphi$. Since all terms $\hat H$ in (\ref{eq:normalstep4}) coincide with their counterparts in (\ref{165}) at $\delta=0$, we immediately obtain that the resulting 4-th order
Birkhoff normal form of  the Hamiltonian (\ref{eq:Hredg5del}) is
\begin{equation}
     H_{NF4}=H^z_{2}(z,p;\delta)+H^{z}_{4}(z,p;\delta)+H^\theta_{2}(\psi,J;\delta)+\delta ^{-1/\alpha}K(\alpha )H_{4}^{\theta }(\psi,J;\delta )+H_{4}^{\theta z}(\psi ,J,z,p;\delta),
\end{equation}
where   \(H^x_{j}(;\delta)\) have bounded coefficients, and, as \(\delta\rightarrow0\) they approach the corresponding terms of (\ref{eq:Hnormalformlem}). Hence, the KAM-nondegeneracy of the elliptic orbit of the  system (\ref{eq:systemofreduced})  at $X=0$ follows from Lemma \ref{lem:nondegcombinednf}.\end{proof}

This completes the proof of Theorem \ref{thm:billiardbox2}.

The  divergence as  \(\delta\rightarrow0\) of the terms or order 5 and higher in \(X\)   does not alter the KAM nondegeneracy result, yet, it implies that quantitative estimates regarding the size of the stability island require analysis of the asymptotic \(\delta\)-dependence of such  terms.

\section{Discussion}

While the lack of ergodicity in Hamiltonian systems is expected, here we found a specific mechanism
for breaking the ergodicity, which persists for arbitrarily high energy for any finite number of particles that interact by repelling forces. We constructed coherent states of the multi-particle gas that correspond to collision-free choreographic solutions which are stabilized at high energies. We have also built similar solutions for systems of weakly interacting particles and systems of attracting particles.

Let us list several future research directions.

\textbf{Non-ergodicity of the gas of repelling particles in containers of dispersive geometry.}
We have established the existence of KAM-stable choreographic solutions for generic multi-particle systems in containers that support a stable periodic billiard motion of one particle. Conjecturally, this includes any generic container with a sufficiently smooth convex boundary. Yet, there are open classes of billiards with piece-wise smooth boundaries which do not allow for stable single-particle motions --
the main example is given by dispersive billiards whose boundary is built of strictly concave smooth pieces. We propose that  it should be possible to apply our method for finding KAM-stable choreographies for such containers  as well. Indeed,
smoothing the billiard potential destroys the hyperbolic structure of the dispersive
billiard, and  there are several known mechanisms for creating elliptic, KAM non-degenerate periodic motions of a single particle in the billiard-like Hamiltonian (\ref{eq:singlepartbiliarddel}) at arbitrary small \(\delta \) for the case of dispersive \(d(\geqslant2)\)-dimensional containers   \cite{Donnay1996,donnay1999non,turaev1998elliptic,TRK03cor,RapRK063d,RapRKT08stab}. It is conjectured that such islands appear for dispersive billiard-like Hamiltonians generically  \cite{turaev1998elliptic}. The phase space volume  of the islands in these cases vanishes  with \(\delta\), yet, its scaling with \(\delta\) is known, and depends on the asymptotic behavior of the container smoothing potential, \(V\), near the boundary. Our techniques imply that such stability islands of the single particle motion can produce also choreographic solutions of the \(N\) particle system provided the perturbations induced by the interaction potential, \(W\), are much smaller than the islands size. Hence, the results regarding the existence of KAM-stable choreographic solutions in containers that support stable motion can probably be extended to any container.

\textbf{Physical relevance of the coherent states.} Under which conditions can the constructed choreographic solutions be observed in  realistic multi-particle systems? General estimates on the probability for an initial condition to belong to a KAM-torus of a Hamiltonian system with $N$-degrees of freedom are quite pessimistic even for small $N$ \cite{chierchia2010properly,chierchia2011planetary}, but examples of such systems where stability islands are well-noticeable are also known \cite{antonopoulos2006chaotic,simo2015dynamical}, e.g. in numerical experiments of    \cite{RapRKT08stab}  the islands are seen for $N$ as large as 20. Therefore,
Nekhoreshev-type estimates on the life time of the coherent states are, probably, most relevant for the
physical realizability question.  Sufficiently long living (i.e., effectively stable) coherent states can be of direct interest when $N$ is not very large. For example, for a gas in a three-dimensional rectangular box, the construction of
\(N < (l_{1}l_2)/\rho^{2}\) particles moving vertically in synchrony, as in Theorem \ref{thm:billiardbox2}, corresponds to pulsating fronts. For a gas in a convex container, the states of
\(N< |L^{*}|/\rho\) particles  constructed in Theorem \ref{thm:billiardsystem} corresponds to rings of current of a specific non-trivial spatial form. For each such state, the effective stability imposes limitations on $N$ and the parameters of the system.  Additional classes of choreographic solutions   may allow to study  coherent states with a larger number of particles. For example, the theory might be extended to create dynamical analogs of Coulomb crystals with a larger number of particles     \cite{amiranashvili1999stability}.

\textbf{Multi-path choreographic solutions} provide such a class. When the single-particle system is non-integrable and has an elliptic periodic orbit, it has, typically, infinitely many elliptic orbits (e.g. around an elliptic periodic orbit there are typically many resonant elliptic orbits, around which there are secondary resonances, and so on \cite{Arnold2007CelestialMechanics}). Similarly, near a homoclinic loop to a saddle center of a non-integrable system many stable periodic orbits co-exist \cite{lerman2021saddle,lerman1991hamiltonian}. When the corresponding periodic paths in the configuration space do not intersect, one can obtain KAM-stable motions of repelling particles along several such paths: the particles on the same path must have the same frequency to avoid collisions, but the particles on different paths may have different frequencies. When the paths do intersect, one needs the frequencies of the motion along the different paths to be in resonance. This condition is not actually restrictive: given any finite set of elliptic orbits with arbitrary periods, one can tune the partial energies such that for every two paths the ratio of the periods would become rational. For particles in the billiard this is done just by normalizing the motion speed for each path by the path length.

\textbf{Choreographies in the box.} The billiard in the box is integrable, and therefore we build choreographic multi-particle regimes based on families of parabolic periodic solutions (instead of elliptic orbits). We have considered only one of such families in this paper, of parallel vertical motion, but there are many types of them, for example, diamond shaped orbits. In our construction of Section \ref{sec:boxsection}, to avoid collisions, a single particle sits on each periodic path. However, for other types of parabolic families, such as diamond shape orbits in a \(d\geqslant3\) box,  many particles may occupy the same path and parallel paths without any collisions. In particular, it may be possible to create KAM-stable choreographic motions of
\(N\propto (l/\rho)^{3}\) particles in a three-dimensional box of characteristic size $l$. Similar solutions can be built for particles in ellipsoidal billiard, and for other systems of weakly-interacting particles whose individual dynamics is integrable.

\textbf{Dynamics of the  averaged system.} We have shown that the stable choreographic motions are controlled by  effective potentials defined on the torus corresponding to the set of phases of the individual particles. Our main result only refers to the fact that near the minimum of such potential the dynamics are, generically, KAM stable. However, one may also ask a question of the global dynamics: can the averaged system be completely integrable, or can one find additional stable motions far from the minimum of the potential? Various types of KAM-stable solutions of
the averaged system should generate new types of non-trivial coherent states which may depend differently on
physical parameters and be relevant for a larger variety of physical settings than the dynamically simplest types of choreographies we found here.  Also, as explained in Appendix \ref{sec:equidistantsol}, in the particular case of the equidistant particles' phases, our effective potentials share the same symmetries as the potentials of the classical Fermi-Pasta-Ulam chains, yet they form a larger class, and it may be interesting to study this broader class of systems.

\textbf{Solid coherent states in a high temperature gas.} In this paper, we operate in the limit where the motion along the periodic orbit is faster than the oscillations of the phase differences between the particles. For highly energetic particles in a container, if the non-averaged interaction potential has a minimum (like the Lennard-Jones potential), one can think of an opposite limit where the frequencies \(\omega\) of the small oscillations of the particles near this minimum are much faster than the frequency of the periodic billiard-like
motion of the center of mass: \(\omega\gg\frac{\sqrt{2h}}{L^*}\) (the kinetic energy of the center-of-mass motion can still be much larger than the energy of the fast but small oscillations of the particles: \(I\ll h/\omega \)). It may be interesting to study whether this can also lead to stable choreographic motions of molecules.

\textbf{Large $N$ limit.} As one can see, different types of coherent states can correspond to different scaling of the number of particles as a function of the size of the container. In general, when considering the limit $N\to \infty$, one should also decide how the parameters of the system (the system size $l$, the energy per particle $h$, effective particle diameter $\rho$, etc.) scale with $N$: different types of  scalings correspond to different physical situations.
Determining which type of scalings correspond to various stable coherent states is the key for resolving the question of the realizability of such states in physically relevant settings.
\section*{Acknowledgements} VRK work was funded by Israel Science Foundation (grant 787/22).
DT work was funded by the Leverhulme Trust (grant RPG-2021-072).
We thank V. Gelfreich and L. Chierchia for useful discussions.

\section*{Data declaration} Data sharing not applicable to this article as no datasets were generated or analysed during the current study.
\appendix
\section*{Appendix}
\section{ Choreographic solutions with equidistant phases.}
\label{sec:equidistantsol}

The evolution of the phases of the choreographic solutions is described, in the first-order approximation, by the averaged Hamiltonian
\begin{equation}\label{hfpu}
H =\frac{1}{2}(I_{0})^2+ U(\theta),\quad (\theta, I_{0})\in  \mathbb{T}^{N}\times \mathbb{R}^{N},
\end{equation}
see Lemma \ref{Lemma1.1} (here we scale $I_0$ such that $a=1$ in the matrix $A$ of (\ref{eq:npartav}), see (\ref{amatr})).
The averaged potential $U$ is given by
$$U(\theta^{(1)}, \dots,\theta^{(N)})=\mathop{\sum_{n,m=1,\dots,N}}_{n\neq m} W_{avg}(\theta^{(n)}-\theta^{(m)}),$$
where $W_{avg}$ is an even, \(2\pi\)-periodic function defined by (\ref{eq:averagedw}).

System (\ref{hfpu}) is a generalization of the classical Fermi-Pasta-Ulam chain: the difference is that in the FPU we have only
$m=n+1 \mod N$ in the sum describing the potential $U$, i.e., the interaction between the phases $\theta$ is short-range, while
in our case all phases typically interact with each other. Like in the FPU, any uniformly distributed particle configuration
\(\theta=\theta_{eq}\),  where
$$\theta_{eq}^{(2)}-\theta_{eq}^{(1)}=\ldots =\theta_{eq}^{(N)}-\theta_{eq}^{(N-1)}=
\theta_{eq}^{(1)}+2\pi - \theta_{eq}^{(N)} = \frac{2\pi }{N}$$
is an equilibrium of system (\ref{hfpu}). Indeed, it is always an extremum of $U$ since
$\frac{\partial U}{\partial\theta^{(n)}}|_{\theta_{eq}}=0$ for each $n$:
$$\frac{\partial U}{\partial\theta^{(n)}}|_{\theta_{eq}}=
\sum_{m\neq n} W_{avg}^\prime (\theta_{eq}^{(n)}-\theta_{eq}^{(m)})=\sum^{N-1}_{k=1} W_{avg}^\prime (\frac{2\pi }{N}k)=0,$$
where the last equality follows because \(W_{avg}^\prime\) is an odd and $2\pi$-periodic function.

Such configurations form a line of extrema of $U$ (parameterized by the choice of $\theta_{eq}^{(1)}$). It is a line of {\em minima}
of $U$ when the Hessian matrix at $\theta=\theta_{eq}$ is positive semi-definite. We have
\begin{equation}\label{hessmatr}
\frac{\partial^{2} U }{\partial\theta^{(n)}\partial\theta^{(m)}}(\theta_{eq}) =
\begin{cases} - W_{avg}''(\frac{2\pi(n-m) }{N})=u_{n-m} & \mbox{ for } \; n\neq m \\
\sum^{N-1}_{j= 1 }W_{avg}''(\frac{2\pi }{N}j)=u_{0}=-\sum^{N-1}_{k=1}u_k & \mbox{ for } \;n=m. \\
\end{cases}\end{equation}
Note that the numbers $u_k$ satisfy $u_k=u_{-k}=u_{N-k}$ because $W_{avg}''$ is even and $2\pi$-periodic.

The Hessian matrix (\ref{hessmatr}) is a circulant matrix, so its eigenvectors are the Fourier modes
$$v_{j}=(1,\exp(i\frac{2\pi j }{N}),\exp(i\frac{4\pi j }{N}),\ldots,\exp(i\frac{2(N-1)j\pi }{N})), \quad j=0,\ldots,N-1.$$
The corresponding eigenvalues are:
$$
\lambda_{j}=u_0+u_1\exp(i\frac{2\pi j }{N})+u_2\exp(i\frac{4\pi j }{N})+\ldots +u_{N-2}\exp(i\frac{2\pi j(N-2) }{N})+u_{N-1}\exp(i\frac{2\pi j(N-1)}{N}),
$$
i.e.,
\begin{equation}\label{lamjn}
\lambda_{j}= 2\sum_{1\leqslant k < N/2} u_k (\cos (k \frac{2\pi j }{N})-1) + ((-1)^j-1) u_{N/2}
\end{equation}
(we take $u_{N/2}=0$ in this formula when $N$ is odd).

As we see, $\lambda_0=0$, which is due to the translational symmetry of the Hamiltonian (\ref{hfpu}). When all other $\lambda_j$ are strictly positive, the line of uniformly distributed particles' configurations consists of minima of the potential $U$. This happens, for example, when $u_k<0$ for all $k\neq 0$, i.e., when $W_{avg}$ is a convex function, which is consistent with the repelling nature of the interaction.

The frequencies of small oscillations around the line of minima are equal to $\sqrt{\lambda_j}$. It follows from
(\ref{lamjn}) that
$$\lambda_j=\lambda_{N-j}$$
for all $1\leqslant j \leqslant N/2$. Hence, the standard non-resonance assumption on the interaction potential breaks at
$\theta=\theta_{eq}$. In particular, adding a small perturbation to (\ref{hfpu}) (without breaking the translational symmetry) would result, in general, in the destruction of the ellipticity.

Yet, we show next that the KAM Assumption IP1 still holds generically. Notice that the resonance relations are due to the discrete symmetries of the potential near $\theta=\theta_{eq}$. Namely, the system in a small neighborhood of the minima line is symmetric  with respect to the
transformations \(T: \theta^{(n)}  \rightarrow \theta^{(n+1)} \) (where \(n\) is taken \(\mod N\)) and
\(S: \theta^{(n)}  \rightarrow \theta^{(N+1-n)} \). The maps $S$ and $T$ generate the so-called \(N\)-th dihedral group. This is the same group of symmetries as in the FPU. The normal form theory for the FPU was built by Rink \cite{rink2001symmetry}. In fact,
he derived  the normal form near an equilibrium of a general \(S,T\)-symmetric Hamiltonian provided the Hamiltonian has no additional resonances. His work applies to our Hamiltonian (even though the interactions here are for all particle pairs whereas in \cite{rink2001symmetry}  the FPU chain with only nearest neighbors interactions was considered). Indeed, notice that for any prescribed sequence \(u_k,k=1,\dots,N-1\) there is a smooth even and $2\pi$-periodic potential \(W_{avg}\) satisfying \(W_{avg}''(\frac{2\pi }{N}k)=u_k\), so, generically, for our system, no additional independent resonance relations appear.

The symmetric Rink normal form is given by Theorem 8.2 in \cite{rink2001symmetry} and for {\em odd} \(N\) is written as follows:
\begin{equation}\label{oddrink}
H_{Rink}=\sum_{1\leqslant j< \frac{N}{2}}\sqrt{\lambda_{j}}\;a_j+\frac{1}{2}\sum _{1\leqslant j,k< \frac{N}{2}}(C^a_{jk}a_ja_k+C^b_{jk}b_jb_k)
\end{equation}
whereas for {\em even} \(N\)
\begin{equation}\label{evenrink}
H_{Rink}=\sum_{1\leqslant j\leqslant \frac{N}{2}}\sqrt{\lambda_{j}}\;a_j+\frac{1}{2}\sum_{1\leqslant j,k\leqslant \frac{N}{2}}C^a_{jk}a_ja_k+
\frac{1}{2}\sum_{1\leqslant j,k < \frac{N}{2}} C^b_{jk}b_jb_k+
\sum_{1\leqslant j \leqslant \frac{N}{4}} C^d_{j}(d_jd_{\frac{N}{2}-j} - c_jc_{\frac{N}{2}-j}).
\end{equation}
The terms $a,b,c,d$ are quadratic functions, so the normal form is of order $4$ (the $C$'s are constant coefficients such that $C^{*}_{jk}=C^{*}_{kj},*\in\{a,b,c,d\}$). The transition to the normal form is done as follows. First, one makes a linear symplectic coordinate transformation which diagonalizes the quadratic part of the Hamiltonian near $(\theta=\theta_{eq}, I_0=0)$. One defines
$$\begin{array}{l} \displaystyle \!\!\!\!\!\!\!\!\!\!\!\!\!\!\!\!\!\!\!\!\!\!\!\!\!\!\!\!\!\!\!\!\mbox{for }\; 1\leqslant j < N/2:\\
\displaystyle\qquad
z_j = \frac{1}{\sqrt{2N}}\sum_{n=1}^N e^{-2\pi i\frac{j}{N} n} (I_0^{(n)} + i\sqrt{\lambda_j}\; (\theta^{(n)}-\theta_{eq}^{(n)})),\\
\displaystyle  \qquad\zeta_j =
\frac{1}{\sqrt{2N\lambda_j}}\sum_{n=1}^N e^{2\pi i\frac{j}{N} n} (I_0^{(n)} - i\sqrt{\lambda_j}\; (\theta^{(n)}-\theta_{eq}^{(n)})),\\
\displaystyle\qquad
z_{N-j} = -\;\frac{1}{\sqrt{2N}}\sum_{n=1}^N e^{-2\pi i\frac{j}{N} n} (I_0^{(n)} - i\sqrt{\lambda_j}\; (\theta^{(n)}-\theta_{eq}^{(n)})),\\
\displaystyle  \qquad\zeta_{N-j} =
\frac{1}{\sqrt{2N\lambda_j}}\sum_{n=1}^N e^{2\pi i\frac{j}{N} n} (I_0^{(n)} + i\sqrt{\lambda_j}\; (\theta^{(n)}-\theta_{eq}^{(n)})),\\
\displaystyle \!\!\!\!\!\!\!\!\!\!\!\!\!\!\!\!\!\!\!\!\!\!\!\!\!\!\!\!\!\!\!\!\mbox{and, if $N$ is even}:\\
\displaystyle\qquad z_{\frac{N}{2}} =
\frac{1}{\sqrt{2N \lambda_{\frac{N}{2}}}}\sum_{n=1}^N (-1)^n (I_0^{(n)} + i\sqrt{\lambda_j}\; (\theta^{(n)}-\theta_{eq}^{(n)})),\\
\displaystyle  \qquad\zeta_{\frac{N}{2}} =-\;
\frac{1}{\sqrt{2N}}\sum_{n=1}^N (-1)^n (I_0^{(n)} - i\sqrt{\lambda_j}\; (\theta^{(n)}-\theta_{eq}^{(n)})),
\end{array}
$$
cf. \cite{rink2001symmetry}, formulas (7.2),(7.3). By the translational invariance of the average potential $U$, the Hamiltonian (\ref{hfpu})
in these coordinates is the  sum of the term \(\frac{1}{2N}\sum_{n=1}^N (I^{(n)}_0)^2\)
and a function which depends only on $(z_j, \zeta_j)$ with $j=1,\ldots, N-1$.
This function is the Hamiltonian of the system reduced by the translation symmetry group ($\theta^{(n)}\mapsto \theta^{(n)} + c,\; c\in \mathbb{R}^1, n=1,\ldots, N$). Since $(z,\zeta)=0$ is an equilibrium of the reduced system, the Taylor expansion of the reduced Hamiltonian
at zero starts with quadratic terms -- these terms coincide with the quadratic part of $H_{Rink}$ in (\ref{oddrink}) and (\ref{evenrink}), i.e.,
with $\sum_{1\leqslant j\leqslant \frac{N}{2}}\sqrt{\lambda_{j}}\;a_j$. Next, one does a symplectic transformation (identity plus terms of the second order and higher) of the variables $(z_j, \zeta_j)$, $j=1,\ldots, N-1$, which brings the reduced Hamiltonian to the form
which coincides with the normal forms (\ref{oddrink}) or (\ref{evenrink}) up to terms of order $5$. It is a standard fact that
this system has KAM-tori if the normal form has KAM-tori in an arbitrarily small neighborhood of zero, so we further focus on finding KAM-tori in
the normal forms.

In the new variables $(z,\zeta)$ the functions $a,b,c,d$ in $H_{Rink}$ are given by
$$a_j=i(z_j\zeta_j - z_{N-j}\zeta_{N-j}), \qquad b_j=i(z_j\zeta_j + z_{N-j}\zeta_{N-j}),$$
$$c_j=\frac{1}{\sqrt{\lambda_j}}(z_j z_{N-j} +\lambda_j \zeta_ j\zeta_{N-j}), \qquad
d_j=\frac{-i}{\sqrt{\lambda_j}}(z_j z_{N-j} -\lambda_j \zeta_ j\zeta_{N-j}),$$
for $1\leqslant j <\frac{N}{2}$; we also have, for even $N$,
$$ a_{\frac{N}{2}}=i z_{\frac{N}{2}} \zeta_{\frac{N}{2}}.$$
Note that $a,b,c,d$  are real and satisfy
\begin{equation}\label{abcdjs}
a_j^2=b_j^2+c_j^2+d_j^2,
\end{equation}
see \cite{rink2001symmetry}, formulas (8.6),(8.7). The coefficients $C_{jk}^*$ in $H_{Rink}$ are polynomials of the coefficients of the original Hamiltonian depending rationally on $\sqrt{\lambda_j}$. One can check that a generic potential $U$ corresponds to a generic choice of $C$'s.

For odd \(N\), the normal form (\ref{oddrink}) is completely integrable and generically has KAM tori \cite{rink2001symmetry}. This gives us that the the KAM assumption IP1 is fulfilled generically for system (\ref{hfpu}) for odd $N$. Let us consider the case of even $N$. By \cite{rink2001symmetry}, Corollary 9.3, the $(N-1)$ degrees of freedom system (\ref{evenrink}) has a number of quadratic integrals in involution: $a_j$, $j=1,\ldots,\frac{N}{2}$ and $b_j-b_{\frac{N}{2}-j}$, $1\leqslant j <\frac{N}{4}$. This set is incomplete; while the normal form $H_{Rink}$ for the nearest-neighbor FPU chain with an even \(N\) has additional integrals and is completely integrable \cite{Henrici2008,Henrici2009}, it is not known whether this normal form is completely integrable for a generic choice of coefficients $C$ and
$\lambda$ in (\ref{evenrink}), or for a general choice of the potential $U$ in (\ref{hfpu}).

However, the restriction of (\ref{evenrink})
to the invariant subspace $\{a_s=0, 1 \leqslant s \leqslant \frac{N}{4}\}$
is completely integrable: since $b_s$, $c_s$ and $d_s$ all vanish for
$1 \leqslant s \leqslant \frac{N}{4}$ by (\ref{abcdjs}), the restricted Hamiltonian is given by
$$H=\sum_{\frac{N}{4} < j \leqslant \frac{N}{2}}\sqrt{\lambda_{j}}\;a_j+\frac{1}{2}\sum_{\frac{N}{4}< j,k \leqslant \frac{N}{2}}C^a_{jk}a_ja_k+
\frac{1}{2}\sum_{\frac{N}{4}< j,k < \frac{N}{2}} C^b_{jk}b_{j}b_k,$$
and the quadratic functions $a_j$ ($\frac{N}{4}< j\leqslant  \frac{N}{2}$) and $b_j$ ($\frac{N}{4}< j < \frac{N}{2}$) give a complete set of its integrals. The restricted Hamiltonian has the same structure as the completely integrable normal form (\ref{oddrink}). So, in the same way as
it was done in \cite{rink2001symmetry} for system (\ref{oddrink}), one establishes that for a generic choice of non-zero values of the integrals
$a_j$ ($\frac{N}{4}< j\leqslant  \frac{N}{2}$) and $b_j$ ($\frac{N}{4}< j < \frac{N}{2}$) the corresponding joint level set of these integrals is a
KAM-torus $\mathcal{T}$ of the restricted system. Moreover, these integrals $a_j$ and $b_j$ are the action variables (see formula (9.3) in \cite{rink2001symmetry}). The dynamics in a small neighborhood of such torus are described, in the main approximation, by the Hamiltonian
(\ref{evenrink}) {\em averaged} over the angle variables conjugate to the actions. The only terms in (\ref{evenrink}) that depend on these
angle variables are $d_{\frac{N}{2}-j}$ and $c_{\frac{N}{2}-j}$ with $j<\frac{N}{4}$; let us show that their averaged values are zero.

Indeed, by the ergodicity of the flow on the invariant torus $\mathcal{T}$, we can replace the averaging over the angle variables by
the time averaging. Due to the commutation relations (see formula (9.1) in \cite{rink2001symmetry})
$$\{b_k,c_m\}=2d_m\delta_{km}, \qquad \{b_k,d_m\}= - 2d_m, \qquad \{a_k,c_m\}=\{a_k,d_m\}=0,$$
where $\{\cdot,\cdot\}$ is the Poisson bracket and $\delta_{km}$ is the Kronecker delta, we have that for $\frac{N}{4}<m<\frac{N}{2}$
$$\frac{d}{dt} c_m =\{c_m,H_{Rink}\}= - \Omega_m\;  d_m,\qquad \frac{d}{dt} d_m =\{d_m,H_{Rink}\}= \Omega_m\;  c_m,$$
where
$$\Omega_m = 2 \sum_{\frac{N}{4} < k < \frac{N}{2}} C^b_{km} b_k.$$
For a generic choice of the actions $b_k$ ($\frac{N}{4} < k < \frac{N}{2}$), the frequencies $\Omega_m$ are all non-zero, hence $c_m$ and $d_m$ perform harmonic oscillations and their time-average is zero. Thus, the averaged system (\ref{evenrink}) is
$$H=\sum_{1\leqslant j\leqslant \frac{N}{2}}\sqrt{\lambda_{j}}\;a_j+\frac{1}{2}\sum_{1\leqslant j,k\leqslant \frac{N}{2}}C^a_{jk}a_ja_k+
\frac{1}{2}\sum_{1\leqslant j,k < \frac{N}{2}} C^b_{jk}b_jb_k.$$
This system has the same structure as (\ref{oddrink}) and is completely integrable (the integrals are $a_j$ and $b_j$).
It follows that the invariant torus $\mathcal{T}$ is normally elliptic.
As in \cite{rink2001symmetry}, one checks that a generic Liouville torus of this systems satisfies the twist condition, hence
$\mathcal{T}$ is surrounded by KAM-tori and these tori persist when we proceed from the averaged system to the original system $H_{Rink}$
in a neighborhood of $\mathcal{T}$.

\end{document}